\documentclass[aos]{imsart}

    \RequirePackage{amsthm,amsmath,amsfonts,amssymb}
\numberwithin{equation}{section} %equation numering by section
\RequirePackage[numbers]{natbib}
\RequirePackage[colorlinks,citecolor=blue,urlcolor=blue]{hyperref}
\RequirePackage{graphicx}

\usepackage{dsfont}

\usepackage{mathabx}

\usepackage{graphicx}

\usepackage{hyperref}

\usepackage{xcolor}

\usepackage[autostyle]{csquotes}  

\usepackage{prodint}

\usepackage{enumitem}

\startlocaldefs
\theoremstyle{plain}

\newtheorem{theorem}{Theorem}[section]
\newtheorem{lemma}[theorem]{Lemma}
\theoremstyle{remark}
\newtheorem{definition}{Definition}
\newtheorem{proposition}[theorem]{Proposition}
\newtheorem{corollary}[theorem]{Corollary}

\newtheorem{example}{Example}
\newtheorem{remark}[theorem]{Remark}
\newtheorem{condition}{Condition}

\newcommand{\pr}{{\amsmathbb{P}}}

\newcommand\independent{\protect\mathpalette{\protect\independenT}{\perp}}
\def\independenT#1#2{\mathrel{\rlap{$#1#2$}\mkern2mu{#1#2}}}

\newcommand*\dd{\mathop{}\!\mathrm{d}}

\usepackage{mathbbol}
\DeclareSymbolFontAlphabet{\amsmathbb}{AMSb}%

\usepackage[outline]{contour}

\endlocaldefs
\begin{document}

\begin{frontmatter}
\title{Censored extreme value estimation}
%\title{A sample article title with some additional note\thanksref{t1}}
\runtitle{Censored extreme value estimation}
%\thankstext{T1}{A sample additional note to the title.}

\begin{aug}
%%%%%%%%%%%%%%%%%%%%%%%%%%%%%%%%%%%%%%%%%%%%%%%
%% Only one address is permitted per author. %%
%% Only division, organization and e-mail is %%
%% included in the address.                  %%
%% Additional information can be included in %%
\author[A]{\fnms{Martin}~\snm{Bladt}\ead[label=e1]{martinbladt@math.ku.dk}}
\and
\author[B]{\fnms{Igor}~\snm{Rodionov}\ead[label=e2]{igor.rodionov@essex.ac.uk}}
%%%%%%%%%%%%%%%%%%%%%%%%%%%%%%%%%%%%%%%%%%%%%%
%% Addresses                                %%
%%%%%%%%%%%%%%%%%%%%%%%%%%%%%%%%%%%%%%%%%%%%%%
\address[A]{Department of Mathematical Sciences, University of Copenhagen \\
Universitetsparken 5, DK-2100 Copenhagen, Denmark \\[2mm] \printead[presep={}]{e1}}
\address[B]{School of Mathematics, Statistics and Actuarial Science, University of Essex \\ Wivenhoe Park, Colchester CO4 3SQ, United Kingdom \\[2mm]
Institute of Mathematics, \'Ecole Polytechnique F\'ed\'erale de Lausanne \\
Route Cantonale, CH-1015 Lausanne, Switzerland \\[2mm] \printead[presep={}]{e2}}

\end{aug}

\begin{abstract}
A novel and comprehensive methodology designed to tackle the challenges posed by extreme values in the context of random censorship is introduced. The main focus is on the analysis of integrals based on the product-limit estimator of normalized upper order statistics, called extreme Kaplan--Meier integrals. These integrals allow for the transparent derivation of various important asymptotic distributional properties, offering an alternative approach to conventional plug-in estimation methods. Notably, this methodology demonstrates robustness and wide applicability among various tail regimes. A noteworthy by-product is the extension of generalized Hill-type estimators of extremes to encompass arbitrary tail behavior, which is of independent interest. The theoretical framework is applied to construct novel estimators for real-valued extreme value indices for right-censored data. Simulation studies confirm the asymptotic results and, in a competitor case, mostly show superiority in mean square error. An application to brain cancer data demonstrates that censoring effects are properly accounted for, even when focusing solely on tail classification.
\end{abstract}



\end{frontmatter}
%%%%%%%%%%%%%%%%%%%%%%%%%%%%%%%%%%%%%%%%%%%%%%
%% Please use \tableofcontents for articles %%
%% with 50 pages and more                   %%
%%%%%%%%%%%%%%%%%%%%%%%%%%%%%%%%%%%%%%%%%%%%%%
%\tableofcontents

\section{Introduction}
Analyzing extreme events that experience right censoring is a crucial undertaking in the field of statistics, and failing to account for these censoring effects can lead to significant bias. Roughly speaking, by treating extreme values that remain concealed due to censoring as fully-observed data points, we may underestimate their true effects. Such circumstances commonly arise when investigating weather phenomena, financial markets or re-insurance claim sizes. Likewise, in biomedical applications, the proper classification of the behavior of the largest observations is a delicate task due to incomplete data.
Hence, exploring censored extremes is a crucial step in mitigating associated risks. Given the escalating worldwide prevalence of catastrophic occurrences, the study of censored extremes plays an important role in constructing enhanced statistical models for these risks. In turn, these correctly-calibrated models facilitate informed decision making.

The foundational statistical work in fully-observed extreme value theory is by now classic and well-understood in the univariate case; see~\cite{beirlant2006statistics, embrechts2013modelling, dehaan}, and to a large degree for univariate time series, as can be found in~\cite{kulik2020heavy}. 
Similarly, the pillars of survival analysis have been laid out since the seventies, in different flavors, with the counting process point of view being one of the the most popular approaches, and with martingale techniques playing an ubiquitous role; see  \cite{fleming2013counting,AndersenBorganGillKeiding1993}. However, a unified methodology for estimating various extreme values under censoring is still lacking. Our paper aims to address this gap.

More precisely, this paper proposes the analysis of randomly right-censored extremes through Kaplan--Meier integrals of normalized upper order statistics (see~\cite{stutewang} for the analogous non-extreme case). This generalizes integral functions with respect to the empirical distribution function of normalized upper order statistics, which have received interest in the literature. For instance, such integrals can often be cast as sums of extreme values from iid samples, which were studied in~\cite{csorgo1986rv} for regularly varying tails and in~\cite{csorgo1991asymptotic}  for arbitrary tail behavior. Upon replacement of the empirical distribution function with the Kaplan--Meier estimator of upper order statistics of a censored sample, the analysis of integral functions is no longer transparent, neither with theory of empirical processes nor using the sums approach. This paper is devoted to studying precisely these integral functions by using the elegant U-statistics approach of \cite{stute}, together with machinery from extreme value theory.

In order to generalize the approach of \cite{stute} to the extreme case, the theory of U-statistics \cite{serfling} and of residual estimation (based on averages of functions of normalized upper order statistics) of \cite{segers} are important. However, they are not directly applicable, the former having to be adapted to accommodate upper order statistics (see~\cite{oorschot2023tail} for a related concept), and the latter requiring extension to arbitrary tail regimes. These adaptations are provided in this paper, and yield a unifying and transparent treatment of extreme estimators and their convergence, imposing natural conditions. Estimators of the parameter governing the tail behavior (the so-called extreme value index) derived from extreme Kaplan--Meier integrals have similar form and in some cases equal asymptotic variance as the ones treated in~\cite{bww}, though the construction principle is rather different and our theory allows for the direct construction of estimators of further extreme value characteristics.

The central result of the paper is Theorem~\ref{T4}, providing a general decomposition of extreme Kaplan--Meier integrals in terms of sums of conditionally i.i.d. random variables. The latter result and its extensions, in combination with other conditions, yield useful limit theorems and applications. The remaining { theoretical}  contributions of
the paper can be summarized as follows:

\begin{enumerate}
\item Suitable conditions are established under which extreme Kaplan--Meier integrals are weakly consistent, Theorem~\ref{Tconsistency}, and asymptotically normal %within the Fr\'echet max-domain of attraction, that is 
for distributions with regularly varying tails. Asymptotic normality is provided in two steps, with only the last one, comprised in Theorem~\ref{T3an} and Corollary~\ref{R2}, requiring a second-order condition. The main unavoidable restriction is a second moment condition for the limit variance to exist.  {As an illustration, the asymptotic normality of a Hill-type estimator initially proposed by \cite{worms2014new} for the positive extreme value index is established for the first time in Corollary~\ref{CorHill}.} Another notable application is Corollary \ref{cor:ekme} where the conditions of consistency and asymptotic normality of the extreme Kaplan--Meier estimator are provided.

\item Residual estimation theory is generalized to arbitrary tail regimes and in particular it is found that the convergence rates vary across regimes. The results are presented for fully-observed data and are of independent interest. Here, the main result is Theorem~\ref{T2gen}. The exposition is in terms of so-called candidate functions, which themselves must satisfy a second-order type condition.  {This extension is key, since residual estimators are the uncensored version of extreme Kaplan--Meier integrals.}
\item The results of Section~\ref{sec:ekmi} are extended to all max-domains of attraction in terms of candidate functions, {  using generalized residual estimators, now allowing for censored data}. Here, certain regularly varying functions with tail index that can be identified as the asymptotic non-censoring proportion play a central role; see Condition~\ref{C6}. We highlight Theorem~\ref{T1genekmi} and Theorem~\ref{T2genekmi} as our main extensions.
\item Finally, we provide a larger application, by establishing full asymptotics for a novel censored version of the moment estimator of \cite{moment}, using all the main components of our theory. We provide in Theorem~\ref{Tmomentconsistency} consistency and in Theorem~\ref{Tmomentan} asymptotic normality.
%Some assumptions from the general theory are relaxable for this application and thus some more technical proofs are provided for precisely this purpose.
\end{enumerate}

The cross-fertilization of extreme value theory and survival analysis has been explored before in various contexts. We provide an account of some of the most important contributions. In~\cite{beirlant2001pareto} the Pareto index estimation was considered for the first time under random right-censorship framework (see also \cite{beirlant2007estimation,einmahl2008statistics} covering arbitrary tail regimes), which has since produced some generalizations, e.g. using covariates in~\cite{ndao2014nonparametric}, truncation in~\cite{gardes2015estimating}, trimming in~\cite{bladt_trim}, and bias-reduction in~\cite{BEIRLANT2018114,goegebeur2019bias}. Another approach, which neatly formulates the censoring problem with extremes as a path-wise problem, was proposed in~\cite{worms2014new} through constructions based on expected values{, while \cite{worms2018extreme} developed a competing risks model using techniques from \cite{stutewang}}. Subsequently, a partial result on the asymptotic distribution was recently provided in~\cite{bww}. %However, their method of proof is not generalizable to our setting. 
In terms of missing extreme data, the main reference is \cite{xu2022handling}, and finally the machinery of extreme value theory has recently been imported to survival analysis in a different context, namely for cure-rate modelling in the case of insufficient follow-up; see~\cite{escobar2019non,escobar2022estimation,escobar2023nonparametric}.

{The focus of the existing literature is on providing methods for estimating the extreme value index in censored samples, or estimators derived thereof. In contrast, our results facilitate further directions (apart from the presented Hill-type and moment estimators) such as: the estimation of the scale and second-order parameters; the construction of estimates for distribution tails; estimation of extreme quantiles and other extreme characteristics for censored samples. Moreover, popular methods in extreme value theory such as maximum likelihood estimation and probability weighted moments can be extended to censored extremes using our methodology.\\}

\textit{Organisation of the paper}. In Section~\ref{sec:prelims} we introduce notation and results from extreme value theory and survival analysis relevant to our setup. Section~\ref{sec:ekmi} is devoted to the study of extreme Kaplan--Meier integrals when both event and censoring variables have distributions with regularly varying tails.
{ Residual estimators, the uncensored versions of extreme Kaplan--Meier integrals, are extended in Section~\ref{genresidual} to arbitrary tail regimes, and their applications are found in the Supplementary Material. Extensions of the results of Section~\ref{sec:ekmi} are finally provided by generalizing Section~\ref{genresidual} for censored data; this is done in Section~\ref{sec:exten}. Finally, Section~\ref{sec:evi_est} provides a prime application of our theoretical results, proposing and analysing in detail an extreme value index estimator for censored data and for all max-domains of attraction. Three simulation studies are provided in Section~\ref{sec:sim}, and an application to a brain cancer dataset is given in Section~\ref{sec:real_data}. All proofs may be found in the Supplementary Material.} 

\section{Preliminaries and setting}\label{sec:prelims}
This section provides standard preliminaries from survival analysis and extreme value theory, and may be used as a reference for notation in the remainder of the paper.

\subsection{Random censoring}
%Let $X_1, \ldots, X_n$ be i.i.d. random variables with common cdf $F.$ Let $\varphi: \amsmathbb{R} \to\amsmathbb{R}$ be some measurable function which is $F$-square integrable, that is $\int \varphi^2 \dd F<\infty$. Then, an application of the central limit theorem (CLT) for $\{\varphi(X_i)\}_{i=1}^n$ can be written as 
%\begin{equation}\label{cltfirst}\sqrt{n} \int \varphi\, \dd (F_n - F) \stackrel{d}{\to} N(0, \sigma_\varphi^2),\end{equation} where $F_n$ is the empirical cdf of $\{X_i\}_{i=1}^n$ and 
%$\sigma_\varphi^2 = \int \varphi^2 \dd F - \left(\int \varphi \,\dd F\right)^2.$ 

We assume the random right-censoring setting, that is, we observe the following datapoints
\[Z_i = \min(X_i, Y_i) \:\:\:\text{ and }\:\:\: \delta_i = I(X_i\le Y_i),\] where $\{X_i\}_{i=1}^n$ and $\{Y_i\}_{i=1}^n$ are independent i.i.d. sequences cdf $F$ and $G$, respectively. Notice that $\delta_i$ indicates whether the variable of interest $X_i$ has been observed or not. Let $H$ denote the cdf of $Z_1$, which by independence of the variable $X_1$ and the censoring mechanism $Y_1,$ is given by
\begin{equation}1 - H = (1-F)(1-G).\label{hfg}\end{equation}

A consistent estimator of $F$, \cite{wang1987note}, is now provided by the Kaplan--Meier product-limit estimator $\amsmathbb{F}_n$ defined by
\begin{equation}\label{kaplanmeyer}1 - \amsmathbb{F}_n(x) = \prod_{i=1}^n \left[1 - \frac{\delta_{[i:n]}}{n-i+1}\right]^{I(Z_{i,n}\le x)}.\end{equation}
Here $Z_{1,n} \le \ldots \le Z_{n,n}$ are the order statistics of $\{Z_i\}$ and $\{\delta_{[i:n]}\}_{i=1}^n$ are the concomitants of $\{Z_{i,n}\}_{i=1}^n,$ that is $\delta_{[i:n]} = \delta_j$ if $Z_{i,n} = Z_j.$

Let $\varphi: \amsmathbb{R} \to\amsmathbb{R}$ be some measurable function. Write $F\{a\} = F(a) - F(a-)$, such that $F\{a\}>0$ iff $a$ is an atom of $F$.  Let $A$ be the set of all atoms of $H,$ possibly empty. Under $\int|\varphi| \dd F <\infty,$ Stute and Wang \cite{stutewang} showed that 
\[\int \varphi\, \dd \amsmathbb{F}_n \stackrel{a.s.}{\to} \int \varphi\, \dd\widetilde F, \quad n\to\infty,\] where $\tau_H = \inf\{x: H(x)=1\}$ and
\begin{equation}\label{tildef}\widetilde F(x) = \left\{\begin{array}{cc} F(x), & \text{ if } x<\tau_H; \\ F(\tau_H -) + I(\tau_H\in A) F\{\tau_H\}, & \text{ otherwise.}\end{array}\right.\end{equation}
Theorem 1.1 in~\cite{stute} then states that under some mild regularity conditions imposed on $\varphi,$ $F$ and $G$, the Kaplan--Meier integral $\int \varphi\,\dd \amsmathbb{F}_n$ can be represented as a sum of i.i.d. random variables and a summand $R_n$ vanishing as $o_{P}(n^{-1/2})$. 

Introduce the subdistribution functions, for $z\in \amsmathbb{R}$,
\begin{align*} 
 H^j(z) &= \pr(Z\le z, \delta = j),\quad j=0,1,
\end{align*} where $(Z, \delta) \stackrel{d}{=} (Z_1, \delta_1).$ The following three functions are necessary to state a key asymptotic decomposition of $\int \varphi\, \dd \amsmathbb{F}_n$, and are provided here for future reference,
\begin{equation}\label{gamma0} \gamma_0(x) = \exp\left\{\int_{z<x}\frac{H^0(\dd z)}{1 - H(z)}\right\},\end{equation}
\begin{equation}\label{gamma1}\gamma_1(x) = \frac{1}{1 - H(x)}\int_{x< z}\varphi(z) \gamma_0(z) H^1(\dd z),\end{equation} and
\begin{equation}\label{gamma2} \gamma_2(x) = \int\int \frac{I(v<x, v<z) \varphi(z) \gamma_0(z)}{(1 - H(v))^2}  H^0(\dd v) H^1(\dd z).\end{equation}
These functions can be simplified if $F$ and $G$ are continuous, for instance
\begin{equation}\gamma_0(x) = \frac{1}{1-G(x)}, \quad x<\tau_H.\label{gamma_ex}\end{equation} 

The decomposition in turn implies the central limit theorem under random censorship (see Corollary 1.2, \cite{stute}),
\begin{equation}\label{cor1.2stute}\sqrt{n} \int \varphi\, \dd(\amsmathbb{F}_n - \widetilde F) \stackrel{d}{\to} N(0, \sigma_\varphi^2),\end{equation}
where now $\sigma_\varphi^2 = {\rm var}\{\varphi(Z)\gamma_0(Z)\delta + \gamma_1(Z)(1-\delta) - \gamma_2(Z)\}.$

\subsection{Extreme value theory} 

When analyzing extremes, a popular approach is to consider a categorization in terms of the possible non-degenerate limits of the excess distribution. It turns out that there are only three such limits, which can be written in terms of the Generalised Pareto Distribution (GPD), with cumulative distribution function given by
\begin{align}\label{eq:GPD_def}
G_\gamma(y) = 
\begin{cases} 
1 - \left(1 + \gamma \frac{y - \mu}{\sigma}\right)^{-1/\gamma}, & \text{if } \gamma \neq 0, \\
1 - \exp\left(-\frac{y - \mu}{\sigma}\right), & \text{if } \gamma = 0,
\end{cases}
\end{align}
for $y \geq \mu$ when $\gamma \geq 0$, and $\mu \leq y \leq \mu - \sigma/\gamma$ when $\gamma < 0$, where $\mu \in \amsmathbb{R}$ is the location parameter, $\sigma > 0$ is the scale parameter, and $\gamma \in \amsmathbb{R}$ is the shape parameter called the extreme value index.

The following celebrated result then links the excesses over a large threshold with the GPD distributions.
\begin{theorem}[Pickands--Balkema--de Haan Theorem]
Let $X$ be a random variable with distribution function $F$. Then the only possible non-degenerate limit of $\amsmathbb{P}(X - u \leq a(u) x +b(u) \mid X > u)$ with measurable functions $a(u)>0$ and $b(u)$ is $G_\gamma$.
\end{theorem}
The parameter $\gamma$ not only governs the tail behavior, but also categorizes distribution into three distinct so-called max-domains of attraction (MDA). Thus, we say that $F$ belongs to the Fr\'echet, Gumbel, or Weibull MDA if $\gamma > 0$, $\gamma = 0$, or $\gamma < 0$, respectively. We then write $F\in\mathcal{D}(G_\gamma)$.

Similarly, for any cdf $F$ define %its survival function as $\bar F,$ that is $\bar F = 1-F,$ and 
its \textit{tail counterpart} as
\begin{equation}F^t(x) = \frac{F(xt) - F(t)}{1 - F(t)}, \quad x\ge 1,\end{equation} which is itself a cdf and of practical utility in the sequel. Further, let $\gamma^t_0,$ $\gamma^t_1$ and $\gamma^t_2$ be defined similarly to \eqref{gamma0}, \eqref{gamma1} and \eqref{gamma2}, by replacing $H,  H^0$ and $ H^1$ with $H^t,$ $ H^{0,t}$ and $ H^{1,t},$ respectively, where
\begin{equation} H^{0,t}(x) = \int_1^{x} (1 - F^t(z)) G^t(\dd z)\:\:\: \text{ and }\:\:\:  H^{1,t}(x) = \int_1^{x} (1 - G^t(z-)) F^t(\dd z).\label{h01t}\end{equation}

%Finally, for future reference, define
%\[C^t(x) = \int_{1}^x \frac{G^t(\dd y)}{(1 - H^t(y))(1 - G^t(y))}.\]

%The sequence of integers $k_n$ is an intermediate sequence if
%\begin{equation*}k_n/n\to 0 \quad \text{and}\quad k_n\to\infty \quad \text{as} \quad n\to\infty.\end{equation*} 

\section{Extreme Kaplan--Meier Integrals}\label{sec:ekmi}
This section studies the consistency and asymptotic normality of extreme Kaplan--Meier integrals for regularly varying tails, with the main results being the decomposition of Theorem~\ref{T4}, and the asymptotic normality of Corollary~\ref{R2}, though the intermediate results are of independent interest.

We consider the case where the variables $X_i$ and the censoring mechanism $Y_i$ are both heavy-tailed, i.e. belonging to the Fr\'echet max-domain of attraction. In later sections this restriction is relaxed to also allow them to belong to the Gumbel and Weibull max-domains of attraction. More specifically, we assume now that $F,$ $G$, and hence $H$, are regularly varying cdfs with respective tail indices $1/\gamma_F,$ $1/\gamma_G$ and $1/\gamma_H,$ where by \eqref{hfg} we obtain \[1/\gamma_H = 1/\gamma_F + 1/\gamma_G.\] In this case, by regular variation of $F,$
\begin{equation}1 - F^t(x) \to x^{-1/\gamma_F}\quad \text { as } \quad t\to\infty, \label{tailfunction}\end{equation} and similarly for $G$ and $H$. %Note that the latter convergence is uniform in $x>1$, which is a standard result provided by the uniform convergence theorem in~\cite{bingham}.

%{ Similarly to Lemma 3.4.1 in~\cite{dehaan},} it follows that the normalized upper order statistics $\{Z_{n-k+i,n}/Z_{n-k,n}\}_{i=1}^{k}$ given $Z_{n-k,n} = t$ for $k<n$ are distributed as a vector of order statistics $\{V^t_{i,k}\}_{i=1}^k$ of a sample $\{V^t_{i}\}_{i=1}^k$ having cdf $H^t$ that converges to the Pareto distribution with index $1/\gamma_H$ as $t\to\infty$. However, to study the distributional convergence to the Pareto law with the correct target index $1/\gamma_F$ it is convenient to consider a product-limit construction, as is given in the next section.

Let $Z_{1,n} \le \ldots \le Z_{n,n}$ be the order statistics of a randomly right censored sample $\{(Z_i,\delta_i)\}$ and let $\{\delta_{[i:n]}\}_{i=1}^n$ be the concomitants of $\{Z_{i,n}\}_{i=1}^n.$ We are interested in analyzing the extreme value behavior of the datapoints $X_1,\dots,X_n$, which is possible through the study of $Z_{i,n}$ and $\delta_{[i:n]}$.
The following estimator targets the tail of the distribution while accounting for censoring effects.
\begin{definition}\rm
The \textit{extreme Kaplan--Meier estimator} (EKM) is given by\footnote{{We assume $0^0 = 1$ by convention.}}
\begin{equation}\label{kaplanmeierextreme} \amsmathbb{F}_{k,n}(x) = 1 -\prod_{i=1}^k \left[1 - \frac{\delta_{[n-i+1:n]}}{i}\right]^{I(Z_{n-i+1,n}/Z_{n-k,n}\le x)}.\end{equation}
\end{definition}

Let $F, G,$ and therefore $H$ be eventually continuous (i.e. continuous to the right of a large enough value). The asymptotic behavior of extreme Kaplan--Meier integrals is studied below, defined as integrals with respect to the EKM estimator:
\begin{equation}S_{k,n}(\varphi) = \int \varphi\, \dd \amsmathbb{F}_{k,n}.\label{extremeKM}\end{equation}
{ It is easily seen from \eqref{kaplanmeierextreme} that 
\begin{equation}\label{explicit}S_{k,n}(\varphi) = \sum_{i=1}^k \omega_{ik}\; \varphi\left(\frac{Z_{n-i+1,n}}{Z_{n-k,n}}\right),\end{equation}
where for $1\le i\le k,$
\begin{equation}\label{nik}\omega_{ik} = \frac{\delta_{[n-i+1:n]}}{i} \prod_{j=i+1}^k \left[\frac{j-1}{j}\right]^{\delta_{[n-j+1:n]}}.\end{equation} 
}
We assume that $\varphi$ is { a continuous and eventually monotone function on $[1,\infty).$}%\footnote{{ It suffices to require almost everywhere continuity with respect to Lebesgue measure.}}.}

%\begin{definition}
%    A continuous and eventually monotone function $\bar\varphi:[1,\infty)\to\amsmathbb{R}_+$ such that $|\varphi| \le \bar \varphi$, is called an envelope of $\varphi.$
%\end{definition}
%\noindent In the sequel, whenever we impose conditions on an envelope, we tacitly assume that the envelope exists. Otherwise, the condition is violated.  The use of envelopes is technical in nature, to avoid requiring %continuity and
%eventual monotonicity directly on $\varphi$  {(and later of $\varphi^\prime$).} For a well-behaved $\varphi$ function, the corresponding envelope can usually be chosen as $\bar\varphi=|\varphi|$.\\

%Important issues when studying EKM integrals arise from the fact that the ``sample'' is no longer independent. { Hence to prove our results, as a preliminary step we consider the behavior of the EKM integral \eqref{extremeKM} conditioning on $Z_{n-k,n} = t.$} 

We now provide the extreme value analogue to Theorem 1.1 of \cite{stute}, where the decomposition is now in terms of sums of functions of upper order statistics. Setting $\varphi=\log$ in \eqref{extremeKM}, we get a version of the celebrated Hill estimator \cite{hill1975simple} of positive $\gamma_F$ in presence of censoring, see  {Corollary~\ref{CorHill} for details} and also \cite{einmahl2008statistics} and \cite{worms2014new} for related estimators.

We impose the following key condition, which is natural for the limit variance to exist.

\begin{condition}\label{envelope_moment}
{ For some $\varepsilon>0$
\begin{equation}\label{onlycondition} \int_{1}^\infty (\varphi(x))^2\,x^{\alpha(\varepsilon)} \dd x < \infty,\end{equation}
where $\alpha(\varepsilon) = 1/\gamma_G - 1/\gamma_F -1 +\varepsilon.$ }
\end{condition}
{ We call the sequence of integers $k_n$ intermediate if
\begin{equation*}k_n/n\to 0 \quad \text{and}\quad k_n\to\infty \quad \text{as} \quad n\to\infty.\end{equation*}}
\begin{theorem}[Key decomposition]\label{T4} 
Let $F$ and $G$ be regularly varying and eventually continuous cdfs, and $\varphi$ be continuous and eventually monotone. Assume Condition~\ref{envelope_moment}, and let $k=k_n$ be an intermediate sequence. Then the following random sequence
\begin{eqnarray}\nonumber
r_{k,n} &=& S_{k,n}(\varphi)\\ \label{T4new} &&-\Big\{\frac{1}{k} \sum_{i=1}^k \varphi(Z_{i,k}^\ast)\gamma_0^t(Z_{i,k}^\ast) \delta_{[n-k+i:k]}\!+\!\frac{1}{k}  \sum_{i=1}^k \gamma_1^t(Z_{i,k}^\ast) (1\!-\!\delta_{[n-k+i:k]})\!-\!\frac{1}{k}  \sum_{i=1}^k \gamma_2^t(Z_{i,k}^\ast)\Big\} 
\end{eqnarray} is such that for every $\epsilon>0$
\begin{align*}
\lim_{n\to\infty} 
\limsup\limits_{t\rightarrow\infty}\pr\big(|k^{1/2} r_{k,n}|>\epsilon|\,Z_{n-k,n}=t\big) = 0.
\end{align*}
{ Here, $(Z^\ast_{1,k},\dots,Z^\ast_{k,k})=( Z_{n-k+1,n}/Z_{n-k,n},\ldots , Z_{n,n}/Z_{n-k,n})$.}
%where $\{\delta^t_i\}_{i=1}^k$ are such that their concomitants $\{\delta^t_{[i:k]}\}_{i=1}^k$ equal $\{\delta_{[n-k+i:n]}\}_{i=1}^k$ a.s.
\end{theorem}

In simple words, Theorem~\ref{T4} states that $S_{k,n}(\varphi)$ given $Z_{n-k,n} = t$ can be represented as a sum of i.i.d. random variables plus a remainder {term, which is negligible of order $k^{-1/2}$ for large $n$, uniformly in large $t$ (confer with the proof for further details)}. The above representation, without additional assumptions, allows us to prove that the extreme Kaplan--Meier integral $S_{k,n}(\varphi)$ \eqref{extremeKM} is a consistent estimator for \[S_\circ(\varphi) := \int_1^\infty \varphi \dd F^\circ,\] where $F^\circ(x) = (1 - x^{-1/\gamma_F})I(x\ge1).$
\begin{theorem}[Weak consistency of EKM integrals]\label{Tconsistency} Assume the conditions of Theorem~\ref{T4}. Then
\[S_{k,n}(\varphi) \stackrel{P}{\to} S_\circ(\varphi).\]
\end{theorem}
\begin{remark} Weak consistency of extreme Kaplan--Meier integrals under laxer assumptions than Condition~\ref{envelope_moment} may be possible to establish with an approach as in Theorem 1.1 in~\cite{stutewang}, where the existence of the first moment of the (possibly modified) target cdf is assumed for consistency of the standard Kaplan--Meier integral. The proof of the latter result uses some specialized techniques which fall far from our setup, for instance by using reverse-time supermartingales. 
%In order not to further overload the notation and proofs, and to keep the narrative coherent, we simply continue to use the natural Condition~\ref{envelope_moment} in Theorem~\ref{Tconsistency}.  
\end{remark}

The next theorem states the asymptotic normality of EKM integrals under the assumptions of Theorem~\ref{T4}. However, note that the centering sequence is random.

To formulate this theorem, let us introduce some notation. Denote by $X^\circ$ and $Y^\circ$ independent random variables with cdfs $F^\circ$ and $G^\circ,$ respectively, where $F^\circ$ is defined above and $G^\circ(x) = (1 - x^{-1/\gamma_G})I(x\ge 1).$  Define $V^\circ = \min(X^\circ, Y^\circ)$ and $\delta^\circ = I(X^\circ\le Y^\circ).$ It is straightforward to see that $V^\circ\independent\delta^\circ$.
Define also $H^\circ,$ $ H^{j,\circ},$ $j=0,1,$ by replacing $F^t$ and $G^t$ in the definitions of $H^t,$ $ H^{j,t},$ $j=0,1,$ respectively, with  $F^\circ$ and $G^\circ.$ Finally, define $\gamma_j^\circ,$ $j=0,1,2$ by replacing $H^t,$ $ H^{j,t},$ $j=0,1,$ with $H^\circ,$ $ H^{j,\circ},$ $j=0,1,$ in the definitions of $\gamma_j^t,$ $j=0,1,2,$ respectively. 

Define the random variable
\begin{equation}W^\circ(\varphi) = \varphi(V^\circ)\gamma_0^\circ(V^\circ)\delta^\circ + \gamma_1^\circ(V^\circ)(1 - \delta^\circ) - \gamma_2^\circ(V^\circ).\label{wcirc}\end{equation}
The variance of $W^\circ(\varphi)$ turns out to be the correct asymptotic variance in all of our weak convergence results.

\begin{theorem}\label{T1an} Assume the conditions of Theorem~\ref{T4}. Then
\begin{align}\label{main}\sqrt{k}\int \varphi\, \dd (\amsmathbb{F}_{k,n} - F^{Z_{n-k,n}}) \stackrel{d}{\to} N(0, \sigma^2_\varphi),\end{align}
where $\sigma^2_\varphi ={\rm var} (W^\circ(\varphi)).$
\end{theorem}

\begin{remark}\label{Rreplacement}
It is possible to replace the random centering sequence $\int \varphi \dd F^{Z_{n-k,n}}$ in \eqref{main} with its non-random counterpart $\int \varphi \dd F^{U_H(n/k)}$, where $U_H$ is the tail quantile function of $H,$ i.e. 
\[U_H(x) = \inf\{y: H(y) \ge 1 - 1/x\}, \quad x>1,\]
under a mild condition imposed on $U_H$ and $U_F.$ We provide this result in Section \ref{AppA} of the Supplementary Material. This is of independent interest, since it is not always possible to replace a non-random centering sequence with a constant in limit relations appearing in extreme value theory. 
%For instance, it is not possible in Smirnov's lemma and its corollary regarding limit behavior of intermediate order statistics; see, e.g., Theorems 2.2.1 and 2.4.1 in~\cite{dehaan}. 
\end{remark}

A non-random replacement of the centering sequence is possible in our case if the intermediate sequence $k = k(n)$ does not grow too fast with respect to $n,$ {as is precised in Theorem~\ref{T3an} below}. The exact speed is determined through the quality of the regular variation property of $U_F,$ which in turn is governed by a second-order parameter denoted by $\rho,$ see~\cite{resnickdehaan}. Thus, to control this speed it is necessary to impose the following standard second-order condition.

\begin{condition}\label{secondorder}
The tail quantile function of $F$ satisfies
\begin{align}\label{secondorderrelation} \lim_{t\to\infty} \frac{U_F(tx)/U_F(t) - x^{\gamma_F}}{a(t)} = x^{\gamma_F} h_\rho(x), \quad x\ge 1,\end{align}
for some $\rho\le0,$ $a(t)$ a positive or negative function, regularly varying with index $\rho,$ and satisfying $a(t) \to 0$ as $t\to\infty,$ and
\begin{equation}\label{h}h_\rho(x) = \left\{\begin{array}{cc} \log x, & \text {if } \rho = 0, \\ (x^\rho - 1)/\rho, & \text{ otherwise}.
\end{array}\right.\end{equation}
\end{condition}

\noindent Properties and possible choices for the function $a(t)$ may be found in Section 2.3, \cite{dehaan}.
\begin{theorem} \label{T3an} Assume the conditions of Theorem~\ref{T4}. Furthermore, assume $\varphi$ %satisfies Condition~\ref{cond:phi_deriv} 
has a continuous and eventually monotone derivative $\varphi^\prime,$ 
and $U_F$ satisfies the second-order Condition~\ref{secondorder}. Let $\sqrt{k}\,a(c_n)\to \lambda$ with finite $\lambda$ and $c_n = (1 - F(U_H(n/k)))^{-1}.$ %and if $\lambda\neq0$ assume additionally that $\varphi^\prime$ is %a.e. continuous. 
Then
\begin{equation}\label{T3anassertion}\sqrt{k}\int \varphi\, \dd (F^{Z_{n-k,n}} - F^{\circ}) \stackrel{P}{\to} \lambda C(\gamma_F, \rho),\end{equation} where
\[C(\gamma_F, \rho) = \int_1^\infty x \varphi^\prime(x) h_\rho(x^{1/\gamma_F}) F^\circ(\dd x).\]
\end{theorem}
\noindent As an immediate consequence of Theorems~\ref{T1an} and \ref{T3an}, we get the following result.
\begin{corollary}\label{R2} Under the assumptions of Theorem~\ref{T3an}, we have
\begin{equation}\sqrt{k}\int \varphi\, \dd (\amsmathbb{F}_{k,n} - F^{\circ}) \stackrel{d}{\to} N(\lambda C(\gamma_F, \rho), \sigma^2_\varphi).\label{seq2}\end{equation}
\end{corollary}

\begin{remark}
Notice that for Theorem~\ref{T3an} to hold, we have required Condition~\ref{envelope_moment}. In particular, if $\varphi$ grows slower than a power function, the condition $\gamma_F<\gamma_G$ implies the required integrability. This non-heavy censoring regime for asymptotic normality of censored extremes has been observed before for related estimators; see for instance
Remark 1 in~\cite{bww}.
\end{remark}

%\begin{remark}
%Analysing the joint behavior of empirical versions of $H^{0,t}$ and $H^{1,t}$ is a somewhat direct route to prove asymptotic normality of $S_{k,n}(\varphi)$ using the functional delta method. However, there is an important caveat in that such assertion can only be proved for differentiable $\varphi$ and integrals up to an endpoint $T$ with $F(T)<1$,  {which can be interpreted as $\varphi$ satisfying a technical truncation condition}. The above approach allows us to establish asymptotic results on the \textit{maximal interval}, which in turn allows integration of infinite-support $\varphi$ functions.
%\end{remark}

{ As an illustration of Theorem~\ref{T3an} and Corollary~\ref{R2}, let us consider asymptotic behavior of the censored analog of the celebrated Hill estimator \cite{hill1975simple} of a positive extreme value index. Setting $\varphi = \log$ in \eqref{extremeKM}, we get  
\begin{align}\widehat\gamma^{c,Hill}_{k} = S_{k,n}(\log) = \int_1^\infty \log(x)\amsmathbb{F}_{k,n} ({\rm d}x).\label{cHill_def}\end{align}
This estimator was introduced in equation (7) of \cite{worms2014new}, under a very different but equivalent expression, and a consistency result was provided. To the best of our knowledge, its asymptotic normality has not been established yet. The following result fills this gap in the literature.
\begin{corollary}\label{CorHill} Let $F$ and $G$ be regularly varying and eventually continuous cdfs with 
$\gamma_G>\gamma_F,$ Condition~\ref{secondorder} hold for $U_F$ and $k = k_n$ be an intermediate sequence. Let $\sqrt{k}\,a(c_n)\to \lambda$ with finite $\lambda$ and $c_n = (1 - F(U_H(n/k)))^{-1}.$ Then 
\[\sqrt{k}(\widehat\gamma^{c,Hill}_{k} - \gamma_F) \xrightarrow{d} N\Big(\frac{\lambda}{1 - \rho}, \frac{\gamma_G \gamma_F^2}{\gamma_G - \gamma_F}\Big).\]
\end{corollary}
}
{  \begin{remark}\label{Rnormalization} Notice that $S_{k,n}(1)=\int 1 \dd \amsmathbb{F}_{k,n} = \amsmathbb{F}_{k,n}(\infty)\le 1$, and then $\amsmathbb{F}_{k,n}/S_{k,n}(1)$ is a true cdf. One obtains from Corollary~\ref{R2} that $\sqrt{k}(S_{k,n}(1)-1)\stackrel{P}{\to}0$. In particular, this means that estimators normalized by $\amsmathbb{F}_{k,n}(\infty)$ have the same distributional limits as their non-normalized counterparts. Consequently,
$$\widehat\gamma^{c,Hill}_{k}/\amsmathbb{F}_{k,n}(\infty)-\widehat\gamma^{c,Hill}_{k}=o_\pr(k^{-1/2}).$$
The above relation is practically important, since normalized estimators can have slightly better finite-sample performance. 
We adopt normalization in the simulation studies of Section~\ref{sec:sim}.
\end{remark}
}
It is also of independent interest to investigate the asymptotic properties of the extreme Kaplan--Meier estimator $\amsmathbb{F}_{k,n}$ itself. Consider a function $\varphi_0(x) = I(x\le x_0)$ for some $x_0>1,$ then $S_{k,n}(\varphi_0) = \amsmathbb{F}_{k,n}(x_0).$ Since $\varphi_0$ is not continuous, one can consider a EKM integral of a continuously differentiable envelope function $\varphi_\epsilon,$ and then taking $\epsilon\to 0$ for $S_{k,n}(\varphi_\epsilon)$. 

%auxiliary function $\varphi_\epsilon,$ which is $0$ as $x\ge x_0+\epsilon,$ $1$ as $x\le x_0,$ and decreasing and continuously differentiable on $[x_0, x_0+\epsilon]$ with $\epsilon>0.$ Applying Theorem \ref{Tconsistency} and Theorem \ref{T3an} to $S_{k,n}(\varphi_\epsilon)$ and taking $\epsilon\to 0,$ we derive the following result. 
\begin{corollary}\label{cor:ekme} Let $F$ and $G$ be regularly varying and eventually continuous cdfs with 
$\gamma_G>\gamma_F.$ Then $\amsmathbb{F}_{k,n}(x_0)$ is consistent for $F^\circ(x_0).$ Moreover, under the assumptions of Corollary \ref{CorHill} we have
\[\sqrt{k}\big(\amsmathbb{F}_{k,n}(x_0) - F^\circ(x_0)\big) \xrightarrow{d} N\left(\lambda \frac{h_\rho(x_0^{1/\gamma_F})}{\gamma_F x_0^{1/\gamma_F}}, \frac{\gamma_H}{\gamma_F} \frac{x_0^{1/\gamma_G} - x_0^{-1/\gamma_F}}{x_0^{1/\gamma_F}}\right).\]
\end{corollary}

In particular, the above result can be used to prove consistency of censored analogues of probability weighted moments and maximum likelihood methods, which are well accepted in extreme value theory, see, e.g., Chapters 3.4 and 3.6.1 in \cite{dehaan}, and often applied to extreme quantile and tail estimation. 

%This is be the topic of our future research.

%\begin{remark} It is also possible to relax the eventual monotonicity assumption imposed on $\varphi$ by replacing $\varphi$ with continuous and eventually monotone function $\bar\varphi$ such that $\bar\varphi(x) \ge |\varphi(x)|,$ $x\ge 1,$ in Condition \ref{envelope_moment}. However, to the best of our knowledge, all functions used in similar situations in extreme value theory possess eventual monotonicity. \end{remark}

\section{Generalized residual estimators for non-censored data}\label{genresidual}

{ Thus far we have studied EKM integrals when distributions of both the event and the censoring variables are heavy-tailed. However, the extension of this methodology to all max-domains of attraction is not straightforward, as we precise below. In particular, the limits in probability of EKM integrals may not depend on $\gamma_F$ if the event variable distribution is not heavy-tailed, making them unsuitable for estimation.

To overcome this difficulty and find an appropriate normalization of EKM integrals leading to its non-degenerate limit behavior, we turn in this section to the theory of residual estimators developed in~\cite{segers} for heavy-tailed distributions and extend it to all max-domains of attraction. This route does not suffer normalization drawbacks. A key role in this extension is played by the so-called candidate functions having certain regular limit behavior; see Chapter 1 and Appendix B in~\cite{dehaan} and Section~\ref{Sectionexamples} in a Supplementary Material for examples.}

%\subsection{Generalized residual estimators} \label{genresidual}

{ This section studies generalized residual estimators $R^f_{k,n}$ \eqref{gre}, which are intrinsically related to generalized EKM integrals \eqref{new_ekmi}. The latter integrals are the main object of interest in Section~\ref{sec:exten}, where the methodology developed in Section~\ref{sec:ekmi} is extended to all max-domains of attraction. The critical link is viewing generalized residual estimators as integrals over an empirical cdf of normalized order statistics %$\{\xi_{n-i+1,n}/\xi_{n-k,n}\}_{i=1}^k$ 
of i.i.d. Pareto($1$) random variables. %$\{\xi_i\}_{i=1}^n.$
Upon replacement of this empirical cdf with the EKM estimator \eqref{fxikn}, we obtain a generalized EKM integral \eqref{new_ekmi} capable of dealing with censoring effects. In fact, the same kind of replacement is implicitly made in Section~\ref{sec:ekmi}, where standard EKM integrals could be seen as being derived from standard residual estimators.}

As before, let $X_1,\dots,X_n$ denote an i.i.d. sequence of random variables with cdf $F,$ and $X_{1,n} \le \ldots \le X_{n,n}$ denote the order statistics of $\{X_i\}_{i=1}^n.$  Denote by $U_F(t),$ the tail quantile function of $F,$ i.e. $U_F(t) = \inf\{x: F(x) \ge 1 - 1/t\}.$ When $F$ belongs to the Fr\'echet max-domain of attraction (that is $F\in \mathcal{D}(G_\gamma)$ with $\gamma>0$), \cite{segers} studied the limit behavior of (standard) residual estimators defined by \begin{equation}\label{standardresidual}R_{k,n} = \frac{1}{k}\sum_{i=1}^{k} \vartheta (X_{n-i+1,n}/X_{n-k,n}),\end{equation} where $\vartheta : [1,\infty) \to \amsmathbb{R}$ is some suitable function, which is here assumed to be %almost everywhere 
continuous.

\begin{theorem}[Theorem 4.1 in \cite{segers}]\label{T1} Let $|\vartheta (x)| \le A x^c,$ $x\ge 1,$ for some $A>0$ and $c<(2\gamma)^{-1}.$ Let $\xi, \xi_1, \xi_2, \ldots$ be independent Pareto$(1)$ random variables such that $X_i = U_F(\xi_i),$ $i\ge 1.$ Define $\mu(t) = \amsmathbb{E}[\vartheta (U_F(\xi t)/U_F(t)) I(U_F(t)>0)].$ Let $k=k_n$ be an intermediate sequence. Then
\begin{equation}\label{segers4.1} \sqrt{k}\big(R_{k,n} - \mu(\xi_{n-k,n})\big) \stackrel{d}{\to} N(0, {\rm var} [\vartheta (\xi^\gamma)]).\end{equation}
\end{theorem}
In further theorems 4.2 and 4.5 ibid., it was proven that $\mu(\xi_{n-k,n})$ in \eqref{segers4.1} can be replaced with the non-random limit $\amsmathbb{E}[\vartheta (\xi^\gamma)]$ under some additional conditions imposed on $F,$ $\vartheta$ and $k.$ 

{ Residual estimators are closely related to the EKM integrals studied in Section~\ref{sec:ekmi}. Indeed, the former can be represented as integrals over an empirical cdf built by normalized order statistics $\{X_{n-i+1,n}/X_{n-k,n}\}_{i=1}^k$: 
\begin{equation}\label{newrkn}R_{k,n} = \int \vartheta \dd F_{k,n},\end{equation} where $F_{k,n}(x) = k^{-1} \sum_{i=1}^k I(X_{n-i+1,n}/X_{n-k,n} \le x).$ Replacing $F_{k,n}$ with the EKM estimator $\amsmathbb{F}_{k,n}$ \eqref{kaplanmeierextreme} in \eqref{newrkn} yields an EKM integral \eqref{extremeKM}. The remainder of this section studies the fully-observed case. We refer to Section~\ref{sec:exten} for the generalization to censored data.}

{ The first difficulty one encounters is that} the direct extension of Theorem~\ref{T1} to the Gumbel and Weibull max-domains of attraction leads to a degenerate limit in \eqref{segers4.1} in almost all cases. Recall the notation $\tau_F = \inf\{x: F(x) = 1\},$ the right endpoint of $F.$ Specifically, we have the following result, whose proof is analogous to the proof of Theorem~\ref{T1}, and thus omitted.

\begin{lemma}\label{L1_bis}  Let $F \in \mathcal{D}(G_\gamma)$ with $\gamma\le 0$ and $\tau_F>0.$ Let $\vartheta : [1, \infty) \to \amsmathbb{R}$ satisfy $|\vartheta (x)|\le Ax^c,$  $x\ge 1,$ for some $A>0$ and $c>0.$ Assume $\mu(t),$ $k$ and $\xi, \xi_1, \xi_2, \ldots$ are as in Theorem~\ref{T1}. Then
\begin{equation}\label{segers_direct_mod} \sqrt{k}\big(R_{k,n} - \mu(\xi_{n-k,n})\big) \stackrel{P}{\to}0.\end{equation} 
\end{lemma}
A reason of a degenerate limit in \eqref{segers_direct_mod} is that for $\gamma\le 0$ and $\tau_F > 0$ the ratio $U_F(xt)/U_F(t),$ closely related to normalized order statistics, tends to $1$ for every $x>0$ as $t\to\infty$ (which follows from techniques in~\cite{dehaan}),
see for details Theorem~\ref{T1gen} and Corollary~\ref{Cor1}  {of the Supplementary Material.}

It follows that to derive a non-degenerate limit, a faster than $\sqrt{k}$ normalizing sequence is required in \eqref{segers_direct_mod}. Consequently, a downside of using a residual estimators technique to develop extreme value estimators ({both in presence of and without random censorship}) is that we may be faced with situations where positive and negative values of the extreme value index are allowed, but the rate of convergence of the estimator inconveniently varies. Another disadvantage is that $\vartheta(1)$ does not depend on $\gamma<0.$ This makes it impossible to construct extreme value index estimators based solely on $R_{k,n}.$

For this reason, let us consider an alternative and general approach, { through \textit{generalized residual estimators}}. As before, let $\xi, \xi_1, \xi_2, \ldots$ be independent Pareto$(1)$ random variables with $n$-th order statistics denoted by $\xi_{1,n}\le\ldots\le \xi_{n,n},$  {such that $X_i = U_F(\xi_i)$, $i\ge 1$.}

\begin{definition}\label{D5}
Let $f:[1,\infty)\times[1,\infty)\to [0,\infty)$, and $g,\,g_1,\,g_2:[1,\infty)\to [0,\infty)$ be measurable functions such that for all $x$ 
\begin{equation}\label{cond1} f(x,t) \to g(x), \quad t\to\infty, \end{equation}
and there exists $T>0$, such that for $t\ge T$
\begin{equation}\label{cond2} g_1(x) \le f(x,t)\le g_2(x).\end{equation}
{ We call $f$ a \emph{candidate function}. If $f(x,t)$ tends to $\infty$ as $x\to x_0$ for some $x_0$ which can be infinite, we assume that $g_1,\,g_2$ tend to $\infty$ together with $f(x,t)$.}
\end{definition}
{  \begin{example}\label{E1} Let $F\in \mathcal{D}(G_{\gamma}).$ Examples of candidate functions include \[\frac{U_F(xt)}{U_F(t)}, \; \frac{U_F(xt) - U_F(t)}{a_F(t)}, \; \frac{\log U_F(xt) - \log U_F(t)}{a_F(t)/U_F(t)},\quad \text{and}\quad \frac{\tau_F - U_F(xt)}{\tau_F - U_F(t)}\]
with corresponding limits $x^{\gamma_+},$ $h_\gamma(x),$ $h_{\gamma_-}(x)$ and $x^{\gamma_{-}},$ as $t\to\infty,$ where $a_F(t)$ is a so-called auxiliary function, see, e.g., Theorem 1.1.6 in~\cite{dehaan}. In Section~\ref{Sectionexamples} in Supplementary Material we discuss properties of the first three above candidate functions, whereas the first, third and fourth ones are utilized in Section~\ref{sec:evi_est}.
%, see also \cite{dehaan}. 
\end{example}}

Let $\vartheta :[0,\infty)\to\amsmathbb{R}$ be continuously differentiable and eventually monotone.
%\footnote{{  If $\vartheta$ and $\vartheta^{\prime}$ are eventually non-decreasing (resp. non-increasing), then it suffices to require that $g_2$ (resp. $g_1$) tends to $\infty$ together with $f(x,t).$}}. 
We are interested in the limit behavior of {the following \textit{generalized residual estimator}:}
\begin{equation}\label{gre}R^f_{k,n} = \frac{1}{k} \sum_{j=1}^{k} \vartheta \left(f\Big(\frac{\xi_{n-j+1,n}}{\xi_{n-k,n}}, \xi_{n-k,n}\Big)\right) = \int \vartheta(f(x, \xi_{n-k,n})) F^\xi_{k,n}(\dd x),\end{equation}
 {where $k<n$} and %$k=k_n$ is an intermediate sequence and 
$F^\xi_{k,n}(x) = k^{-1} \sum_{j=1}^k I(\xi_{n-j+1,n}/\xi_{n-k,n}\le x).$  {Note that taking $f(x,t) = U_F(xt)/U_F(t)$ in \eqref{gre} we get the standard residual estimator \eqref{standardresidual}.}

For the asymptotic results we require the following quantity:
\begin{equation}\label{muf}\mu_f(t) = \amsmathbb{E} \big[\vartheta (f(\xi,t))\big].\end{equation}

\begin{theorem}\label{T1gen} Let $k=k_n$ be an intermediate sequence and $\vartheta$ be continuous and eventually monotone. Assume for some $\delta>0$ 
\begin{equation}\label{cond4}\amsmathbb{E}\Big[\max \big\{|\vartheta (g_1(\xi))|, |\vartheta (g_2(\xi))|\big\}\Big]^{1+\delta}<\infty.\end{equation} Then 
$R^f_{k,n}$ is consistent for $\amsmathbb{E}[\vartheta (g(\xi))].$ If, moreover, \eqref{cond4} is satisfied for some $\delta>1,$ then
\begin{equation}\label{T1genassertion} \sqrt{k}\big(R^f_{k,n} - \mu_f(\xi_{n-k,n})\big) \stackrel{d}{\to} N\big(0, {\rm var} [\vartheta (g(\xi))]\big).\end{equation}
\end{theorem}

To replace $\mu_f(\xi_{n-k,n})$ with $\mu_f(\infty) := \amsmathbb{E}[\vartheta (g(\xi))]$ in Theorem~\ref{T1gen}, we need to impose another condition (compare with Condition~\ref{secondorder}).

\begin{condition}\label{second_order_2}
The following second-order type condition holds
\begin{equation}\label{cond5}\frac{f(x,t) - g(x)}{\hat a(t)} \to \tilde g(x), \quad t\to\infty,\end{equation} where $\hat a$ is some positive or negative function. Furthermore, there exists a measurable function $\hat g$ such that for all $x$ and $t\ge T$
\begin{equation}\label{cond6} |f(x,t) - g(x)| \le |\hat a(t)| \hat g(x).\end{equation}
\end{condition}

\noindent In particular, if $\tilde g$ is regularly varying, then \eqref{cond5} implies \eqref{cond6}. 

{  Notice that all candidate functions from Example~\ref{E1} satisfy Condition~\ref{second_order_2}.
%, see Section~\ref{Sectionexamples} in Supplementary Material and \cite{dehaan}. 
For instance, for the candidate function $a^{-1}_F(t) U_F(t)(\log U_F(tx) - \log U_F(t))$ it reads as follows:
\begin{equation}\label{secondorderlog}\lim_{t\to\infty}\frac{\frac{\log U_F(tx) - \log U_F(t)}{a_F(t)/U_F(t)} - h_{\gamma_-}(x)}{Q(t)} = \int_{1}^x s^{\gamma_--1}\int_{1}^s u^{\rho^\prime-1} \dd u \dd s=: H_{\gamma_-, \rho^\prime}(x),\end{equation}
for $ x>0$, some $\rho^\prime\le 0$ and $Q(t)$ positive or negative with $Q(t) \to 0, t \to \infty,$ see, e.g., formula (3.5.11) in~\cite{dehaan}.  Here, the function $H_{\gamma, \rho}$ can be re-expressed as
\begin{equation}\label{hgammarho}H_{\gamma, \rho}(x) = \frac{1}{\rho}\left(\frac{x^{\gamma+\rho}-1}{\gamma+\rho} - \frac{x^\gamma - 1}{\gamma}\right),\end{equation} which for the cases $\gamma = 0$ and $\rho = 0$ is understood to be equal to the limit of \eqref{hgammarho} as $\gamma\to 0$ or $\rho \to 0,$ respectively. 
%Later we use \eqref{secondorderlog} in Section~\ref{sec:evi_est}.}

\begin{theorem}\label{T2gen} Let the conditions of Theorem~\ref{T1gen} and Condition~\ref{second_order_2} hold. Assume { $\sqrt{k}\hat a\big((1 - F(X_{n-k,n}))^{-1}\big)\stackrel{P}{\to}\lambda$} with finite $\lambda,$ and $\vartheta^\prime$ is continuous and eventually monotone.
Assume also
\begin{equation}\label{cond8}\amsmathbb{E}\Big[\hat g(\xi) \max \big\{|\vartheta^\prime(g_1(\xi))|, |\vartheta^\prime(g_2(\xi))|\big\}\Big]<\infty.\end{equation} 
Then
\begin{equation}\label{T2genassertion} \sqrt{k}\big(R^f_{k,n} - \amsmathbb{E}[\vartheta (g(\xi))]\big) \stackrel{d}{\to} N(C\,\lambda, {\rm var} [\vartheta (g(\xi))]\big),
\end{equation}
where $C = \amsmathbb{E}\big[\tilde g(\xi) \vartheta^\prime(g(\xi))\big].$
\end{theorem}

 {Applications of Theorem~\ref{T2gen} for all max-domains of attraction can be found in Section~\ref{Sectionexamples} of the Supplementary Material.}

{Next, we come back to the censored setting and link the theory of generalized residual estimators to generalized EKM integrals for all max-domains of attraction in Section~\ref{sec:exten}. Thereafter, in Section~\ref{sec:evi_est}, the combination of methods developed in Section~\ref{sec:ekmi} and Section~\ref{sec:exten} yields as a larger example the moment estimator  of $\gamma\in\amsmathbb{R}$ itself.}

\section{{ Generalized extreme Kaplan--Meier integrals in all max-domains of attraction}} \label{sec:exten}
 {Extreme Kaplan--Meier integrals extend residual estimators to the censored case, and consequently possess a degenerate limit behavior if $F$ is not heavy-tailed. To overcome this difficulty we adopt ideas from Section~\ref{genresidual}, with the aim of extending the results of Section~\ref{sec:ekmi} to other max-domains of attraction. Thus, this section studies the resulting generalized extreme Kaplan--Meier integrals.} 

Let us use the same notation as in Section~\ref{sec:ekmi}. Crucially, unlike in Section~\ref{sec:ekmi}, we do not assume that $F,$ $G$ and $H$ are regularly varying distribution functions. Let us assume that $F,$ $G$ and therefore $H$ are continuous and consider a sequence $\xi, \xi_1, \xi_2, \ldots$ of independent Pareto$(1)$ random variables such that
$Z_i = U_H(\xi_i),$ $i\ge 1,$ and a function $f(x,t)$ satisfying Definition \ref{D5}. Since we are only interested in the tails, it suffices to require eventual continuity of $F$ and $G$ as in Section~\ref{sec:ekmi}. Let $\vartheta: [0,\infty) \to \amsmathbb{R}$ be a  continuous and eventually monotone function. %as in Section~\ref{genresidual}.
The following { \textit{generalized extreme Kaplan--Meier integral}} is of our main interest:
\begin{equation}\label{new_ekmi} S_{k,n}^f(\vartheta) = \int \vartheta(f(x, \xi_{n-k,n})) \amsmathbb{F}^\xi_{k,n}(\dd x), \end{equation}
where $\amsmathbb{F}^\xi_{k,n}$ is the extreme Kaplan--Meier estimator built by the sequence $\{\xi_i\},$ i.e.
\begin{equation}\label{fxikn}\amsmathbb{F}_{k,n}^\xi(x) = 1 -\prod_{i=1}^k \left[1 - \frac{\delta_{[n-i+1:n]}}{i}\right]^{I(\xi_{n-i+1,n}/\xi_{n-k,n}\le x)}.\end{equation} 
 {In the special case of $f(x,t) = U_H(xt)/U_H(t)$ (the first candidate function in Example~\ref{E1}) the integral \eqref{new_ekmi} reduces to the standard EKM integral \eqref{extremeKM}.}

 {The generalized EKM integral} $S^f_{k,n}(\vartheta)$ can be thought as an extreme Kaplan--Meier integral \eqref{extremeKM} for the function $\vartheta^{\xi_{n-k,n}}_f(x) =\vartheta(f(x,\xi_{n-k,n}))$ built through the sequence $\{\xi_i\}.$ The only, but substantial, difference is that the function $\vartheta^{\xi_{n-k,n}}_f,$ the analogue of $\varphi$ in \eqref{extremeKM}, is random. 
 {Crucially, generalized EKM integrals \eqref{new_ekmi} can be obtained from generalized residual estimators \eqref{gre} studied in Section~\ref{genresidual} by replacing $F_{k,n}^\xi$ with the EKM estimator $\amsmathbb{F}_{k,n}^\xi$.} {  %In terms of computation, 
Note that an explicit formula for the generalized  EKM integral \eqref{new_ekmi} is
\[S_{k,n}^f(\vartheta) = \sum_{i=1}^k \omega_{ik} \; \vartheta\left(f\Big(\frac{\xi_{n-i+1,n}}{\xi_{n-k,n}}, \xi_{n-k,n}\Big)\right),\] where $\omega_{ik},\, 1\le i\le k,$ are defined by \eqref{nik}. However, $S_{k,n}^f(\vartheta)$ depends on the unknown $\{\xi_i\}$ which makes it impossible in practice to compute it for candidate functions like the second and third ones of Example~\ref{E1}, which depend on an unknown auxiliary function $a.$ }{  Fortunately, for the latter candidate functions, combining two or more generalized EKM integrals yields the resulting quantity independent of the auxiliary function, see, e.g., Section~\ref{sec:evi_est}. Such combinations include $(S_{k,n}^f(x))^\ell/S_{k,n}^f(x^\ell)$ for $\ell\in \amsmathbb{R}$ and certain polynomials of $S_{k,n}^f((\log x)^m),$ $m\in \amsmathbb{N},$ such as $S_{k,n}^f(\log^2 x) - (S_{k,n}^f(\log x))^2.$}

 {As mentioned, we are interested in extending the results of Section~\ref{sec:ekmi} to all max-domains of attraction. In this context, the inverse function of $U_H,$ i.e. $U_H^{-1}(x) = 1/(1 - H(x)),$ plays an important role in standardizing the distributions across max-domains of attraction. Indeed, we always have that a.s. $\xi_i = U_H^{-1}(Z_i)=\min(U_H^{-1}(X_i), U_H^{-1}(Y_i))$ is standard Pareto. Together with the unchanged  indicators $\delta_i$, the latter construction is a unique way of obtaining analogous versions of  $X_i$ and $Y_i$ from Section~\ref{sec:ekmi}. Specifically, $U_H^{-1}(X_i)$ and $U_H^{-1}(Y_i)$ are now the component random variables
with distributions given by
}
%For formulating results of the latter section and proving them, cumulative distributions functions $F$ and $G$ of random samples $\{X_i\}$ and $\{Y_i\}$ play a crucial role, so we need to find their analogs in the setting of the current section. Below $U_H^{-1}$ denotes the inverse function of $U_H,$ i.e. $U_H^{-1}(x) = 1/(1 - H(x)).$ Since we use $\{\xi_i\} = \{U_H^{-1}(Z_i)\}$ as $\{Z_i\}$ in \eqref{new_ekmi} and $\{\delta_i\}$ remain the same as in Section~\ref{sec:ekmi}, the analogs of $\{X_i\}$ and $\{Y_i\}$ are $\{U_H^{-1}(X_i)\}$ and $\{U_H^{-1}(Y_i)\}.$ Indeed, $\xi_i = U_H^{-1}(Z_i) = \min(U_H^{-1}(X_i), U_H^{-1}(Y_i))$ and $\delta_i = I(X_i \le Y_i) = I(U_H^{-1}(X_i) \le U_H^{-1}(Y_i)),$ $i=1,\ldots,n.$ Denote cdfs of $\{U_H^{-1}(X_i)\}$ and $\{U_H^{-1}(Y_i)\}$ by $\tilde F$ and $\tilde G,$ respectively. Clearly,
\[\tilde F(x) = F(U_H(x)), \quad \tilde G(x) = G(U_H(x)), \quad x\ge 1,\] and, evidently, $\tilde H(x) = H(U_H(x)) = 1 - 1/x,$ $x\ge 1,$  {the common cdf of $\{\xi_i\}.$}  {From this point forward, the methods resemble those of Section~\ref{sec:ekmi}, now using that $S_{k,n}^f(\vartheta)$ is an EKM integral for the function $\vartheta_f^{\xi_{n-k,n}}$ and having $\{(\xi_i, \delta_i)\}$ as the ``observed'' vectors.} 

Let us define $\tilde H^{i,t},$ $i=0,1,$ and $\tilde \gamma^t_j,$ $j=0,1,2,$ similarly to the Section~\ref{sec:ekmi} by replacing $F^t,$ $G^t,$ $H^t,$ and $\varphi$ with $\tilde F^{t},$ $\tilde G^t,$ $\tilde H^t,$ and $\vartheta^t_f,$ respectively. In the same fashion, define the random variable
\begin{equation}\label{tildew}\tilde W^\circ(\vartheta) = \vartheta(g(\xi))\tilde \gamma_0^\circ(\xi)\delta^\circ + \tilde\gamma_1^\circ(\xi)(1 - \delta^\circ) - \tilde\gamma_2^\circ(\xi).\end{equation} Note that the random variables $\{\delta^t_i\}$ and $\delta^\circ$ remain unchanged from Section~\ref{sec:ekmi}.

\begin{condition}\label{C6} Assume $1-\tilde F$ is a regularly varying function with index $-\alpha_F,$ with $\alpha_F \ge 0.$ Assume also $\tilde F$ is differentiable and $\tilde F^\prime$ is eventually monotone. If $\alpha_F = 0$ then assume additionally that $\tilde F^\prime$ is regularly varying.\footnote{{ If $\vartheta$ is monotone and $f(x,t)$ is monotone in the first argument, then the conditions on eventual monotonicity of $\tilde F^\prime$ and even existence of $\tilde F^\prime$ can be omitted; see e.g. the proof of Theorem~\ref{Tmomentconsistency}.}} \end{condition}

Note that Condition~\ref{C6} implies that $1 - \tilde G$ is also a regularly varying function because $1 - \tilde G = (1 - \tilde H)/(1 - \tilde F)$ and $1 - \tilde H(x) = 1/x,\, x \ge 1.$ Denote $\alpha_G,$ the index of regular variation of $1 - \tilde G.$ Note that $\alpha_F$ and $\alpha_G$ belong to the interval $[0,1],$ moreover, $\alpha_F + \alpha_G = 1$ by the definition of $F,$ $G$ and $H$ and the properties of regularly varying functions. 
%Given $F$ and $G$ are eventually continuous, it is not difficult to show that $\tilde F$ and $\tilde G$ are eventually continuous as well.

The next result shows that the first part of Condition~\ref{C6} is rather easily satisfied. In particular, it holds true if $F$ and $G$ belong to the same max-domain of attraction. Moreover, some somewhat converse result can be stated, which is useful for the proofs of the results from the next section.
\begin{proposition}\label{P2} \;
\begin{enumerate} \item[(i)] Assume $F$ and $G$ belong to the same max-domain of attraction. Then $1-\tilde F$ is a regularly varying function with index $-\alpha_F \le 0.$ Moreover, if $\gamma_F\neq 0$ then $\alpha_F = \gamma_H/\gamma_F.$
\item[(ii)]
Assume $F$ and $G$ belong to two possibly different max-domains of attraction, $1 - \tilde F$ is a regularly varying function with index $-\alpha_F,$ and $\alpha_F>0.$ Then $F$ and $H$ belong to the same max-domain of attraction and, moreover, $\tau_F = \tau_H.$
\end{enumerate}
\end{proposition}

Finally, we introduce a second moment condition  {which combines} %similar to
Condition~\ref{envelope_moment} and \eqref{cond4}. Recall that %$\bar \vartheta$ denotes an envelope of $\vartheta,$ and 
$g_1, g_2$ are defined by Definition \ref{D5}. 

\begin{condition}\label{varthetaonlycondition}
There exists $\varepsilon>0$ such that
\begin{equation*} \int_{1}^\infty \big(\max(| \vartheta(g_1(x))|, |\vartheta(g_2(x))|)\big)^{2} x^{\tilde\alpha(\varepsilon)} \dd x < \infty,\end{equation*}
where $\tilde\alpha(\varepsilon) = \alpha_G - \alpha_F -1 +\varepsilon.$
\end{condition}

Denote 
\[S_\circ^f(\vartheta) := \int_1^\infty \vartheta(g(y)) \tilde F^\circ(\dd y),\] where similarly to Section~\ref{sec:ekmi}, $\tilde F^\circ(x)\!=\!(1\!-\!x^{-\alpha_F})I(x\!\ge\!1).$ If $\alpha_F\!=\!0$ then $\tilde F^\circ(x)$ should be understood as $0$ for all $x\ge 1,$ and thus $S_\circ^f(\vartheta)=\lim_{t\to\infty}\vartheta(g(t))$ if this limit exists. %Clearly, the latter case is not useful.

\begin{theorem} \label{T1genekmi}
Assume Condition~\ref{C6} and Condition~\ref{varthetaonlycondition}, that $F$ and $G$ are eventually continuous, that $\vartheta$ is continuous and eventually monotone, and let $k=k_n$ be an intermediate sequence. Then $S^f_{k,n}(\vartheta)$ is consistent for $S_\circ^f(\vartheta)$ and
\begin{equation}\label{T1genekmiassertion} \sqrt{k}\big(S^f_{k,n}(\vartheta) - \mu_f(U_H^{-1}(Z_{n-k,n}))\big) \stackrel{d}{\to} N\big(0, {\rm var} [\tilde W^\circ(\vartheta)]\big),\end{equation}
 where 
\[\mu_f(t) =  \int_1^\infty \vartheta(f(y,t)) \tilde F^t(\dd y).\]
\end{theorem}

To replace $\mu_f(\xi_{n-k,n})$ with $\mu_f(\infty) := S_\circ^f(\vartheta)$ in the latter theorem, we impose a second-order-type condition as was done in Theorem~\ref{T3an} and Theorem~\ref{T2gen}. However, in contrast to those theorems, the second-order behavior of both $f$ and $\tilde F^t$ impacts the asymptotics, and consequently we need to impose two second-order conditions. First, keeping the notation of Section~\ref{genresidual}, we impose Condition~\ref{second_order_2} on $f$. Similar to Theorem~\ref{T2gen}, we need to control the behavior of $\vartheta^\prime$ motivating the following condition.

\begin{condition}\label{varthetaprime} Let $\vartheta$ have 
a continuous and eventually monotone derivative $\vartheta^\prime$ such that for some $\varepsilon^\prime>0$
\begin{equation*}
    \label{onlyconditionnew2} \int_1^\infty \hat g(x) \max\big(|\vartheta^\prime(g_1(x))|, |\vartheta^\prime(g_2(x))|\big) \,x^{-\alpha_F -1 + \varepsilon^\prime} \dd x < \infty,
\end{equation*}
where $g_1,$ $g_2,$ and $\hat g$ are defined by Definition \ref{D5}.
\end{condition}

Secondly, similarly to Theorem~\ref{T3an}, we require a second-order condition on $U_{\tilde F},$ the analogue of Condition~\ref{secondorder}. Further we focus on the {non-degenerate} case $\alpha_F>0$.  {Denote $\vartheta_g = \vartheta \circ g.$}
%The following analogue of Condition~\ref{cond:phi_deriv}, handles the behavior of the derivative of the function $\vartheta_g = \vartheta \circ g.$

%\begin{condition}\label{cond:phi_deriv_2}
%Let $\vartheta_g$ have a derivative $\vartheta_g^\prime$ such that for some $\varepsilon^\prime>0$ 
%\begin{equation*}
%     \int_1^\infty \bar{\vartheta_g^\prime}(x)\,x^{-\alpha_F + \varepsilon^\prime} \dd x < \infty.
%\end{equation*}
%\end{condition}

\begin{theorem} \label{T2genekmi} Let the conditions of Theorem~\ref{T1genekmi}, Condition~\ref{second_order_2}, and Condition~\ref{varthetaprime} hold. % and Condition~\ref{cond:phi_deriv_2} hold. 
Assume Condition~\ref{secondorder} holds for $U_{\tilde F},$ $a = a_{\tilde F}$ and $\alpha_F>0$, that { $\sqrt{k} a(n/k)\to\lambda,$ $\sqrt{k}\hat a\big(U_H^{-1}(Z_{n-k,n})\big)\stackrel{P}{\to}\hat\lambda$} with finite $\lambda,$ $\hat \lambda,$ and %$\vartheta^\prime$ and 
$\vartheta_g^\prime$ is continuous and eventually monotone. %are %almost everywhere 
%continuous if $ \hat\lambda\neq0$ and $\lambda\neq0,$ respectively. 
Then
\begin{equation}\label{T2genekmiassertion} \sqrt{k}\big(S^f_{k,n}(\vartheta) - S_\circ^f(\vartheta)\big) \stackrel{d}{\to} N(\mu, {\rm var} [\tilde W^\circ(\vartheta)]\big),
\end{equation}
where \[\mu = \hat\lambda \int \tilde g(x) \vartheta^\prime(g(x)) \tilde F^\circ(\dd x) + \lambda\int x \vartheta_g^\prime(x) h_\rho(x^{\alpha_F})\tilde F^\circ(\dd x).\]
\end{theorem}

\begin{remark}
In most practical cases only one term in asymptotic bias may be positive, since the auxiliary functions $a$ and $\hat a$ generally have different asymptotic behavior. This is in line with classical results in statistics of extremes: depending on the extreme value index and the second-order parameter, the formula for the asymptotic bias changes substantially; see e.g. Theorem 3.5.4, \cite{dehaan}, and in particular formula (3.5.16) therein.
\end{remark}

\section{Extreme value index estimation}\label{sec:evi_est}

There are numerous extreme value index estimators proposed in the literature; see e.g. Section 6 in the review \cite{fedotenkov}. Residual estimators (see Section~\ref{genresidual} for a formal definition) are used in some of them. For instance, the moment estimator, \cite{moment}, is defined as follows
\begin{equation}g_{k,n} = M_{k,n}^{(1)} + 1 - \frac{1}{2}\left(1 - \frac{\big(M_{k,n}^{(1)}\big)^2}{M_{k,n}^{(2)}}\right)^{-1},\label{momentestimator}\end{equation} where
\[M_{k,n}^{(\ell)} = \frac{1}{k} \sum_{i=1}^k \big(\log X_{n-i+1,n} - \log X_{n-k,n}\big)^\ell = \frac{1}{k} \sum_{i=1}^k f_\ell(X_{n-i+1,n}/X_{n-k,n}), \:\:\: \ell =1,2,\]
are the residual estimators with $f_\ell(x) = (\log x)^\ell.$ The estimator $M_{k,n}^{(1)}$ was proposed in~\cite{hill1975simple} for positive tail indices, while $g_{k,n}$ targets all max-domains of attraction. Note that  {according to \eqref{newrkn},} $M_{k,n}^{(\ell)}$ can be rewritten as follows:
\begin{equation}\label{residualrepresentation} M_{k,n}^{(\ell)} = \int_{1}^\infty f_\ell(x) F_{k,n}(\dd x),\end{equation} where \[F_{k,n}(x) = \frac{1}{k} \sum_{i=1}^k I(X_{n-i+1,n}/X_{n-k,n} \le x)\] is the empirical distribution function of the set of the normalized largest order statistics $\{X_{n-i+1,n}/X_{n-k,n}\}_{i=1}^k.$ 

The aim of this section is to apply the ideas of previous sections for extreme value index estimation in the presence of random censorship. Specifically, we propose a version of the moment estimator \eqref{momentestimator}, where we use the extreme Kaplan--Meier estimator $\amsmathbb{F}_{k,n}$ in place of $F_{k,n}$. Thus, denote \begin{equation}\label{residualrepresentation_censored} \amsmathbb{M}_{k,n}^{(\ell)} = \int_{1}^\infty (\log x)^\ell\, \amsmathbb{F}_{k,n}(\dd x), \:\:\: \ell =1,2,\end{equation} and then define $\mathbb{g}_{k,n}$ by replacing $M_{k,n}^\ell$ with $\amsmathbb{M}_{k,n}^\ell,$ $\ell = 1,2,$ in \eqref{momentestimator}, that is
\begin{equation}\mathbb{g}_{k,n} = \amsmathbb{M}_{k,n}^{(1)} + 1 - \frac{1}{2}\left(1 - \frac{\big(\amsmathbb{M}_{k,n}^{(1)}\big)^2}{\amsmathbb{M}_{k,n}^{(2)}}\right)^{-1}.\label{momentestimator_censored}\end{equation}
 {Explicit expressions for $\amsmathbb{M}_{k,n}^{(1)}, \amsmathbb{M}_{k,n}^{(2)}$ and consequently $\mathbb{g}_{k,n}$ can be derived from \eqref{explicit}.}

In what follows we show the asymptotic properties of the resulting estimator $\mathbb{g}_{k,n}$ taking full advantage of the theory of extreme Kaplan--Meier integrals.

As in Condition~\ref{C6}, we assume that $1 - F(U_H(x))$ and therefore $1-G(U_H(x))$ are regularly varying functions with indices $-\alpha_F$ and $-\alpha_G,$ respectively. Recall that by definition of $F,$ $G$ and $H,$ $\alpha_F,\alpha_G \in [0,1]$. %Example~\ref{Ex2} 
Proposition~\ref{P2} (i) implies that this condition includes  {all three cases} considered by \cite{einmahl2008statistics}:
\begin{enumerate}
\item[Case 1.] $\gamma_F>0,$ $\gamma_G>0;$
\item[Case 2.] $\gamma_F = \gamma_G = 0$ and $\tau_F = \tau_G = \infty;$
\item[Case 3.] $\gamma_F<0,$ $\gamma_G<0$ and $\tau_F = \tau_G.$
\end{enumerate}
In~\cite{einmahl2008statistics} other possibilities are not explored due to being close to either the ``uncensored case'' (when $\gamma_F<0$ and $\gamma_G>0$) or to the ``completely censored situation'' (when $\gamma_F>0$ and $\gamma_G<0$). However, the situation $\gamma_F=\gamma_G=0$ and $\tau_F = \tau_G<\infty$ does not fit into the above cases and is interesting.  {More precisely, the following situations are not included in Cases 1-3 and possible in our setting, albeit perhaps not as common in practice:
\begin{itemize}
\item $\gamma_F \le 0, \gamma_G > 0;$
\item $\gamma_F=0,$ $\gamma_G<0$ and $\tau_F\le\tau_G;$
\item $\gamma_F<0,$ $\gamma_G\le 0$ and $\tau_F<\tau_G;$
\item $\gamma_F = \gamma_G = 0$ and $\tau_F<\tau_G\le\infty.$
\end{itemize}
}\noindent
%Consequently, we continue with our own general condition.

The following result is the direct consequence of Theorem~\ref{Tconsistency} and the first statement of Theorem~\ref{T1genekmi}.

\begin{theorem}\label{Tmomentconsistency}
Assume $F$ and $G$ are eventually continuous and belong to (possibly different) max-domains of attraction such that $1 - F(U_H(x))$ is regularly varying function with index $-\alpha_F,$ where $\alpha_F>1/2$. Additionally, assume $\tau_F>0.$ Let $k=k_n$ be an intermediate sequence. Then \begin{align*}
\mathbb{g}_{k,n}\stackrel{P}{\to}\gamma_F.
\end{align*}
\end{theorem}

%Notice that the positiveness of $\alpha_F$ in the latter theorem is crucial since in case of zero $\alpha_F$ the limit in probability of $\mathbb{g}_{k,n}$ does not depend on $\gamma_F.$
Notice that the assumption $\alpha_F>1/2$ ensures that $F$ and $H$ belong to the same maximum domain of attraction by Proposition~\ref{P2} (ii). 

If both $F$ and $G$ belong to the Fr\'echet or Weibull max-domains of attraction, we get by Proposition~\ref{P2} (i) %Example~\ref{Ex2} 
\[\alpha_F = \frac{\gamma_H}{\gamma_F} = \frac{\gamma_G}{\gamma_F+\gamma_G},\] hence for positive $\gamma_F$ and $\gamma_G$ the condition $\alpha_F>1/2$ is equivalent to the natural assumption $1/\gamma_F>1/\gamma_G$ (and $1/\gamma_F<1/\gamma_G$ for negative $\gamma_F$ and $\gamma_G$) appearing in~\cite{bww}, that is the tail of $F$ is lighter than the tail of $G.$

Now, let us consider the asymptotic normality of the estimator $\mathbb{g}_{k,n}.$ This result can be derived immediately applying Theorem~\ref{T2genekmi}.
%; see Corollary~\ref{C5.3} below. 
However, Theorem~\ref{T2genekmi} requires imposing the second-order conditions both on $U_H$ and $U_{\tilde F}$ which can be avoided for $\gamma_F\neq 0.$ For positive $\gamma_F$ imposing Condition~\ref{secondorder} is enough; for negative $\gamma_F,$
%The proof of this property is based on Theorem~\ref{T3an} (for $\gamma_F>0$) and Theorem~\ref{T1genekmi} (for $\gamma_F\le0$). We can derive this result immediately applying Theorem~\ref{T2genekmi}; see Corollary~\ref{C5.3} below.  {However, we first dispose of some conditions imposed in Theorem~\ref{T2genekmi} in the special case we consider here, see Theorem~\ref{Tmomentan}.} Note that Theorem~\ref{T2genekmi} requires imposing the second order conditions both on $U_H$ and $U_{\tilde F}.$ Our proof is slightly longer since we also show that it is possible to avoid such restriction for $\gamma_F\neq 0.$ For positive $\gamma_F,$ we prove the asymptotic normality of the estimator $\mathbb{g}_{k,n}$ using the results of Section~\ref{sec:ekmi}. 
%For negative $\gamma_F,$  using the regular variation property of $\tau_F - U_F,$ we assume the analogue of the second order condition \eqref{standardsecondorder} with $a(t) = -\gamma_F(\tau_F - U_F(t)).$
based on Remark B.3.15 in~\cite{dehaan}, 
we assume the following condition
\begin{equation}\label{secondordernegative}\lim_{t\to\infty}\frac{\frac{\tau_F - U_F(tx)}{\tau_F - U_F(t)} - x^{\gamma_F}}{A(t)} = x^{\gamma_F} h_\rho(x)\end{equation} for some $\rho\le0,$ and $A(t)$ positive or negative, regularly varying with index $\rho$ and satisfying $A(t)\to 0$ as $t\to\infty,$ which is the analogue of Condition~\ref{secondorder} for negative $\gamma_F.$

Recall the notation $c_n = (1 - F(U_H(n/k)))^{-1}.$
\begin{theorem}\label{Tmomentan} Assume the conditions of Theorem~\ref{Tmomentconsistency}.
\begin{enumerate}
\item[(i)] Let $\gamma_F\!>\!0,$ $U_F$ satisfy the second-order Condition~\ref{secondorder} and $\sqrt{k}\,a(c_n)\!\to\!\lambda$ with finite $\lambda.$ Then
\begin{equation}\label{lastofthelast}\sqrt{k}(\mathbb{g}_{k,n} - \gamma_F) \xrightarrow{d} N(b, \sigma^2),\end{equation} where
\begin{equation}\label{biasvariance}b = \lambda \frac{\gamma_F - \gamma_F \rho + \rho}{\gamma_F(1-\rho)^2}, \quad \sigma^2 = \frac{\gamma^3_G + \gamma_G\gamma_F^2(\gamma_G - \gamma_F+1)^2}{(\gamma_G - \gamma_F)^3}.\end{equation}
\item[(ii)] Let $\gamma_F = 0,$  $U_H$ and $U_{\tilde F}$ satisfy \eqref{secondorderlog} and Condition~\ref{secondorder} with auxiliary functions $Q,$ $\tilde a$ and second-order parameters $\rho^\prime,$ $\tilde \rho,$ respectively. Assume $\sqrt{k}Q(n/k) \to\lambda$ and $\sqrt{k}\tilde a(n/k) \to\tilde \lambda$ with finite $\lambda, \tilde \lambda.$ Then \eqref{lastofthelast} holds with
\[b= \lambda \frac{1+\alpha_F}{\alpha_F^2} + \tilde\lambda \frac{1 -\tilde\rho+ \alpha_F \tilde \rho}{(1-\tilde \rho)^2}, \quad \sigma^2 = \frac{\alpha_F(\alpha_F^2 + \alpha_G^2)}{(\alpha_F - \alpha_G)^3}. %\frac{\alpha^4_F + (\alpha_F - \alpha_G+\alpha_F\alpha_G)^2 I(\alpha_G>0)}{\alpha_F(\alpha_F-\alpha_G)^3}.
\]
\item[(iii)] Let $\gamma_F<0,$ $U_F$ satisfy the second-order condition \eqref{secondordernegative}, $\sqrt{k}A(c_n) \to \lambda$ and $\sqrt{k}(1 - U_H(n/k)/\tau_H) \to \hat \lambda$ with finite $\lambda, \hat \lambda.$  Then \eqref{lastofthelast} holds with
\begin{eqnarray*} b &=& \lambda \frac{(1-\gamma_F)(1-2\gamma_F)}{1-\gamma_F-\rho}\left(\frac{2-3\gamma_F-\rho}{1-2\gamma_F-\rho}-\frac{2}{\gamma_F} + \frac{1}{\gamma_H(1-\gamma_F)^2(1-2\gamma_F)}\right)\\ &&+ \hat \lambda \left(\frac{(1-\gamma_F)(1-\gamma_F-\gamma_F^2)}{1 - 3\gamma_F} - \frac{\gamma_F}{2\gamma_H(1 - \gamma_F)(1 - 2\gamma_F)}\right),
\end{eqnarray*}
and $\sigma^2$ defined by \eqref{sigma}.
%\[\sigma^2 = \frac{\gamma_G(1 - \gamma_F)^2(1 - 2\gamma_F^2)(\gamma_G^2 + \gamma_F^2 - \gamma_G^2\gamma_F + 5\gamma_G\gamma_F^2 + 6\gamma_G^2\gamma_F^2)}{\prod_{j=2}^4(\gamma_G-\gamma_F - j\gamma_G\gamma_F)}.\]
\end{enumerate}
\end{theorem}

%Note that it is possible to prove the assertion $(i)$ of the latter theorem by imposing the second order condition \eqref{standardsecondorder} instead of Condition~\ref{secondorder}, as it was done, e.g., in Theorem 3.5.4, \cite{dehaan}, proving asymptotic normality of the moment estimator in the uncensored case. We avoid taking such a route to highlight that in case $\gamma_F>0$ the proof quite simply follows from Theorem~\ref{T3an}.
The result above is in agreement with the asymptotic normality result of the moment estimator in the uncensored case, see, e.g., Theorem 3.5.4, \cite{dehaan}. Indeed, setting $\gamma_G = \infty,$ $\gamma_G=-\infty$ and $\alpha_G=0$ for positive, negative and zero $\gamma_F,$ respectively (corresponding to the non-censored case), the asymptotic variances in Theorem~\ref{Tmomentan} coincide with all the corresponding asymptotic variances in Theorem 3.5.4, \cite{dehaan}. The slight difference in asymptotic biases between Theorem~\ref{Tmomentan} $(i)$ and the corresponding case in Theorem 3.5.4, \cite{dehaan}, is caused by different second-order conditions and in particular different auxiliary functions.

%\begin{remark}[Comparison with the moment estimator adapted for censoring]\rm
%The moment estimator adapted for censoring presented in~\cite{einmahl2008statistics} is based on the simple observation that  {the censoring proportion} $p = \gamma_H/\gamma_F$ and thus one can estimate $\gamma_H$ and $p$ to estimate $\gamma_F$ indirectly, and their conditions are of very different nature to the ones we impose. The asymptotic variance of their estimator is usually smaller than the asymptotic variance of $\mathbb{g}_{k,n}$. However, the method proposed in~\cite{einmahl2008statistics} works exclusively for the estimation of extreme value index. Moreover, the conditions imposed in~\cite{einmahl2008statistics} on the (uniformity and speed of) convergence of $p(z) = \pr(\delta=1|Z=z)$ to $p$ (whose definition coincides with our $\alpha_F$ under differentiability of $F$ and $G$ -- a condition we do not require) are rather technical. Our conditions offer clear alternatives.

The condition $\gamma_G>\gamma_F$ imposed in Theorem~\ref{Tmomentconsistency} and Theorem~\ref{Tmomentan} {is crucial, and we can only hope to remove it for consistency using other proof techniques than the ones presented in this article. An alternative is to construct a moment estimator for the censoring mechanism and for the entire sample, and then using $\gamma_F=(\frac{1}{\gamma_H}-\frac{1}{\gamma_G})^{-1}$ to provide a finite asymptotic variance estimator of $\gamma_F$. These considerations are still an open and important research topic.

The results of Section \ref{sec:exten} can be used to advance other methods, apart from extreme value index estimation presented in this section. For instance, the scale estimator provided in Section 4.2, \cite{dehaan}, can be adopted in censored settings with the use of our methodology. This results in the method of extreme quantile and tail estimation valid for all max-domains of attraction, which is the subject of our future work.

\section{Finite-sample behavior}\label{sec:sim}
{ We provide three simulation studies related to the performance of estimators derived as EKM integrals. The first relates to the positive extreme value index setting from Section~\ref{sec:ekmi}; the second relates to the proposed moment estimator of Section~\ref{sec:evi_est}; the third one explores, for the latter estimator, the quality of the approximate asymptotic distribution for finite samples.}
We also highlight that the scope of applications of our methodology is much broader, and the aim of the simulations below is to illustrate the finite-sample validity of our construction.

\subsection{Simulation study for $\gamma_F>0$}
{ 
In this section, we demonstrate the performance of estimators of Section~\ref{sec:ekmi} under several simulation scenarios, and also compare them to a benchmark estimator. 

The targets are the first three log-moments of the limit event tail distribution $F^\circ$, which are respectively given by $m_1=\gamma_F$, $m_2=2\gamma_F^2$, and $m_3=6\gamma_F^3$. We also consider a plug-in estimator using an estimator of $\gamma_F$ from our methodology. The benchmark estimator is a plug-in estimator derived in \cite{einmahl2008statistics}. Specifically, we consider two different distributions: the Burr and Fr\'echet distributions, with parameters chosen to achieve degrees of non-censoring of 60\%, 75\%, and 90\%, respectively. The event distribution extreme value index is fixed at \(\gamma_F = 0.5\).

The nine estimators considered are defined as follows, for $r=1,2,3$:
\begin{align*}
\widehat{m}^r_{k} =  S_{k,n}(\log^r)/S_{k,n}(1),\quad \widehat{m}^{KM,r}_{k} = r! (\widehat\gamma^{c,Hill}_{k}/S_{k,n}(1))^r, \quad\widehat{m}^{B,r}_{k} &= r! (\widehat\gamma^{B,Hill}_k)^r,
\end{align*}
where $\widehat\gamma^{c,Hill}_{k}$ is given in \eqref{cHill_def} and $\widehat\gamma^{B,Hill}_k$ is the Hill-type estimator of \cite{einmahl2008statistics}. Recall that $S_{k,n}(1) = \amsmathbb{F}_{k,n}(\infty)$ is an asymptotically negligible normalizing factor, see Remark \ref{Rnormalization}. {Note that since our construction of the Hill estimator in \eqref{cHill_def} is identical to the one introduced in equation (7) of \cite{worms2014new}, we do not focus on the numerical comparison of \eqref{cHill_def} with the other estimator proposed in \cite{worms2014new} and the one from \cite{einmahl2008statistics} in this simulation study; this comparison can be found in the former reference.}

\begin{figure}[htbp]
    \centering
    \includegraphics[width=1\textwidth,trim={0cm 0cm 0cm 0cm},clip]{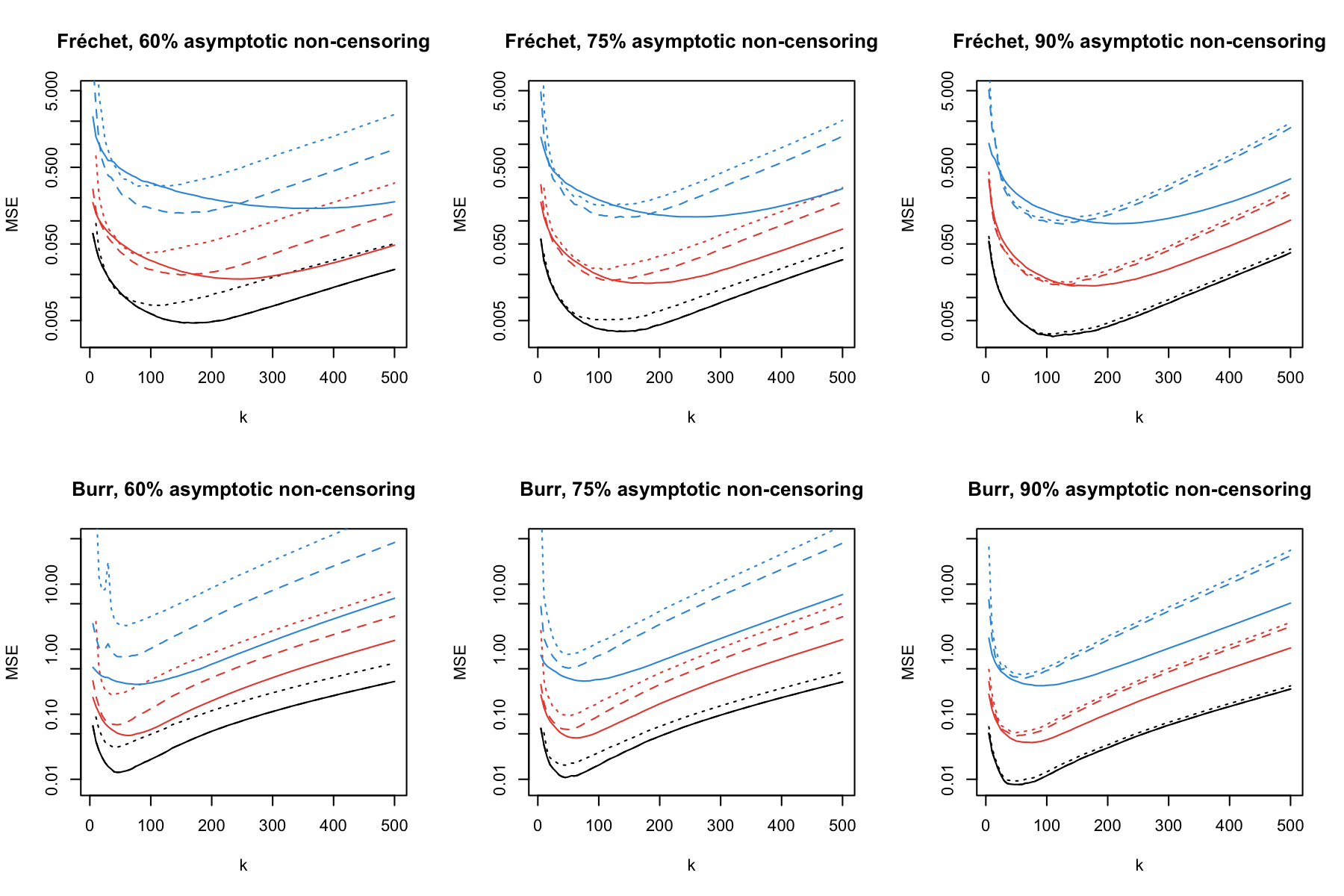}
    \caption{{ Mean Square Error (MSE) of the estimators $\widehat{m}^r_{k}$ (solid), $\widehat{m}^{KM,r}_{k}$ (dashed), and  $\;\widehat{m}^{B,r}_{k}$ (dotted) as a function of $k$, for $n=1000$ and across $1000$ simulations. We consider $k\in\{5,\dots,n/2\}$ and $r=1,2,3$ in black, red and blue, respectively.}} 
    \label{fig:sim1}
\end{figure}

The results of the study for $n=1000$ and $k\in\{5,\dots,n/2\}$ are depicted in Figure \ref{fig:sim1}, where we plot the Mean Square Error (MSE) of the estimators as a function of $k$, averaged across $1000$ simulations. For the Burr and Fr\'echet distributions with a 60\% and 75\% degree of non-censoring, we observe that our proposed estimators perform particularly well. The lowest MSE occurs at different $k$ for $\widehat{m}^r_{k}$ and $\widehat{m}^{KM,r}_{k}$, though both manifest a similar global minimum.  The behavior of $\widehat{m}^{B,r}_{k}$ is similar to the latter, though consistently worse. Thus, the finite-sample performance of EKM integrals seems to be rather favorable. The asymptotic variance of $\widehat\gamma^{c,Hill}_{k}$ is larger than that of $\widehat\gamma^{B,Hill}_k$, so the finite-sample bias is the main driver for the lower MSE.
At a 90\% degree of non-censoring, we observe the gap close between all estimators, as expected.
}

\subsection{Simulation study for $\gamma_F\in \amsmathbb{R}$}
{ 

In this section, we examine the performance of two different estimators: the moment estimator from Section~\ref{sec:evi_est}, namely $\mathbb{g}_{k,n}$, and the benchmark moment estimator from \cite{einmahl2008statistics}, which we denote $\widehat{\gamma}^M_{k}$. We assess their performance across three distributions belonging to different max-domains of attraction and for two different sample sizes, $n=10^3, \,10^4$. The number of simulations is again $1000$. The metrics used for evaluation are the MSE and the empirical probability of correct classification into max-domains of attraction. \footnote{{ The classification metric is equal to one if the estimator has the correct sign, and zero otherwise. In case $\gamma_F=0$, we define correct classification as being in a neighborhood of zero, with radius $1/5$.}} To build the moment estimator we also use the normalized empirical moments $\widehat m^r_k$ instead of $\amsmathbb{M}^{(r)}_{k,n}$ as in the previous simulation study.

The distributions considered in our simulation study are:

\begin{itemize}
    \item[a)] $F,G\sim\mbox{Beta}(1,2)$; $\gamma_F=-1/2$.
    \item[b)] $F\sim\mbox{Burr}(\sqrt{2},\sqrt{2})$ and $G\sim\mbox{Burr}(1/\sqrt{3},1/\sqrt{3})$; $\gamma_F=1/2$. 
    \item[c)]  $F\sim\mbox{Weibull}(1,1)$ and $G\sim\mbox{Weibull}(1/2,1)$; $\gamma_F=0$.
\end{itemize}

\begin{figure}[htbp]
    \centering
    \includegraphics[width=1\textwidth,trim={0cm 0cm 0cm 0cm},clip]{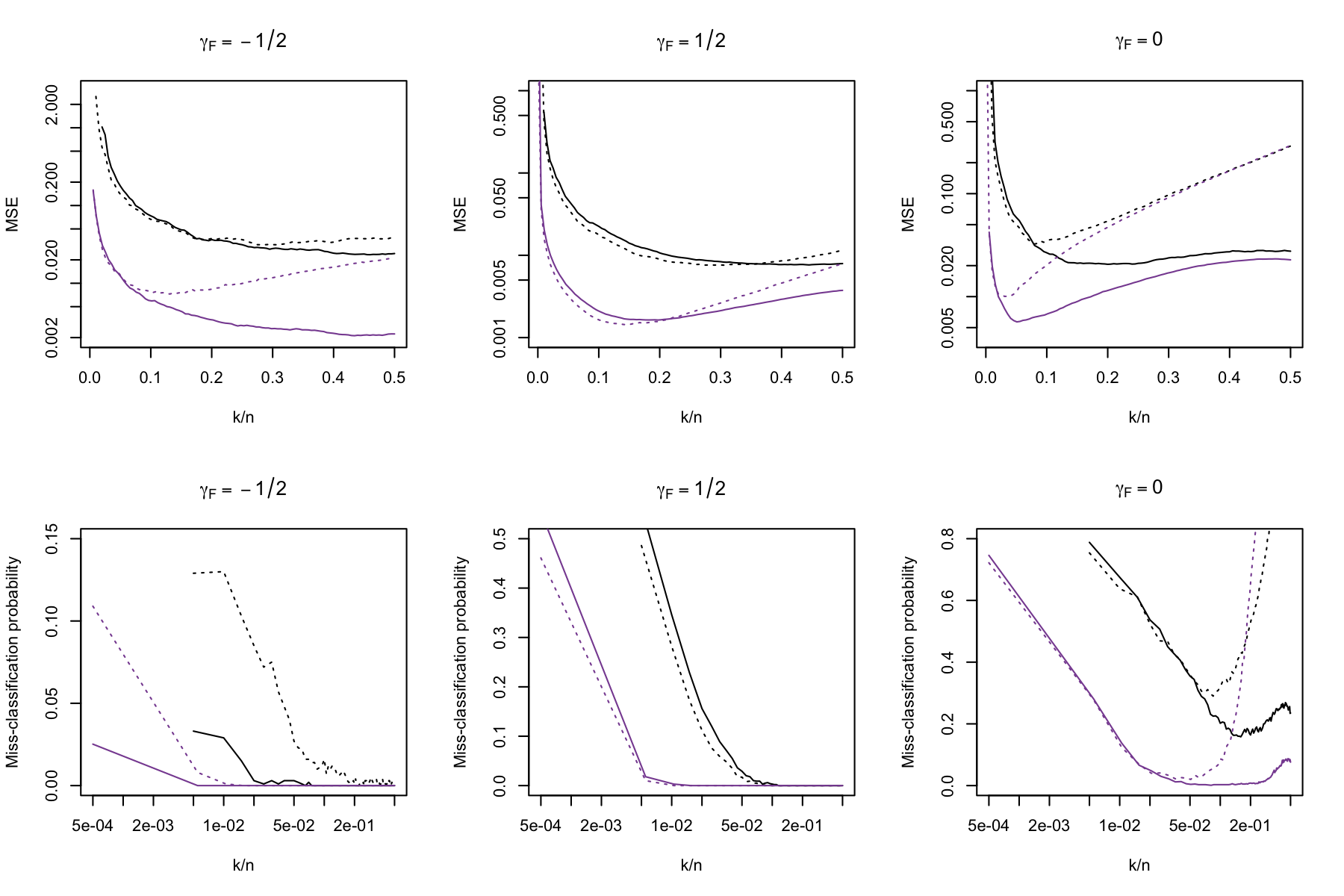}
    \caption{{ Top panels: Mean Square Error (MSE)  of the estimators $\mathbb{g}_{k,n}$ (solid) and $\widehat{\gamma}^M_{k}$ (dashed) as a function of $k/n$, for $n=10^3,\,10^4$ (black and purple, respectively) and across $1000$ simulations. Bottom panels: corresponding miss-classification rates.}} 
    \label{fig:sim2}
\end{figure}

The results of the study as a function of $k/n$ are depicted in Figure \ref{fig:sim2}, where we plot the MSE and the miss-classification rate (one minus the correct classification rate), averaged across simulations. For the Beta and Weibull distributions we observe that our proposed estimator is particularly well; the MSE dominates the benchmark from below, and the global minimum is lower. The miss-classification probability is also provided in Figure \ref{fig:sim2} and follows a similar pattern. However, for the Burr distribution, we observe that the behavior is reversed, and the benchmark is favoured. The results are stable over the two different sample sizes.
}

\subsection{Finite versus asymptotic behavior for $\gamma_F\in\amsmathbb{R}$}
{ 

In this section, we examine the finite-sample behavior of the estimator $\mathbb{g}_{k,n}$ from Section~\ref{sec:evi_est}. In particular, we are interested to see how the asymptotic variance relates to finite-sample coverage probabilities, which ultimately is important for statistical inference.%\footnote{{ Namely, we compute the relative frequency with which $\mathbb{g}_{k,n}-\gamma_F$ falls outside of the $95\%$ confidence interval from its asymptotic normal approximation.}} This footnote is wrong!
The number of simulations is again $1000$. The distributions considered in our simulation study are:
\begin{itemize}
    \item[a)] $F,G\sim \mbox{Beta}(1,2);$ $\gamma_F=-1/2$.
    \item[b)] $F\sim\mbox{Pareto}(2)$ and $G\sim\mbox{Pareto}(2/3)$; $\gamma_F=1/2$.
    \item[c)]  $F\sim\mbox{Exp}(6)$ and $G\sim\mbox{Exp}(1)$; $\gamma_F=0$, $\alpha_F=6/7$.
\end{itemize}

\begin{figure}[htbp]
    \centering
    \includegraphics[width=1\textwidth,trim={0cm 0cm 0cm 0cm},clip]{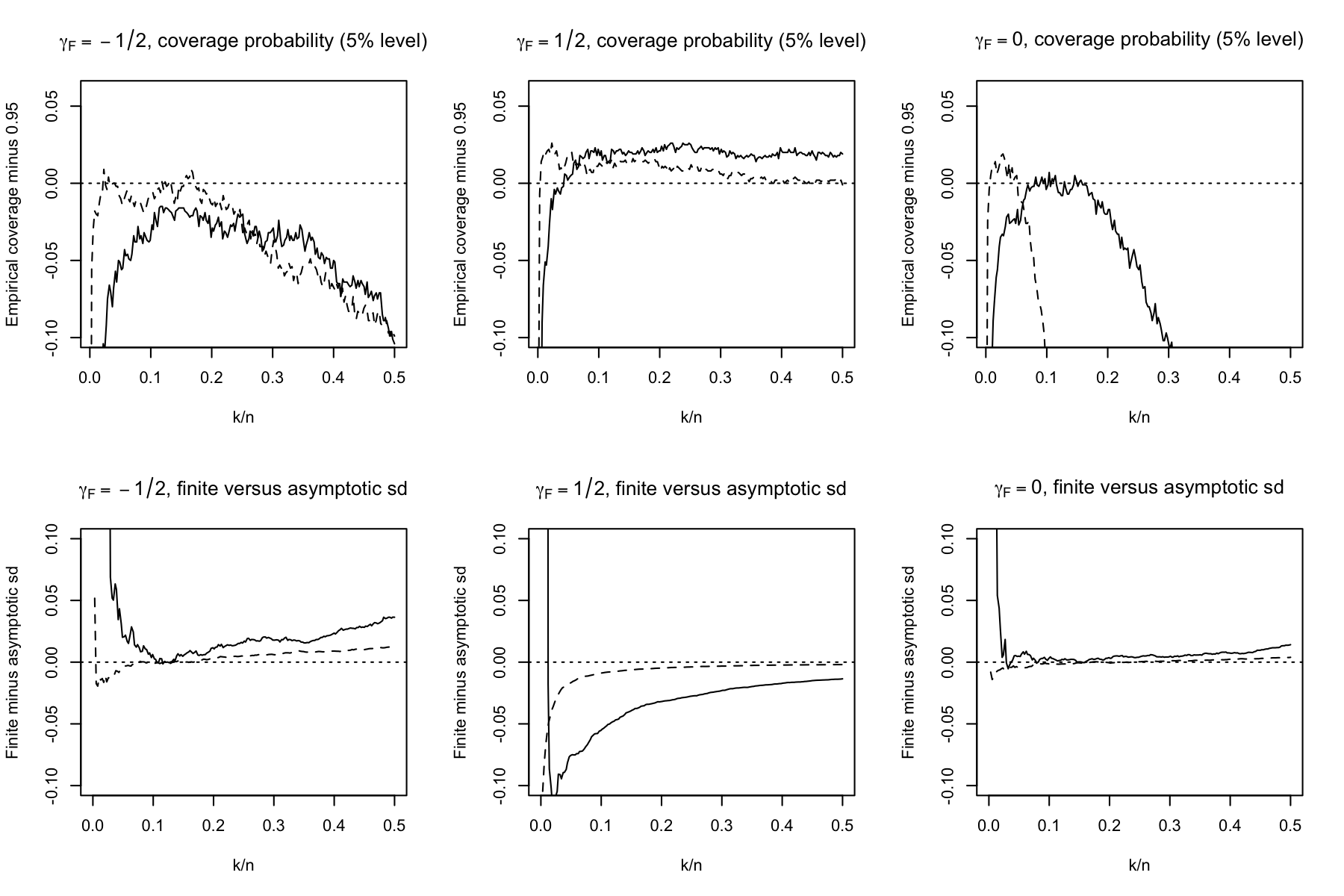}
    \caption{{ Top panels: difference between empirical and desired coverage probabilities of the estimator $\mathbb{g}_{k,n}$, for three distributions in different max-domains of attraction, sample sizes $n=10^3,\,10^4$ (solid and dashed, respectively), as a function of $k/n$, and across $1000$ simulations. Bottom panels: corresponding difference between finite-sample and asymptotic standard deviations.}} 
    \label{fig:sim3}
\end{figure}

The results of the study as a function of $k/n$ are provided in Figure \ref{fig:sim3}, where we plot the difference between empirical and asymptotic coverage probabilities, where the former are averaged across simulations. For the Weibull and Fr\'echet max-domains of attraction, we observe that the coverage levels are kept rather closely, and get better with sample size increasing, as expected. As these distributions are the archetypes of their max-domains of attraction, good coverage can be seen to extend even for larger $k$. This is in contrast with the exponential distribution, which exhibits significantly reduced coverage for all but the top $k$. For reference, the difference  between the finite-sample standard deviation and its asymptotic approximation (which is of order $1/\sqrt{k}$) is provided in the bottom panels of Figure \ref{fig:sim3}.

}
\section{Brain cancer dataset}\label{sec:real_data}

The analysis examines survival data from a clinical study, which includes significant censoring (40\% right-censoring). {The dataset\footnote{The dataset is available at \url{https://portal.gdc.cancer.gov/}, under the \emph{Brain} category, and then through the \emph{Clinical Data Analysis} tool.} consists of 1342 entries corresponding to time-to-death observations from male and female patients of brain cancer.} We compare the performance of the proposed moment estimator based on extreme Kaplan--Meier integrals with the established moment estimator of \cite{einmahl2008statistics} in estimating the extreme value index under random censoring, in this case our estimator suggests a negative value. We note that the estimator of \cite{einmahl2008statistics} and the usual moment estimator \cite{moment} have by construction the same sign throughout all values of $k$, despite the latter not taking censoring effects into account. Thus, from a classification point of view, our estimator provides an alternative which does incorporate censoring effects. %, especially when the censoring distribution is lighter than the target distribution. But it contradicts our assumption \alpha_F>\alpha_G!

Our estimator and the estimator of \cite{einmahl2008statistics} are computed over a range of thresholds representing different proportions of extreme observations, as can be seen in Figure~\ref{fig:bmb}. For this single sample, the proposed estimator shows stability across thresholds, except very high ones with $k\le 40.$ The estimator of \cite{einmahl2008statistics} does not stabilize at larger quantiles, and it has a shorter region below zero. However, both estimators stay below zero for high thresholds, with evidence that the underlying distribution belongs to the Weibull MDA.%, in particular having a finite endpoint.

While sensitivity to threshold selection is a challenge for using any of the two approaches as an estimation or classification tool, such selection for the proposed estimator in this case is easier, with the plot for $k \le 100$ being more conclusive for $\mathbb{g}_{k,n}$. This suggests improved robustness, particularly in handling the effects of censoring in the tail, since both estimators reduce to the usual moment estimator of \cite{moment} when no censoring is present. This robustness for the Weibull max-domain of attraction is also observed in the second simulation study above.
Thus, we highlight our estimator as a practical alternative, which in some cases provides stability and robustness, for real-world extreme value analysis under censoring. 

We finally notice that the proportion of non-censored observations is above $50\%$ only when considering the very top observations. Although our theoretical condition $\gamma_F>\gamma_G$ is only necessary for convergence of $\mathbb{g}_{k,n}$ as $k/n\to0$, in finite samples the estimator seems to be quite stable even when the proportion of non-censored observations dips below $50\%$. We believe that this restriction is not necessary to show consistency, and it is the subject of our future research.

\begin{figure}[htbp]
    \centering
    \includegraphics[width=1\textwidth,trim={0cm 0cm 0cm 0cm},clip]{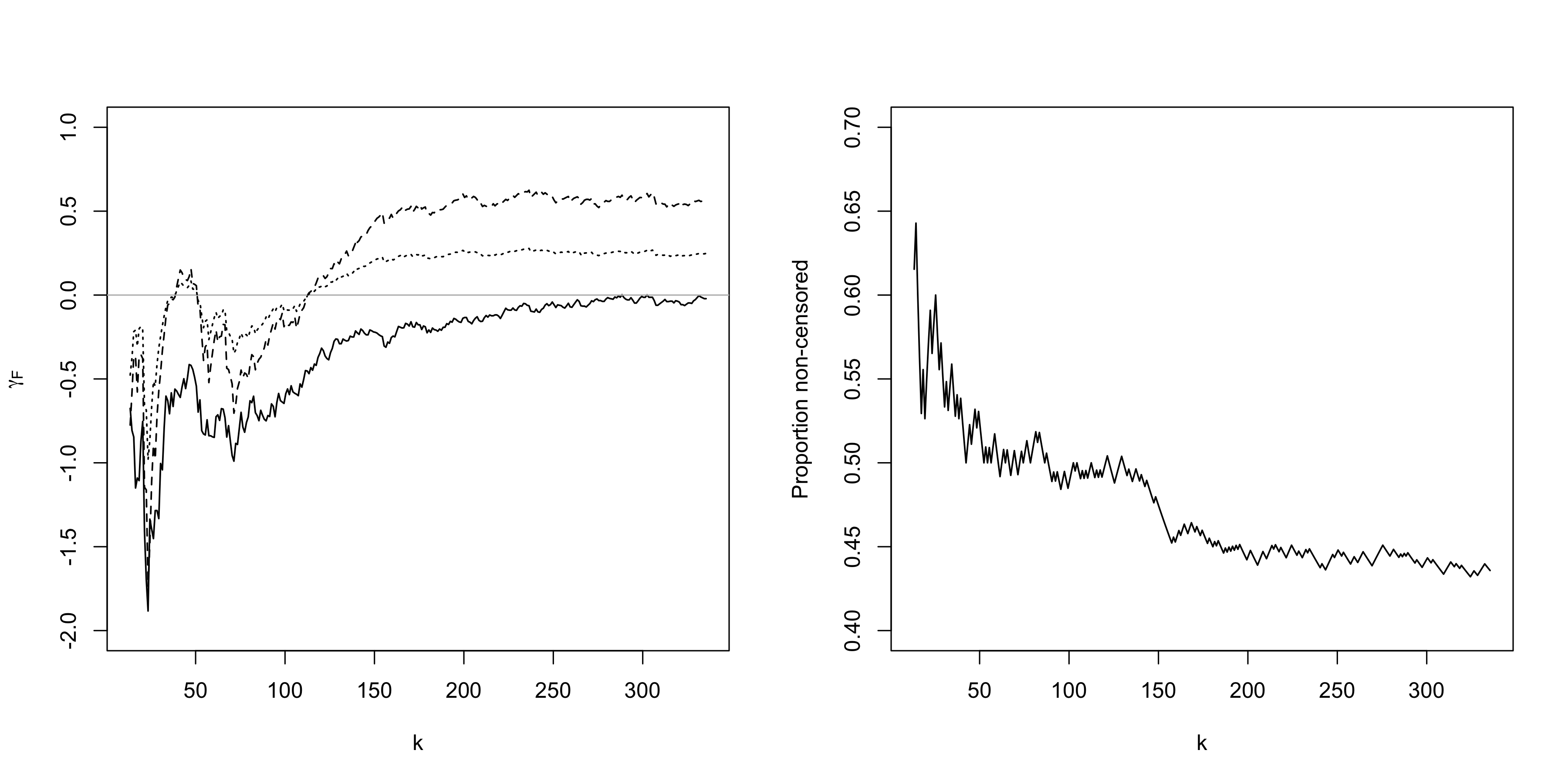}
    \caption{Brain cancer dataset. Left panel: estimators $\mathbb{g}_{k,n}$ (solid) and $\widehat{\gamma}^M_{k}$ (dashed) as a function of $k$, where the sample size is $n=1342$. For reference, we also include the usual moment estimator (dotted), given by either of the two previous estimators when setting all censoring indicators to one. Right panel: proportion of non-censored observations given by $(1/k)\sum_{i=1}^k \delta_{[n-i+1:n]}$.} 
    \label{fig:bmb}
\end{figure}

\newpage

%%%%%%%%%%%%%%%%%%%%%%%%%%%%%%%%%%%%%%%%%%%%%%
%% Example with single Appendix:            %%
%%%%%%%%%%%%%%%%%%%%%%%%%%%%%%%%%%%%%%%%%%%%%%
\begin{appendix}

%\section{Technical definitions}\label{AppDef} 
\end{appendix}
\newpage 
%% if your bibliography is in bibtex format, uncomment commands:
\bibliographystyle{imsart-number} % Style BST file (imsart-number.bst or imsart-nameyear.bst)
\bibliography{main.bib}       % Bibliography file (usually '*.bib')

%% or include bibliography directly:
%\begin{thebibliography}{4}
%%%
%\bibitem{r1}
%\textsc{Billingsley, P.} (1999). \textit{Convergence of
%Probability Measures}, 2nd ed.
%Wiley, New York.
%
%\bibitem{r2}
%\textsc{Bourbaki, N.}  (1966). \textit{General Topology}  \textbf{1}.
%Addison--Wesley, Reading, MA.
%
%\bibitem{r3}
%\textsc{Ethier, S. N.} and \textsc{Kurtz, T. G.} (1985).
%\textit{Markov Processes: Characterization and Convergence}.
%Wiley, New York.
%
%\bibitem{r4}
%\textsc{Prokhorov, Yu.} (1956).
%Convergence of random processes and limit theorems in probability
%theory. \textit{Theory  Probab.  Appl.}
%\textbf{1} 157--214.
%\end{thebibliography}
 
\newpage
\setcounter{page}{1}
\pagenumbering{roman}
\thispagestyle{empty}
\medskip
\centerline{\large SUPPLEMENTARY MATERIAL: {CENSORED EXTREME VALUE ESTIMATION}}
\medskip
\centerline{BY MARTIN BLADT \& IGOR RODIONOV}

\section{{  Applications of generalized residual estimators}} \label{Sectionexamples}

In this section, we present some applications of Theorem~\ref{T1gen} and Theorem~\ref{T2gen}. The most natural candidate function is \[f(x,t) = \frac{U_F(xt)}{U_F(t)}\] which appears in Theorem~\ref{T1} and Lemma~\ref{L1_bis}. 

Thus, let $F\in\mathcal{D}(G_\gamma)$ and assume %for convenience 
in this section
that $\tau_F>0.$ By the remark below Lemma~\ref{L1_bis}, $g(x) = x^{\gamma_+},$ where $\gamma_+ = \max(\gamma, 0)$. Moreover, by Potter's bounds, for arbitrary $\delta_1, \delta_2>0$, there exists $T = T(\delta_1,\delta_2)$ such that for $t\ge T$
\[g_1(x): = \max\big((1-\delta_1) x^{\gamma_+ - \delta_2}, 1\big) \le f(x,t) \le (1+\delta_1) x^{\gamma_+ + \delta_2} =: g_2(x),\] where we select $\delta_2<\gamma$ for positive $\gamma.$ 

Next, if $\gamma>0$ then we can use the second-order Condition~\ref{secondorder}. To derive an analogous condition in case $\gamma\le 0,$ recall that $F\in \mathcal{D}(G_\gamma)$ is equivalent to
\begin{equation}\label{basic}\lim_{t\to\infty}\frac{U_F(tx) - U_F(t)}{a(t)} = h_\gamma(x), \quad x>0,\end{equation} for some positive function $a,$ see, e.g., Theorem 1.1.6 in~\cite{dehaan}. Thus, for $\gamma\le 0$ and $x>0$
\begin{equation}\label{fakesecondorder} \lim_{t\to\infty} \frac{U_F(tx)/U_F(t) - x^{\gamma_+}}{a(t)/U_F(t)} = h_\gamma(x),\end{equation} hence in \eqref{cond5} we can set \begin{equation}\label{tildeg}\tilde g(x) = x^{\gamma_+} h_{\hat\rho}(x)\end{equation} with $\hat \rho = \rho I(\gamma>0) + \gamma I(\gamma\le 0)$ and $\hat a$ equal to $a$ in Condition~\ref{secondorder} for $\gamma>0$ and $a/U$ in \eqref{fakesecondorder} for $\gamma\le 0,$ respectively. Next, the conditions of Lemma 4.6, \cite{segers}, are satisfied under Condition~\ref{secondorder} for $\gamma>0$ and \eqref{fakesecondorder} for $\gamma\le 0,$ respectively, and thus in \eqref{cond6} we can set $\hat g(x) = K x^{\gamma_+ + \delta}$ for some $K>0,$ $\delta>0$ and $t\ge T.$ 
Finally, since $\hat a$ is regularly varying in our case, it suffices to require $\sqrt{k} \hat a(n/k) \to \lambda$ instead of $\sqrt{k} \hat a(\xi_{n-k,n}) \stackrel{P}{\to}\lambda$ by Smirnov's lemma (see, e.g., Lemma 2.2.3 in~\cite{dehaan} or \eqref{smirnov}).

The following result is an immediate consequence of Theorem~\ref{T1gen} and Theorem~\ref{T2gen}.

\begin{corollary}\label{Cor1}\, \begin{enumerate} \item[(i)] Let $F\in \mathcal{D}(G_\gamma)$ with $\gamma\in \amsmathbb{R}$ and $\tau_F>0$, and continuous and eventually monotone $\vartheta : [1, \infty) \to \amsmathbb{R}$ satisfy $\amsmathbb{E}\big[|\vartheta (\xi^{\gamma_+ + \delta})|^{1+\delta}\big] <\infty$ for some $\delta>0,$ 
and $\xi\sim \text{Pareto}(1).$ Assume $\mu(t),$ $k$ and $\xi_1, \xi_2, \ldots$ are as in Theorem~\ref{T1}. Then $R_{k,n}$ is consistent for $\amsmathbb{E}\big[\vartheta (\xi^{\gamma_+})\big].$ If, moreover, $\amsmathbb{E}\big[|\vartheta (\xi^{\gamma_+ + \delta})|^{2+\delta}\big] <\infty$ for some $\delta>0,$ then
\begin{equation*}\label{segers4.1v2_plus} \sqrt{k}\big(R_{k,n} - \mu(\xi_{n-k,n})\big) \stackrel{d}{\to} N(0, {\rm var} [\vartheta (\xi^{\gamma_+})]).\end{equation*}
\item[(ii)] Assume additionally Condition~\ref{secondorder} for $\gamma>0,$ $\sqrt{k}\hat a(n/k)\to\lambda$ with finite $\lambda$ for $\hat a$ defined above, and $\vartheta^\prime$ is continuous and eventually monotone.
Finally, assume $\amsmathbb{E}\big[\xi^{\gamma_+ +\delta} |\vartheta ^\prime(\xi^{\gamma_+ +\delta})|\big]<\infty$ for some $\delta>0.$
Then
\begin{equation*}\label{segers4.1v3_plus} \sqrt{k}\big(R_{k,n} - \amsmathbb{E}\big[\vartheta (\xi^{\gamma_+})\big]\big) \stackrel{d}{\to} N(C \lambda, {\rm var} [\vartheta (\xi^{\gamma_+})]),\end{equation*} where $C = \amsmathbb{E}\big[\tilde g(\xi) \vartheta ^\prime(\xi^{\gamma_+})\big]$ and $\tilde g$ is defined by \eqref{tildeg}.
\end{enumerate}
\end{corollary}

\begin{remark} In the above, it is always possible to find an intermediate sequence $k = k_n$ such that $\sqrt{k}\hat a(n/k)\to\lambda$ is satisfied. Indeed, if $\gamma>0$ then $\hat a(t)\to 0$ as $t\to\infty$ by  definition (see Condition~\ref{secondorder}), and if $\gamma\le0$ then $\hat a(t)\to 0$ as $t\to\infty$ by Lemma 1.2.9, \cite{dehaan}. \end{remark}

Note again that to estimate the extreme value index, Corollary~\ref{Cor1} is not enough since for negative $\gamma$, $R_{k,n}$ is a consistent estimator of $\vartheta (1)$, which itself does not depend on $\gamma.$ \\

Consequently, we now consider another candidate function $f(x,t).$ As a starting point, note that it follows from \eqref{basic} that
\begin{equation}\label{formoment}\lim_{t\to\infty} \frac{\log U_F(tx) - \log U_F(t)}{a(t)/U_F(t)} = h_{\gamma_-}(x)\end{equation} with $\gamma_- = \min(\gamma,0)$ and the same $a(t)$ as in \eqref{basic}; see, e.g., formula (3.5.4) in~\cite{dehaan}. The limit \eqref{formoment} implies that $\log U_F(t)$ is {\it extended regularly varying} with index $\gamma_-,$ see the properties of such functions in section B.2, \cite{dehaan}. The relation \eqref{formoment} has previously been used to establish the asymptotic behavior of the moment estimator of EVI; see~\cite{moment}. Thus, let us consider the candidate function
\[f(x,t) = \frac{\log U_F(tx) - \log U_F(t)}{a(t)/U_F(t)},\] and therefore $g(x) = h_{\gamma_-}(x).$ For $g_1$ and $g_2$ in \eqref{cond2} one can use $0$ and $x\mapsto\delta x^\delta$ for some $ \delta\in(0,1)$ if $\gamma\ge 0$ by Corollary B.2.10, \cite{dehaan}, and \begin{equation}x\mapsto\gamma^{-1} \big(\max[(1\mp\delta) x^{\gamma\mp\delta} - (1 \pm \delta), 0]\big)\label{g1g2}\end{equation} for some $\delta\in(0, -\gamma),$ if $\gamma<0$ by Theorem B.2.2 ibid. and Potter's bounds, respectively.

 {Recall that the second-order condition for $\log U_F(t)$ is given by \eqref{secondorderlog}.} %reads as follows
%\begin{equation}\label{secondorderlog}\lim_{t\to\infty}\frac{\frac{\log U_F(tx) - \log U_F(t)}{a(t)/U_F(t)} - h_{\gamma_-}(x)}{Q(t)} = \int_{1}^x s^{\gamma_--1}\int_{1}^s u^{\rho^\prime-1} \dd u \dd s=: H_{\gamma_-, \rho^\prime}(x),\end{equation}
%for $ x>0$, some $\rho^\prime\le 0$ and $Q(t)$ positive or negative with $Q(t) \to 0, t \to \infty,$ see, e.g., formula (3.5.11) in~\cite{dehaan}.  Here, the function $H_{\gamma, \rho}$ can be re-expressed as
%\begin{equation}\label{hgammarho}H_{\gamma, \rho}(x) = \frac{1}{\rho}\left(\frac{x^{\gamma+\rho}-1}{\gamma+\rho} - \frac{x^\gamma - 1}{\gamma}\right),\end{equation} which for the cases $\gamma = 0$ and $\rho = 0$ is understood to be equal to the limit of \eqref{hgammarho} as $\gamma\to 0$ or $\rho \to 0,$ respectively.
Note that \eqref{secondorderlog} follows from the second-order condition, which is standard for extreme value theory, \begin{equation}\label{standardsecondorder} \lim_{t\to\infty}\frac{\frac{U_F(tx) - U_F(t)}{a(t)} - h_{\gamma}(x)}{A(t)} =  H_{\gamma, \rho}(x), \end{equation} if $\gamma\neq \rho,$ see, e.g., formula (3.5.10) in~\cite{dehaan}, where the relation between $Q$ and $A$ is stated by formula (3.5.12) ibid. and $\rho^\prime$ is the function of $\gamma$ and $\rho$ defined in Lemma B.3.16 ibid.

Hence, \eqref{secondorderlog} implies $\tilde g = H_{\gamma_-, \rho^\prime}.$ Since $\gamma_-$ and $\rho^\prime$ are non-positive, one can derive similarly to Lemma 4.6, \cite{segers}, for some $K>0$ and $\delta>0$
\begin{equation}\left|\frac{\log U_F(tx) - \log U_F(t)}{a(t)/U_F(t)} - h_{\gamma_-}(x)\right| \le K |Q(t)|\, x^{\gamma_- + \delta}, \quad x\ge 1, t\ge T.\label{gbar}\end{equation} Clearly, $\delta$ in \eqref{gbar} can be chosen such that $\gamma + \delta<0$ for negative $\gamma.$ Finally, $Q(t)$ appearing in \eqref{secondorderlog} is regularly varying with index $\rho^\prime$ by Theorem B.3.1, \cite{dehaan}, therefore as above we can require $\sqrt{k} Q(n/k) \to \lambda$ instead of $\sqrt{k} Q(\xi_{n-k,n}) \to \lambda.$

Consider the statistic
\[T_{k,n} = \frac{1}{k}\sum_{i=1}^k \vartheta \left(\frac{\log X_{n-i+1,n} - \log X_{n-k,n}}{a(U^{-1}(X_{n-k,n}))/X_{n-k,n}}\right)\] for some continuous function $\vartheta $ and denote
\[\mu_f(t) =  \amsmathbb{E}\, \vartheta \left(\frac{\log U_F(t\xi) - \log U_F(t)}{a(t)/U_F(t)}\right),\] where $\xi\sim \text{Pareto}(1)$ as before. 

The following result is a consequence of Theorem~\ref{T1gen} and Theorem~\ref{T2gen}.
\begin{corollary}\label{Cor2} \, \begin{enumerate} \item[(i)] Let $F\in \mathcal{D}(G_\gamma)$ with $\gamma\in \amsmathbb{R}$ and $\tau_F>0.$ Let $\vartheta : [0, \infty) \to \amsmathbb{R}$ be continuous and eventually monotone, and if $\gamma \ge 0,$ satisfy $|\vartheta (x)|\le A\max(x^c,1),$ $x\ge 0,$ for some $A>0$ and $c>0.$ Assume $k$ and $\xi_1, \xi_2, \ldots$ are as in Theorem~\ref{T1}. Then $T_{k,n}$ is consistent for $\amsmathbb{E}\big[\vartheta (h_{\gamma_-}(\xi))\big]$ and
\begin{equation*}\label{segers4.1v2_minus} \sqrt{k}\big(T_{k,n} - \mu_f(\xi_{n-k,n})\big) \stackrel{d}{\to} N\big(0, {\rm var} \big[\vartheta (h_{\gamma_-}(\xi))\big]\big).\end{equation*}
\item[(ii)] Assume additionally \eqref{secondorderlog}, $\sqrt{k}Q(n/k)\to\lambda$ with finite $\lambda$ for $Q$ defined in \eqref{secondorderlog}, and continuous and eventually monotone $\vartheta^\prime$ is such that $|\vartheta^\prime(0)|<\infty.$ 
Finally, if $\gamma \ge 0,$ assume $|\vartheta^\prime(x)|\le A\max(x^c,1),$ $x\ge 0,$ for some $A>0$ and $c>0.$ Then
\begin{equation*}\label{segers4.1v3_minus} \sqrt{k}\big(T_{k,n} - \amsmathbb{E}\big[\vartheta (h_{\gamma_-}(\xi))\big]\big) \stackrel{d}{\to} N\big(C\lambda, {\rm var} \big[\vartheta (h_{\gamma_-}(\xi))\big]\big),\end{equation*} where $C = \amsmathbb{E}\big[H_{\gamma_-,\rho^\prime}(\xi) \vartheta ^\prime(h_{\gamma_-}(\xi))\big]$ and $H_{\gamma_-, \rho^\prime}$ is defined by \eqref{secondorderlog}.
\end{enumerate}
\end{corollary}

{  In the same way that Corollary~\ref{Cor1} did not yield a suitable estimators when operating in the $\gamma<0$ domain, Corollary~\ref{Cor2} is again not entirely suitable for all max-domains. Although it succeeds for when $\gamma\le0$, the $\gamma>0$ domain fails, since then $T_{k,n}$ converges in probability to $\amsmathbb{E}[\vartheta(\eta)],$ where $\eta\sim \mbox{Exp}(1).$ However, the tools developed in this section are sufficient -- upon suitable combination -- for the $\gamma\in \amsmathbb{R}$ case. 

%Next, we link the theory of generalized residual estimators to generalized EKM integrals for all max-domains of attraction in Section~\ref{sec:exten}. Thereafter, in Section~\ref{sec:evi_est}, the precise method of combination yields as an example the moment estimator  of $\gamma\in\amsmathbb{R}$ itself.
%\cite{moment}, and subsequently taking the censoring effects into account %which we develop in
}

\newpage
\section{Proofs of Section~\ref{sec:ekmi}}\label{AppA} %% if no title is needed, leave empty \section*{}.
%Appendices should be provided in \verb|{appendix}| environment,
%before Acknowledgements.
%
%If there is only one appendix,
%then please refer to it in text as \ldots\ in the \hyperref[appn]{Appendix}.
\subsection*{Notation} 
{ First, define the random vectors $\{(V^t_i, \delta^t_i)\}_{i=1}^k$ such that \begin{equation}\label{bardk}\bar D^t_k := \left(\binom{V^t_{1,k}}{\delta^t_{[1:k]}},\dots,\binom{V^t_{k,k}}{\delta^t_{[k:k]}}\right)\stackrel{d}{=}\left( \binom{Z_{n-k+1,n}/Z_{n-k,n}}{\delta_{[n-k+1:n]}},\ldots, \binom{Z_{n,n}/Z_{n-k,n}}{\delta_{[n:n]}} \right)\end{equation} given $Z_{n-k,n}=t,$ $k<n,$ where $\{\delta^t_{[i:k]}\}_{i=1}^k$ denote concomitants of $\{\delta^t_i\}_{i=1}^k.$}

The following notation provides the classical big-o and small-o notation adapted to uniformly control the tail of extreme value estimators. It is used throughout the proofs.

\begin{definition}
{ For the random vectors $\bar D^t_k$ introduced above,} %$\{V^t_i\}_{i=1}^k$ introduced in Section~\ref{evt} and $\{\delta^t_i\}_{i=1}^k$ such that $$(V^t_{1,k},\dots,V^t_{k,k})\stackrel{d}{=}( Z_{n-k+1,n}/Z_{n-k,n},\ldots , Z_{n,n}/Z_{n-k,n})$$ given $Z_{n-k,n}=t,$ 
an array of functions $(f_{k,n})_{k,n\in\amsmathbb{N}}$, and a positive deterministic sequence $a_1,a_2,\ldots$ we say that $f_{k,n}(\bar{D}^t_k)$ is $o_{{P}}(a_k)$ uniformly for large $t$ %w.r.t. $Z_{n-k,n}$
if for every $\epsilon>0$,
\begin{align}
\lim_{n\to\infty} 
\limsup\limits_{t\rightarrow\infty}\pr(|f_{k,n}(\bar{D}^t_k)/a_k|>\epsilon) = 0.
\end{align} 
We write for short (when no confusion arises) that $f_{k,n}(\bar{D}^t_k)=\overline{o}_{P}(a_k)$. Analogously $f_{k,n}(\bar{D}^t_k)=\overline{O}_{P}(a_k)$ using the notions of bounded in probability in place of convergence to zero in probability.
\end{definition}

Notice that according to \eqref{bardk}, $r_{k,n} = r_{k,n}\big((Z^\ast_i, \delta_i)_{i=1}^k\big)$ defined in Theorem~\ref{T4} is equal in distribution to $r^t_{k,n} := r_{k,n}(\bar{D}^t_k)$ given $Z_{n-k,n} = t.$ Then, since the latter relation is a distributional property, Theorem~\ref{T4} states that $r^t_{k,n} = \overline{o}_{P}(k^{-1/2}).$  %Thus, in what follows, we regard $r_{k,n}$ through the latter construction.

\begin{proof}[Proof of Theorem~\ref{T4}] 
The following condition serves as a building block towards the general case.

\begin{condition}\label{cond:bounded_phi}
The function $\varphi$ satisfies
\begin{equation}\label{infinite}\varphi(x) = 0 \:\: \text{ for all }\:\: x>T \:\:\text{ and some }\:\: T< \infty.\end{equation}
\end{condition}

\begin{theorem}[Decomposition, truncated case]\label{theo:Rkn_term}
Let $\varphi$ be continuous and
%have envelope $\bar\varphi$ 
satisfy Condition~\ref{cond:bounded_phi}. Assume the conditions imposed on $F$ and $G$ in Theorem~\ref{T4}. Let $k=k_n$ be an intermediate sequence. Then the random sequence $r^t_{k,n}$ %from \eqref{T4new}
is $\overline{o}_{P}(k^{-1/2})$. 
\end{theorem}

\begin{proof}[Proof of Theorem~\ref{theo:Rkn_term}]

 We can find i.i.d. random variables $\{V^t_i\}_{i=1}^k$ such that the joint distribution of the set of their order statistics $(V^t_{1,k}, \ldots, V^t_{k,k})$ equals the joint distribution of the set $(Z_{n-k+1,n}/Z_{n-k,n}, \ldots, Z_{n,n}/Z_{n-k,n})$ given $Z_{n-k,n} = t,$ which follows by slight modification of lemma 3.4.1 in \cite{dehaan}. But it is not evident whether the random variables $\{\delta^t_i\}_{i=1}^k$ introduced at the beginning of this section %in Theorem~\ref{theo:Rkn_term} 
are i.i.d. as well. The following lemma answers this question.
\begin{lemma}\label{L1} The sequence of vectors $\{(V_i^t, \delta^t_i)\}_{i=1}^k$ 
 is i.i.d. \end{lemma}
\begin{proof} Since the distribution of $\{(Z_i,\delta_i)\}_{i=1}^n$ is completely defined by the distribution of $\{(X_i, Y_i)\}_{i=1}^n,$ let us consider the latter. Let $\{T_i\}_{i=1}^k$ be the times of exceedances over $t$ by the sequence $\{Z_i\}_{i=1}^n,$ %given $Z_{n-k,n} = t,$ 
that is 
\[T_i = \min\{j>T_{i-1}: \; Z_j > t\}, \quad T_0 = 0.\]
Clearly, the total number of $T_i$ is equal to $k$ given $Z_{n-k,n} = t.$ Consider the joint distribution of the pairs $\{(X_{T_i}, Y_{T_i})\}_{i=1}^k$ given $Z_{n-k,n} = t.$ First, the condition $Z_{n-k,n} = t$ means that $X_{T_i}\ge t,$ $Y_{T_i}\ge t,$ $i=1,\ldots,k,$ so it is easy to show that
\[\mathcal{L}\left(\left.\binom{X_{T_1}}{ Y_{T_1}}, \ldots, \binom{X_{T_k}}{ Y_{T_k}}\right| Z_{n-k,n} = t\right) = \]\[\mathcal{L}\left(\left.\binom{X_{T_1}}{ Y_{T_1}}, \ldots, \binom{X_{T_k}}{Y_{T_k}}\right| \binom{\text{ the number of exceedances is exactly } k} {\text{and } X_{T_i}\ge t, Y_{T_i}\ge t, i=1,\ldots,k}\right)\] Clearly, given that the number of exceedances is exactly $k,$ the vector $(T_1, \ldots, T_k)$ is uniformly distributed on the set of all k-tuples of $\{1,\ldots,n\},$ thus the latter distribution is equal to 
\[\mathcal{L}\left(\left.\binom{X_{1}}{ Y_{1}}, \ldots, \binom{X_{k}}{Y_{k}}\right| \binom{\text{ the number of exceedances is exactly } k} {\text{and } X_{i}\ge t, Y_{i}\ge t, i=1,\ldots,k}\right).\] The latter distribution is clearly independent of $n,$ thus it is equal to
\[\mathcal{L}\left(\left.\binom{X_{1}}{ Y_{1}}, \ldots, \binom{X_{k}}{Y_{k}}\right| X_{i}\ge t, Y_{i}\ge t, i=1,\ldots,k\right).\] Since $\{X_i\}$ and $\{Y_i\}$ are independent, this distribution coincides with
\[\mathcal{L}\left(\binom{X^t_{1}}{ Y^t_{1}}, \ldots, \binom{X^t_{k}}{Y^t_{k}}\right),\] where $\{X_i^t\}_{i=1}^k$ and $\{Y_i^t\}_{i=1}^k$ are two independent i.i.d. samples with cdfs $F^t(x/t)$ and $G^t(x/t),$ respectively. Since the distribution of $(Z_i, \delta_i)$ is completely determined by the distribution of $(X_i, Y_i),$ we can select $V_i^t = \min(X_i^t, Y_i^t)/t$ and $\delta_i^t = I(X_i^t\le Y_i^t).$
\end{proof}
Consider now the extreme Kaplan--Meier integral $S_{k,n}(\varphi),$ defined by \eqref{extremeKM}, given $Z_{n-k,n} = t.$ By \eqref{bardk} and Lemma~\ref{L1}, under this condition the extreme Kaplan--Meier estimator \eqref{kaplanmeierextreme} equals in distribution to
\[\amsmathbb{F}_{k,n}^t(x) =  1 - \prod_{i=1}^k\left[1 - \frac{\delta^t_{[i:k]}}{k-i+1}\right]^{I(V^t_{i,k}\le x)},\]
and replacing in the latter $k$ by $n,$ $\delta^t_{[i:k]}$ by $\delta_{[i:n]},$ and $V_{i,k}^t$ by $Z_{i,n}$ gives us exactly the classical Kaplan--Meier estimator \eqref{kaplanmeyer}. In particular, this implies that, for fixed $k$ and $n$, we may follow along the lines of the initial steps of the proof of Theorem 1.1 in~\cite{stute} (specifically, compare the below formula with (2.1) therein) to obtain 
\begin{align}\label{full} S_{k,n}(\varphi)\, {\stackrel{d}{=}}\, \frac{1}{k} \sum_{i=1}^k \varphi(V_i^t) \gamma^t_0(V_i^t) \delta^t_i [1 + B_{ik}^t + C_{ik}^t] + \frac{1}{k} \sum_{i=1}^k \frac{1}{2}\varphi(V_i^t) \delta^t_i e^{\Delta_i^t}[B_{ik}^t + C_{ik}^t]^2,\end{align}
where
\begin{equation}\label{bint}B_{ik}^t = k \int_{1}^{V_i^t-}\log \left[1 + \frac{1}{k(1 - H_k^t(z))}\right]H^{0,t}_k(\dd z) - \int_{1}^{V_i^t-} \frac{1}{1 - H_k^t(z)}H^{0,t}_k(\dd z),\end{equation}
\begin{equation}\label{cint}C_{ik}^t = \int_{1}^{V_i^t-} \frac{1}{1 - H_k^t(z)}H^{0,t}_k(\dd z) - \int_{1}^{V_i^t-} \frac{1}{1 - H^t(z)}H^{0,t}(\dd z),\end{equation} $\Delta_i^t$ is between the first term in the rhs of \eqref{bint} and second term in the rhs of \eqref{cint}, and \[H^t_k(z) =\frac{1}{k} \sum_{i=1}^k I(V_{i}^t \le z), \quad  H^{j,t}_k(z) = \frac{1}{k} \sum_{i=1}^k I(V_{i}^t \le z, \delta_i^t = j), \quad j=0,1.\] Moreover, 
\begin{align}
\frac{1}{k} \sum_{i=1}^k \varphi(V_i^t) \gamma^t_0(V_i^t) \delta^t_i C_{ik}^t  = 
& - \int\int\int \frac{I(v<u, v<w) \varphi(w) \gamma_0^t(w)}{(1 - H^t(v))^2} H_k^t(\dd u) H_{k}^{0,t}(\dd v) H_{k}^{1,t}(\dd w) \nonumber \\ 
& + 2\int\int I(v<w) \frac{\varphi(w) \gamma_0^t(w)}{1 - H^t(v)} H^{0,t}_k(\dd v) H^{1,t}_k(\dd w) \label{termlemma2.3} \\
& - \int\int I(v<w) \frac{\varphi(w) \gamma_0^t(w)}{1 - H^t(v)} H^{0,t}(\dd v) H^{1,t}_k(\dd w) + r_{k,n}^{1,t},\nonumber
\end{align}
where 
\begin{equation}r_{k,n}^{1,t} = \int\int \varphi(w) \gamma_0^t(w)  I(z<w) \frac{(H_k^t(z) - H^t(z))^2}{(1 - H^t(z))^2 (1 - H^t_k(z))} H^{0,t}_k(\dd z)H^{1,t}_k(\dd w).\label{rkn1}\end{equation}

As a consequence of Condition~\ref{cond:bounded_phi}, all denominators appearing in the proof are bounded away from zero. Hereinafter, integrals without {integration limits} are understood to be from 1 to $\infty.$ The following result is the analogue of Lemma 2.3, \cite{stute}, in our setting.
\begin{lemma}
Under the above conditions we have that the following term is $\overline{o}_{P}(k^{-1/2})$:
\begin{align*}
\int\int I(v<w) \frac{\varphi(w) \gamma_0^t(w)}{1 - H^t(v)} [(H^{0,t}_k(\dd v) -  H^{0,t}(\dd v))(H^{1,t}_k(\dd w) - H^{1,t}(\dd w))].
\end{align*}
\end{lemma}
\begin{proof}
Consider the second term in the rhs of \eqref{termlemma2.3}
\begin{align*}V^t_{k,n} &:= \int\int I(v<w) \frac{\varphi(w) \gamma_0^t(w)}{1 - H^t(v)} H^{0,t}_k(\dd v) H^{1,t}_k(\dd w)\\
&= \frac{1}{k^2}\sum_{i=1}^k \sum_{j=1}^k \frac{I(V_i^t<V_j^t) (1 - \delta_i^t) \delta_j^t \varphi(V_j^t) \gamma_0^t(V_j^t)}{1 - H^t(V_i^t)}.\end{align*}
We now rely on several results from classical U-statistics, specifically, Lemma 5.2.1A and Theorem 5.3.2 in~\cite{serfling}. First, note that $V^t_{k,n}$ is in fact a V-statistic (see for definition, e.g., Section 5.1.2, \cite{serfling}). Next, define 
\begin{align*}U^t_{k,n}:&=\frac{k}{k-1}V^t_{k,n}\\
 & =  \frac{1}{k(k-1)}\sum_{i=1}^k \sum_{j\neq i} \frac{I(V_i^t<V_j^t) (1 - \delta_i^t) \delta_j^t \varphi(V_j^t) \gamma_0^t(V_j^t)}{1 - H^t(V_i^t)} \\
& =  \frac{1}{k(k-1)} \sum_{i=2}^k\sum_{j=1}^{i-1}\left(\frac{I(V_i^t<V_j^t) (1 - \delta_i^t) \delta_j^t \varphi(V_j^t) \gamma_0^t(V_j^t)}{1 - H^t(V_i^t)} \right. \\  & \qquad\qquad\qquad\qquad+ \left.\frac{I(V_j^t<V_i^t) (1 - \delta_j^t) \delta_i^t \varphi(V_i^t) \gamma_0^t(V_i^t)}{1 - H^t(V_j^t)}\right)\\
& =  \frac{1}{k(k-1)} \sum_c \left(h_0\Big(\binom{V_i^t}{\delta_i^t}, \binom{V_j^t}{\delta^t_j}\Big) + h_0\Big(\binom{V_j^t}{\delta^t_j}, \binom{V_i^t}{\delta_i^t}\Big)\right) \\&= \frac{2}{k(k-1)} \sum_c h\Big(\binom{V_i^t}{\delta_i^t}, \binom{V_j^t}{\delta^t_j}\Big),\end{align*} 
where $\sum_c$ denotes summation over the $\binom{k}{2}$ combinations of 2 distinct elements $\{i_1, i_2\}$ from $\{1, \ldots, k\},$ \[h_0(\textbf{x}, \textbf{y}) = \frac{I(x_1 < y_1)(1 - x_2) y_2 \varphi(y_1) \gamma_0^t(y_1)}{1 - H^t(x_1)}\] with $\textbf{x}, \textbf{y} \in \amsmathbb{R}^2,$  and $h(\textbf{x}, \textbf{y}) = \big[h_0(\textbf{x}, \textbf{y}) + h_0(\textbf{y}, \textbf{x})\big]/2.$ Clearly, the above expression casts $U^t_{k,n}$ as a U-statistic, since the kernel $h$ is symmetric. Denote
\[\theta^t = \amsmathbb{E} h\Big(\binom{V_1^t}{\delta_1^t}, \binom{V_2^t}{\delta^t_2}\Big).\] Then, according to the definition from Section 5.3.1, \cite{serfling}, the Hajek projection of the U-statistic $U^t_{k,n}$ is
\[\widehat{U}^t_{k,n} = \sum_{i=1}^k \amsmathbb{E}\left[U^t_{k,n}\left|\binom{V_i^t}{\delta^t_i}\right.\right] - (k-1) \theta^t = \frac{2}{k}\sum_{i=1}^k h_1\Big(\binom{V_i^t}{\delta^t_i}\Big) - \theta^t,\] where
\[h_1(\textbf{x}) = \amsmathbb{E}\,h\Big(\textbf{x}, \binom{V_1^t}{\delta^t_1}\Big).\]
Similarly to Section 5.3.2, \cite{serfling}, we have that $U^t_{k,n} - \widehat{U}^t_{k,n}$ is also a U-statistic,
\[U^t_{k,n} - \widehat{U}^t_{k,n} = \frac{1}{\binom{k}{2}} \sum_c H\Big(\binom{V_i^t}{\delta_i^t}, \binom{V_j^t}{\delta^t_j}\Big)\] but now with the symmetric kernel
\[H(\textbf{x}, \textbf{y}) = h(\textbf{x}, \textbf{y}) - h_1(\textbf{x}) - h_1(\textbf{y}) + \theta^t.\]
Following the notation $\zeta_c$ from Section 5.2.1, \cite{serfling}, we define in our case for $c=0,1,2$,
\begin{equation} \label{zeta}\zeta^t_0 = 0, \:\:\zeta^t_1 = {\rm var}\, h_1\Big(\binom{V_1^t}{\delta^t_1}\Big)\:\: \text{ and }\:\: \zeta^t_2 = {\rm var}\,h\Big(\binom{V_1^t}{\delta_1^t}, \binom{V_2^t}{\delta^t_2}\Big),\end{equation} and note that $\zeta^t_0 \le \zeta^t_1 \le \zeta^t_2.$ Define also $\chi^{t}_c,$ $c = 0,1,2,$  by replacing the kernel $h$ with $H$ in \eqref{zeta} and note that $\chi^t_0 = \chi^t_1 = 0,$ see the argument before Theorem 5.3.2 in~\cite{serfling}. Based on (*) in Lemma 5.2.1, \cite{serfling}, we derive 
\begin{equation}\label{variances}{\rm var}\, U^t_{k,n} = \frac{4(k-2)}{k(k-1)} \zeta^t_1 + \frac{2}{k(k-1)} \zeta^t_2, \quad {\rm var}\, \big(U^t_{k,n} - \widehat U^t_{k,n}\big) = \frac{2}{k(k-1)} \chi^t_2.\end{equation} Moreover, by the Cauchy-Schwarz inequality
\[\chi_2^t \le 6 \zeta^t_1 + 3 \zeta^t_2.\]
Contrary to Lemma 2.3 in~\cite{stute}, in our setting ${\rm var}\, U^t_{k,n}$ and ${\rm var}\, \big(U^t_{k,n} - \widehat U^t_{k,n}\big)$ might depend on $t$ asymptotically as $t\to\infty,$ and thus an appeal to \eqref{variances} is not enough to prove 
$U^t_{k,n} - \widehat U^t_{k,n} = O(k^{-1})$ in probability uniformly for large $t$ when transitioning from the conditional setting given $Z_{n-k,n} = t$ to the unconditional setting. 

We now argue that ${\rm var}\, U^t_{k,n}$ and ${\rm var}\, \big(U^t_{k,n} - \widehat U^t_{k,n}\big)$ do not depend on $t$ asymptotically as $t\to\infty$ and $\zeta^t_1,$ $\zeta^t_2$ can be bounded from above uniformly in large $t.$ %and thus, the arguments of Lemmas 2.2 and 2.3, \cite{stute}, may be adapted to our setting. 
Indeed, for $\theta^t,$ we get by \eqref{gamma_ex}, \eqref{h01t},
Potter's bounds, and the dominated convergence theorem
\begin{align}\nonumber\theta^t &= \int_1^\infty\int_1^\infty  \frac{I(v<w)\varphi(w)\gamma_0^t(w)}{1 - H^t(v)} H^{0,t}(\dd v) H^{1,t}(\dd w)\\& = \nonumber \int_1^\infty\int_1^w \varphi(w) \frac{1}{(1 - G^t(v))} G^t(\dd v) F^t(\dd w)\\
&\label{P0} \to \frac{1}{\gamma_G \gamma_F}\int_1^\infty \varphi(w) \log(w) w^{-1/\gamma_F - 1} \dd w =:\theta.\end{align}

Let us explain the latter convergence in more detail since we  use the same argument repeatedly in the proof. The integral in the second line of \eqref{P0} can be represented as a mathematical expectation of some random variable depending on $t.$ To apply the dominated convergence theorem, we should show that this random variable is bounded uniformly in large $t$ by some random variable with finite expectation. First, let us determine the former. We have,
\begin{eqnarray}\nonumber\int_1^\infty\int_1^w \varphi(w) \frac{1}{(1 - G^t(v))} G^t(\dd v) F^t(\dd w) &=& \int_1^\infty\int_1^{U_{F^t}(x)}  \frac{\varphi(U_{F^t}(x))}{(1 - G^t(v))} G^t(\dd v) \frac{1}{x^2} \dd x\\
&=& \amsmathbb{E}\left[\varphi(U_{F^t}(\xi)) \int_1^{U_{F^t}(\xi)}  \frac{1}{(1 - G^t(v))} G^t(\dd v)\right],\label{clarification}\end{eqnarray}
where $\xi$ denotes a Pareto$(1)$ random variable and $U_F$ denotes the tail quantile function of $F,$ i.e. $U_F(x) := \inf\{y: F(y) \ge 1 - 1/x\}.$ Thus, the random sequence of interest is the one inside the expectation in the right-hand side of \eqref{clarification}. 

By Potter's bounds, eventual continuity of $G$ and continuity of $\varphi,$ for some $\delta\in(0,1)$ and large $t$
\begin{align}\nonumber
&|\varphi(w)|\int_1^w  \frac{1}{(1 - G^t(v))} G^t(\dd v) \le - |\varphi(w)| \log(1 - G^t(w)) \\ 
\label{a9}&\le |\varphi(w)| \big(\log w\,(1/\gamma_G + \delta) - \log(1 - \delta)\big)\\ \nonumber
&\le A\, I(w\in[1,1+\delta]) + B\, I(w>1+
\delta) |\varphi(w)| \log w,\end{align} for some positive constants $A$ and $B$. Similarly to \eqref{clarification}, we have
\[\int_{1}^\infty |\varphi(w)| \log w\, F^t(\dd w) = \amsmathbb{E} \big[|\varphi(U_{F^t}(\xi))| \log (U_{F^t}(\xi)) \big].\] On the other hand, one can show that
\begin{equation}\label{uft}U_{F^t}(x) = \frac{U_F(x/(1 - F(t)))}{U_F(1/(1 - F(t)))}.\end{equation}
By Potter's bounds, we get for some $\bar\delta,\varepsilon>0,$ $T_0>1$ and all $t>T_0$
\begin{equation} \label{potterforintegral} f_\ast(x) := (1-\bar\delta) x^{\gamma_F - \varepsilon} \le \frac{U_F(x/(1 - F(t)))}{U_F(1/(1 - F(t)))} \le (1+\bar\delta) x^{\gamma_F + \varepsilon} = f^\ast(x).\end{equation}
Assume that $|\varphi(x)|$ is monotone for $x>T_1\ge1.$ Moreover, assume without loss of generality that $|\varphi(x)|$ is non-decreasing on this set. Denote $T_2 = \max(T_0, T_1).$
Then by the three latter identities and properties of $\varphi$
\begin{eqnarray*}
|\varphi(U_{F^t}(w))| \log(U_{F^t}(w)) &\le& C I(w\in[1,T_2]) + |\varphi(U_{F^t}(w))| \log(U_{F^t}(w)) I(w>T_2)\\
&\le& C I(w\in[1,T_2]) + |\varphi(f^\ast(w))| \log(f^\ast(w)) I(w>T_2) \end{eqnarray*} for some $C>0,$ where the latter function does not depend on $t$ as required. Finally,
\begin{eqnarray*} \amsmathbb{E}\big[|\varphi(f^\ast(\xi))| \log(f^\ast(\xi)) I(\xi>T_2)\big] &=& \int_{T_2}^\infty|\varphi(f^\ast(x))| \log (f^\ast(x))  \frac{1}{x^2} \dd x \\ &\le& C_1 \int_{T_3}^\infty |\varphi(x)| \log x\, x^{- 1/(\gamma_F + \varepsilon) - 1} \dd x,\end{eqnarray*} for some $C_1>0$ and $T_3>1,$ and the last integral is finite by continuity of ${\varphi}$ and Condition~\ref{cond:bounded_phi}. Summarizing the above, for all large $t$ the random variable 
\[\eta_t = |\varphi(U_{F^t}(\xi))| \int_1^{U_{F^t}(\xi)} \frac{1}{(1 - G^t(v))} G^t(\dd v)\] is bounded by some integrable random variable, therefore the conditions of the dominated convergence theorem hold, and the latter implies \eqref{P0}.

In a similar fashion, applying the dominated convergence theorem and using additionally Potter's bounds for $H^t,$ we get the uniform upper bound for $\zeta^t_2$ and thus,
\begin{eqnarray*}
\zeta^t_2 & = & \int_1^\infty \int_1^\infty \frac{I(v<w)\varphi^2(w)(\gamma_0^t(w))^2}{(1 - H^t(v))^2} H^{0,t}(\dd v) H^{1,t}(\dd w) - (\theta^t)^2 \\ & = & \int_1^\infty\int_1^w \varphi^2(w) \frac{1 - F^t(v)}{(1 - H^t(v))^2 (1 - G^t(w))} G^t(\dd v) F^t(\dd w) - (\theta^t)^2\\ &\to& \frac{\gamma_H}{\gamma_G \gamma_F}\int_1^\infty \varphi^2(w) w^{1/\gamma_H - 1} \dd w - \theta^2.
\end{eqnarray*} 
The analysis of $\zeta^t_1$ is similar, thus we omit it. 

To conclude, by a final application of Potter's bounds, we have established that the variance of $U^t_{k,n} - \widehat{U}^t_{k,n}$ is $O(k^{-2})$  as $k,\:n/k\to\infty$ uniformly in large $t,$ %and thus the proof of Lemma 2.3, \cite{stute} (and, consequently, Lemma 2.2 ibid.; see Lemma~\ref{LA.3} below), can be reused in our setting. In other words, as $k,\:n/k\to\infty$,
and then by applying Chebyshev's inequality we get that the random sequence $U^t_{k,n} - \widehat{U}^t_{k,n}$ is $\overline{o}_{P}(k^{-1/2})$.
\end{proof}
We may prove 
an analogous result for the first term in the rhs of \eqref{termlemma2.3}.
\begin{lemma}\label{LA.3}
Under the above conditions we have that the following term is $\overline{o}_{P}(k^{-1/2})$:
\begin{align*}
\int\int\int \frac{I(v<u, v<w) \varphi(w) \gamma_0^t(w)}{(1 - H^t(v))^2}
\cdot[
H_k^t(\dd u) H_{k}^{0,t}(\dd v) H_{k}^{1,t}(\dd w) -H_k^t(\dd u) H^{0,t}(\dd v) H^{1,t}(\dd w)
\\
-H^t(\dd u) H_{k}^{0,t}(\dd v) H^{1,t}(\dd w)
-H^t(\dd u) H^{0,t}(\dd v) H_{k}^{1,t}(\dd w)
+2 H^t(\dd u) H^{0,t}(\dd v) H^{1,t}(\dd w)
].
\end{align*}
\end{lemma}
\begin{proof}
Similar to the previous case.
\end{proof}

Thus, similarly to Corollary 2.4, \cite{stute}, we derive that
\begin{equation}\label{c24} \frac{1}{k} \sum_{i=1}^k \varphi(V_i^t) \gamma^t_0(V_i^t) \delta^t_i C_{in}^t = \frac{1}{k}\sum_{i=1}^k \gamma_1^t(V_i^t) (1 - \delta^t_i) - \frac{1}{k}\sum_{i=1}^k \gamma_2^t(V_i^t) + r^{1,t}_{k,n} + r^{2,t}_{k,n},\end{equation} where $r^{2,t}_{k,n} = \overline{o}_{P}(k^{-1/2})$.

We turn our attention to the behavior of $r_{k,n}^{1,t},$ defined by \eqref{rkn1} (compare with $R_{n1}$ in~\cite{stute}).

\begin{lemma}
We have that  \begin{equation}|r_{k,n}^{1,t}| = \overline{O}_{P}(k^{-1}).\label{rn1}\end{equation}
\end{lemma}
\begin{proof}
First, by Condition~\ref{cond:bounded_phi}
\begin{align}|r_{k,n}^{1,t}| &\le \int_1^\infty |\varphi(w)| \gamma_0^t(w) H^{1,t}_k(\dd w) \cdot \int_1^\infty I(z<T) \frac{(H_k^t(z) - H^t(z))^2}{(1 - H^t(z))^2 (1 - H^t_k(z))} H^{0,t}_k(\dd z)\nonumber\\
&=: I^t_{k,1}\cdot I^t_{k,2},\label{rn1ineq}\end{align} therefore we can evaluate the asymptotic of $I^t_{k,1}$ and $I^t_{k,2}$ separately. First, for every $\varepsilon>0$ by  Markov's inequality and an argument similar to that used for proving \eqref{P0}
\begin{align*}
\pr(I^t_{k,1} > \varepsilon) &= \pr\left(\frac{1}{k}\sum_{i=1}^k|\varphi(V_i^t)|\gamma_0^t(V_i^t)\delta_i^t > \varepsilon\right)\\
&\le \frac{1}{\varepsilon} \int_{1}^\infty |\varphi(w)| F^t(\dd w) \le \frac{1}{\varepsilon\gamma_F}\int_1^\infty |\varphi(w)| w^{-1/\gamma_F-1+\delta} \dd w \end{align*} for large $t$ and some $\delta>0.$ Hence we derive that $I^t_{k,1} = \overline{O}_{P}(1).$ 

Next, consider $I^t_{k,2}.$ By \eqref{tailfunction} and Potter's bounds, we can bound the denominator in $I_{k,2}$ from below by some constant independent of $t$ since $(1 - H^t(z))/(1-H_k^t(z))$ is stochastically bounded under the condition $z<T.$ Hence, it is enough to find the asymptotic of the following integral
\[I^t_{k,3}:= \int_1^\infty I(z<T) (H_k^t(z) - H^t(z))^2 H^{0,t}_k(\dd z).\]
Clearly,
\[I^t_{k,3} \le \int_1^\infty (H_k^t(z) - H^t(z))^2 H^{t}_k(\dd z).\] By the classical goodness-of-fit technique, the latter integral equals
\[\int_0^1 (Q_k(x) - x)^2 Q_k(\dd x) =: I_{k,4},\] and thus does not depend on $t,$ where $Q_k(x) = \frac{1}{k}\sum_{i=1}^k I(U_i\le x),$ the empirical cdf of independent $U[0,1]$ random variables $\{U_i\}_{i=1}^k,$ and \begin{equation}\label{ui}U_i = H^t(V_i^t),\quad i=1, ..., k.\end{equation} Finally,
\[I_{k,4} = \int_0^1 (Q_k(x) - x)^2 \dd x + \int_0^1 (Q_k(x) - x)^2 \dd (Q_k(x) - x),\] where the first summand on the rhs is exactly the Cram\'er-von Mises statistic divided by $k$ and the second summand is asymptotically smaller in probability than the first one, since $\sqrt{k}(Q_k(x) - x)$ converges in distribution to a Brownian bridge on $[0,1]$ in a Skorokhod space. Hence, $I_{k,4}$ and therefore $I^t_{k,2}$ are $\overline{O}_{P}(k^{-1})$ as desired.
\end{proof}

Next, consider $S_{k,n}^{1,t},$ given by
\[S_{k,n}^{1,t} = \frac{1}{k} \sum_{i=1}^k \varphi(V_i^t) \gamma_0^t(V_i^t) \delta_i^t B_{ik}^t,\] (compare with $S_{n1}$ from Lemma 2.6 in~\cite{stute}).
\begin{lemma}
We have that
\begin{equation}\label{skn1}S_{k,n}^{1,t} = \overline{O}_{P}(k^{-1}).\end{equation}
\end{lemma}

\begin{proof}
Similarly to the relation (2.4) in~\cite{stute},
\begin{equation} \label{bin} -\frac{1}{2k} \int_1^{V_i^t-} \frac{H^{0,t}_k(\dd z)}{(1 - H^t_k(z))^2} \le B_{ik}^t \le 0. \end{equation}
The analysis of $S_{k,n}^{1,t}$ resembles that of $r_{k,n}^{1,t}.$ We have by \eqref{rn1ineq} and \eqref{bin},
\begin{eqnarray*}|S_{k,n}^{1,t}| &\le& \frac{1}{2k}\int_1^\infty\int_1^\infty |\varphi(w)|\gamma_0^t(w) I(z<w) \frac{1}{(1 - H_k^t(z))^2} H_{k}^{0,t}(\dd z) H_{k}^{1,t}(\dd w)\\
& \le & \frac{1}{2k} I^t_{k,1} \cdot \int_{1}^\infty I(z<T) \frac{1}{(1 - H_k^t(z))^2} H_{k}^{0,t}(\dd z) =: \frac{1}{2k} I^t_{k,1} \cdot I^t_{k,5},\end{eqnarray*} where $I^t_{k,1}$ was defined in \eqref{rn1ineq} and proved to be $\overline{O}_P(1).$ Recalling that $Q_k$ is the empirical cdf of the $U[0,1]$ random variables defined by \eqref{ui} and does not depend on $t$, we have by Potter's bound
\begin{eqnarray}\nonumber I^t_{k,5} & \le & \int_{1}^\infty I(z<T) \frac{1}{(1 - H_k^t(z))^2} H_{k}^{t}(\dd z) \\ \nonumber & = & \int_{0}^1 I(x<H^t(T)) \frac{1}{(1 - Q_k(x))^2} Q_{k}(\dd x)\\
&\le & \int_{0}^1 I(x<1 - (1 - \delta)T^{-1/\gamma_H-\delta}) \frac{1}{(1 - Q_k(x))^2} Q_{k}(\dd x) \label{ik5}\end{eqnarray} for large $t$ and some $\delta>0,$
which implies the result. \end{proof}

Finally, consider $S_{k,n}^{2,t},$ given by 
\begin{equation}S_{k,n}^{2,t} = \frac{1}{2k} \sum_{i=1}^k |\varphi(V_i^t)| \delta^t_i e^{\Delta_i^t}\left[B_{ik}^t + C_{ik}^t\right]^2,\label{skn2}\end{equation}
(compare with $S_{n2}$ from Lemma 2.7 \cite{stute}).
\begin{lemma}
We have that
\begin{equation}\label{s2kn}S_{k,n}^{2,t} = \overline{O}_{P}(k^{-1}).\end{equation}
\end{lemma}
\begin{proof}
In fact, \eqref{bin} and \eqref{ik5} imply $B_{ik}^t = \overline{O}_{P}(k^{-1})$. Let us show that $C_{ik}^t = \overline{O}_{P}(k^{-1/2}).$ From \eqref{cint} and \eqref{ui} it follows
\[C_{ik}^t = \int_{0}^{U_i-} \frac{1}{1 - Q_k(x)} Q^{0}_k(\dd x) - \int_{0}^{U_i-} \frac{1}{1 - x} Q^{0}(\dd x),\] where
\[ Q^{0}_k(x) = \frac{1}{k}\sum_{i=1}^k I(U_i\le x, \delta_i^t = 0)\quad \text{ and } \quad Q^0(x) = \pr (U_1\le x, \delta_1^t = 0).\]
We have,
\begin{align*}
C_{ik}^t &\le \left|\int_{0}^{U_i-} \left(\frac{1}{1 - Q_k(x)} - \frac{1}{1 - x}\right) Q^{0}_k(\dd x)\right| + \left|\int_{0}^{U_i-} \frac{1}{1 - x} Q^{0}_k(\dd x) - \int_{0}^{U_i-} \frac{1}{1 - x} Q^{0}(\dd x)\right|\\
&=: C_{ik}^{t1} + C_{ik}^{t2}.
\end{align*}
Note that we can consider only $C_{ik}^t$ with $U_i \le H^t(T)$ in \eqref{skn2} due to Condition~\ref{cond:bounded_phi}, where $H^t(T) \le T_0$ for large $t$ and some $T_0\in(0,1)$ by \eqref{ik5}. Thus, by Kolmogorov's theorem on the convergence rate of the supremum of the absolute difference between empirical and true cdfs (see, e.g., Theorems 1 and 2 in~\cite{kolmogorov}) and the fact that $(1 - x)/(1 - Q_k(x))$ is stochastically bounded from above on $x<U_{k,k}$ (see p. 415, \cite{shorackwellner}), for large $k$
\[\sup_{x\in [0, U_i)} \left|\frac{1}{1 - Q_k(x)} - \frac{1}{1 - x}\right| \le \frac{\sup_{x\in [0, T_0]} |Q_k(x) - x|}{(1 - Q_k(T_0))(1 - T_0)} = \overline{O}_{P}(k^{-1/2})\]
since neither $Q_k(x)$ nor $x$ depend on $t.$ This implies $C_{ik}^{t1} = \overline{O}_{P}(k^{-1/2})$. 

Let us analyze $C_{ik}^{t2}.$ The main difficulty is that $Q^0_k$ and $Q^0$ depend on $t$ via $\{\delta^t_i\}.$ First, let us equivalently represent the integrals in $C_{ik}^{t2}$ as integrals with respect to proper cdfs (and not only by sub-distribution functions). Consider the auxiliary random variables $\bar U_i, i = 1,...,k,$ where
\[\bar U_i = U_i (1 - \delta_i^t) + \delta_i^t(1 - \varepsilon_i),\] and $\{\varepsilon_i\}_{i=1}^k$ are independent random variables, uniformly distributed on $[0,\bar\varepsilon]$ with $\bar\varepsilon<1 - T_0$ and independent of $\{U_i\}_{i=1}^k.$ Denote the cdf and empirical cdf of $\{\bar U_i\}_{i=1}^k$ by $\bar Q(x)$ and $\bar Q_k(x),$ respectively. Then, for all $x\in [0, T_0]$
\[ Q^{0}(x) = \bar Q(x) \quad \text{ and } \quad  Q^{0}_k(x) = \bar Q_k(x).\] Next, let a non-increasing step function $f_k:[0,1)\to [1,\infty)$ and a partition $\{\Delta_j\}_{j=1}^\infty$ of the interval $[0,1)$ be such that
\[f_k(x) = 1 + \frac{j-1}{k} \quad \text{and} \quad 0\le\frac{1}{1-x} - f_k(x) < \frac{1}{k} \quad \text{for } x\in \Delta_j,\; j\in\amsmathbb{N}.\]
Clearly, $f_k$ and $\{\Delta_j\}_{j=1}^\infty$ are defined uniquely by the latter. Then for every $j\in\amsmathbb{N}$ such that $\Delta_j \cap [0,T_0] \neq \emptyset$
\begin{eqnarray} \nonumber
\left|\int_{\Delta_j} \frac{1}{1-x}\big(\bar Q_k(\dd x) - \bar Q(\dd x)\big)\right| &\le& \left|\int_{\Delta_j} f_k(x)\big(\bar Q_k(\dd x) - \bar Q(\dd x)\big)\right| + \frac{1}{k} \left|\int_{\Delta_j}\big(\bar Q_k(\dd x) - \bar Q(\dd x)\big)\right|\\
\label{newsight}&\le& 2 (C+k^{-1}) \sup_{x\in [0,1]} \big|\bar Q_k(x) - \bar Q(x)\big|,
\end{eqnarray} where $C = k f_k(T_0).$ 
Denote $x_{j-1}$ and $x_j,$ the endpoints of $\Delta_j,$ $j\in \amsmathbb{N},$ and notice that $x_0 = 0.$ By \eqref{newsight} we have
\begin{eqnarray*} \int_0^{U_i-} \frac{1}{1-x}\big(\bar Q_k(\dd x) - \bar Q(\dd x)\big) &=& \sum_{j=1}^{kf_k(U_i-)} \Big(1 + \frac{j-1}{k}\Big)\big(\bar Q_k(x_j) - \bar Q_k(x_{j-1}) - \bar Q(x_j) + \bar Q(x_{j-1}) \big)\\ &&+ kf_k(U_i-) \cdot O_P(1/k) \sup_{x\in [0,T_0]} \big|\bar Q_k(x) - \bar Q(x)\big|\\
\\ &=& f_k(U_i-) \big(\bar Q_k(x_m) - \bar Q(x_m)\big)  + O_P(1) \sup_{x\in [0,T_0]} \big|\bar Q_k(x) - \bar Q(x)\big|\\&& + \frac{1}{k}\sum_{j=1}^{kf_k(U_i-)} \big(\bar Q_k(x_j) - \bar Q_k(x_{j-1})\big)\\
&=& O_P(1) \sup_{x\in [0,T_0]} \big|\bar Q_k(x) - \bar Q(x)\big|,
\end{eqnarray*} where $m$ denotes $kf_k(U_i-),$ the final relation holds since $f_k(U_i-) \le f_k(T_0) < \infty,$ and $O_P(1)$ in the last line does not depend on $t.$ By another application of Kolmogorov's theorem,
\[ \sup_{x\in[0, T_0]} |\bar Q_k(x) - \bar Q(x)| = \overline{O}_{P}(k^{-1/2}),\] since the latter supremum does not depend on $t$ in distribution. This proves $C_{ik}^t = \overline{O}_{P}(k^{-1/2})$.\\

Next, it follows from the asymptotics of $B_{ik}^t$ and $C_{ik}^t$ that $e^{\Delta_i} = \overline{O}_{P}(1)$. Finally,
\[\frac{1}{k} \sum_{i=1}^k |\varphi(V_i^t)| \delta^t_i \le \frac{1}{k} \sum_{i=1}^k |\varphi(V_i^t)| \gamma_0^t(V_i^t)\delta^t_i = \overline{O}_{P}(1).\] Summarizing the above, we have found that $S^{2,t}_{k,n} = \overline{O}_{P}(k^{-1})$, as desired. \end{proof}

Piecing all the above asymptotics together, that is appealing to \eqref{full}, \eqref{c24}, \eqref{rn1}, \eqref{skn1}, and \eqref{s2kn} completes the proof of Theorem~\ref{theo:Rkn_term}.
\end{proof}

We proceed to prove Theorem~\ref{T4}, that is in fact the general version of Theorem~\ref{theo:Rkn_term}, by lifting the truncation Condition \ref{cond:bounded_phi}. In the proof of this extension, we require analogues of (1.5) and (1.6) in~\cite{stute}, which help to control the tail. For that purpose we require the following lemma.

\begin{lemma}\label{P1}
Assume Condition~\ref{envelope_moment}. Then for large $t$
\begin{equation}\label{1.5}\int_1^\infty \varphi^2 (\gamma^t_0)^2  \dd H^{1,t} <\infty\end{equation}
and
\begin{equation}\label{1.6} \int_{1}^\infty |\varphi|\sqrt{C^t} \dd F^t < \infty,\end{equation}
{ where \[C^t(x) = \int_{1}^x \frac{G^t(\dd y)}{(1 - H^t(y))(1 - G^t(y))}.\]}
\end{lemma}

\begin{proof}[Proof of Lemma~\ref{P1}]
By eventual continuity of $F$ and $G,$ 
\[\int_1^\infty \varphi^2(x) (\gamma^t_0(x))^2 H^{1,t}(\dd x) = \int_1^\infty \varphi^2(x) \frac{F^t(\dd x)}{1 - G^t(x)}.\] By the regular variation property of $F$ and $G,$ and the argument similar to that used to prove \eqref{P0}, under Condition~\ref{envelope_moment} we derive that
\begin{equation}\label{pre2.5}\int_1^\infty \varphi^2(x) (\gamma^t_0(x))^2 H^{1,t}(\dd x) \le C\int_1^\infty \varphi^2(x) x^{1/\gamma_G - 1/\gamma_F - 1+\varepsilon} \dd x <\infty\end{equation} for some $C>0$ and large $t,$ so \eqref{1.5} is satisfied. Next, it is easy to show that $C^t(x)$ is regularly varying with index $1/\gamma_H,$ thus by the same argument for some $C>0$ and large $t$
\begin{eqnarray}\label{pre2.6}\int_{1}^\infty |\varphi(x)|\sqrt{C^t(x)} F^t(\dd x) &\le& C \int_{1}^\infty |\varphi(x)|\, x^{(2\gamma_G)^{-1} - (2\gamma_F)^{-1} - 1 + \varepsilon/2} \dd x\\ &\le&C \left(\int_1^\infty \varphi^2(x) x^{1/\gamma_G - 1/\gamma_F - 2 + \varepsilon} \dd x\right)^{1/2} < \infty,\nonumber\end{eqnarray} and so \eqref{1.6} is satisfied as well.
\end{proof}

 The main steps of the proof of Theorem~\ref{T4} use the last steps of the proof of Theorem 1.1 in~\cite{stute}. First, for a given $\epsilon>0$ let us choose a continuous function $\tilde \varphi$ vanishing outside of $[1, T],$ for some $T<\infty,$ and satisfying Condition~\ref{envelope_moment}, such that
\begin{equation}\label{onlyeps} \int_1^\infty (\varphi(x) - \tilde \varphi(x))^2 x^{\alpha(\varepsilon)} \dd x \le \epsilon, \end{equation} which is possible under Condition~\ref{envelope_moment}. Indeed, it suffices to choose $\tilde \varphi(x) = \varphi(x) I(x\le T),$ $x\ge 1.$
Note that by \eqref{onlyeps} it is possible to find such $\epsilon_1>0,$ vanishing together with $\epsilon$, such that
\begin{equation}\label{2.5}\int_1^\infty(\varphi - \tilde \varphi)^2 (\gamma^t_0)^2 \dd H^{1,t} \le \epsilon_1\end{equation} and
\begin{equation}\label{2.6}\int_1^\infty|\varphi - \tilde\varphi|\sqrt{C^t} \dd F^t \le \epsilon_1\end{equation} for large $t$ similarly to \eqref{pre2.5} and \eqref{pre2.6}, respectively.

Set $\varphi_1 = \varphi - \tilde \varphi.$ We prove that
\begin{equation}\label{trick}k^{1/2} \left(\int_1^\infty \varphi_1 \dd \amsmathbb{F}^t_{k,n} - \int_{1}^\infty \varphi_1 \dd F^t\right) = \overline{O}_{P}(\epsilon_1^{1/2}),\end{equation} which together with Theorem~\ref{theo:Rkn_term} and the bound 
\[\int_{1}^\infty |\varphi_1|\dd F^t < \epsilon_1\] following from \eqref{2.6} implies the conclusion of Theorem~\ref{T4}. Similarly to the first formula on p.436 in~\cite{stute}, we get 
\begin{align*}
k^{1/2}\int_1^\infty \varphi_1\, \dd ( \amsmathbb{F}^t_{k,n} -  F^t) &= k^{-1/2} \sum_{i=1}^k \left[\varphi_1(V_i^t)\delta_i^t \gamma_0^t(V_i^t) - \int_1^\infty \varphi_1 \dd F^t\right]\\
&\quad + k^{-1/2}  \sum_{i=1}^k \varphi_1(V_i^t)\delta_i^t \gamma_0^t(V_i^t) \big[\exp(B_{ik}^t + C_{ik}^t) - 1\big].\end{align*}
By \eqref{2.5}, the variance of the first summand in the right-hand side of the latter relation is not more than $2\epsilon_1,$ thus by Chebyshev's inequality the first summand is $\overline{O}_{P}(\epsilon_1^{1/2}).$ By the trivial bound $|\exp(x) - 1| \le |x| \exp(|x|),$ the second summand is bounded in absolute value by
\[k^{-1/2} \sum_{i=1}^k |\varphi_1(V_i^t)|\delta_i^t \gamma_0^t(V_i^t) (|B_{ik}^t| + |C_{ik}^t|) \exp\big[|B_{ik}^t| + |C_{ik}^t|\big].\]
By \eqref{bin},
\[|B_{ik}^t| \le \frac{1}{2k} \int_1^{V_i^t-} \frac{H^{0,t}_k(\dd z)}{(1 - H^t_k(z))^2} \le \frac{1}{2k} \int_1^{V_i^t-} \frac{H^{t}_k(\dd z)}{(1 - H^t_k(z))^2} \le \frac{1}{k(1 - H^t_k(V^t_{k,k}-))}\le 1.\] Now let us prove that \begin{equation}\label{supchik}\sup_{1\le i\le k} |C_{ik}^t| = \overline{O}_{P}(1)\end{equation} as well. 
Indeed, 
\[\sup_{1\le i\le k} |C_{ik}^t| \le \sup_{z \in [1, V^t_{k,k})} \left|\int_{1}^z \frac{H^{0,t}_k(\dd x)}{1 - H^t_k(x)} - \int_{1}^z \frac{H^{0,t}(\dd x)}{1 - H^t(x)}\right|.\] Theorem 2.1 in~\cite{zhou} yields an upper bound on the exceedance probability of the supremum on the right-hand side of the latter relation over some level $a$ that depends on $a$ alone (concretely, the bound can be taken to be $358\cdot a^{-2/3}$) and does not depend on distributional characteristics, which implies \eqref{supchik}.

Thus, since we proved $\sup_{1\le i\le k} |B_{ik}^t|+|C_{ik}^t| = \overline{O}_{P}(1),$ it remains to bound the following term from above
\[k^{-1/2} \sum_{i=1}^k |\varphi_1(V_i^t)|\delta_i^t \gamma_0^t(V_i^t) (|B_{ik}^t| + |C_{ik}^t|),\] which can be done through the same reasoning as in the corresponding part of the proof of Theorem 1.1 in~\cite{stute} replacing the assumption (2.10) therein with \eqref{2.6}, and obtaining a bound which is equal to a constant times $\epsilon_1$, which in particular is $\overline{O}_{P}(\epsilon_1)=\overline{o}_{P}(\epsilon_1^{1/2})$.
\end{proof}

\begin{proof}[Proof of Theorem~\ref{Tconsistency}]
We use the proof technique from Theorem 2.1 in~\cite{segers}. By Theorem~\ref{T4}, we have
\begin{equation}S_{k,n}(\varphi)  =  \frac{1}{k} \sum_{j=1}^k W_j + r_{k,n},\label{reformulation}\end{equation} where $\{W_j\}_{j=1}^k$ are i.i.d. given $Z_{n-k,n} = t,$ $\sqrt{k} r_{k,n}$ equals $\sqrt{k} r^t_{k,n}$ in distribution given $Z_{n-k,n} = t,$ and  $\sqrt{k} r^t_{k,n} = \overline{o}_{P}(1).$ Thus the last term in \eqref{reformulation} does not affect the asymptotic of the EKM integral. Note that by \eqref{trick}, we can consider $\tilde \varphi = \varphi I(\cdot \le T_0)$ instead of $\varphi$ in \eqref{reformulation} for some large $T_0,$ since $\varphi_1 = \varphi - \tilde \varphi$ does not affect the asymptotic of $S_{k,n}(\varphi)$ by Theorem~\ref{T4}.

Denote
\begin{equation}\label{wast}W^t = \varphi(V^t)\gamma_0^t(V^t)\delta^t + \gamma_1^t(V^t)(1 - \delta^t) - \gamma_2^t(V^t)\end{equation} with $(V^t, \delta^t) \stackrel{d}{=} (V_1^t, \delta_1^t)$ and consider the characteristic function of the normalized sum in \eqref{reformulation} given $Z_{n-k,n} = t.$ By \eqref{T4new}, we have,
\begin{eqnarray*}b_k^t(s) &:= & \amsmathbb{E}\left[\left.\exp\Big(\frac{is}{k}\sum_{j=1}^k W_j\Big)\right|Z_{n-k,n} = t\right]\\
 & = &  \left(\amsmathbb{E}\left[\left.\exp\Big(\frac{is}{k} W_1\Big)\right|Z_{n-k,n} = t\right]\right)^k =\left(\amsmathbb{E}\left[\exp\Big(\frac{is}{k} W^t\Big)\right]\right)^k\end{eqnarray*}
By properties of characteristic functions (more precisely, here we use Lemma A.1, \cite{segers}),
\begin{equation}\amsmathbb{E}\big[\exp(iz\, W^t)\big] = 1 + iz \amsmathbb{E} W^t + \eta(z, t),\label{characteristic1}\end{equation} with $\eta(z,t) \le 2|z|^{1+\delta} \amsmathbb{E}|W^t|^{1+\delta}$ and some $\delta>0$ which we may specify at a later stage. As mentioned above, replacing $\varphi$ with $\tilde \varphi$ does not change the asymptotic behavior of $S_{k,n}(\varphi),$ however, the boundedness of $\amsmathbb{E}|W^t|^{1+\delta}$ for large $t$ is trivial after such a replacement. Moreover, applying the argument used to prove \eqref{P0}, we derive that $\amsmathbb{E}|W^t|^{1+\delta}$ is bounded uniformly in $t\ge T.$
Thus, for $t$ larger than some $T > 1,$ $\eta(z,t)$
is bounded by $m|z|^{1+\delta}$ for some positive constant $m$ independent of $t.$ 
But then for every $s\in \amsmathbb{R}$
\begin{equation}\label{2.1}b_k^t(s) \to \exp(is\, \amsmathbb{E} W^t) \text{ as } k\to\infty \text{ uniformly in } t\ge T.\end{equation}
Now, it is straightforward to see that $\amsmathbb{E} \big[\gamma_1^t(V^t)(1 - \delta^t)\big] = \amsmathbb{E} \gamma^t_2(V^t),$ and thus by Condition~\ref{envelope_moment}, the argument used for proving (\ref{P0}), and dominated convergence theorem
\begin{equation}\amsmathbb{E} W^t = \amsmathbb{E} \big[\varphi(V^t)\gamma_0^t(V^t) \delta^t\big] = \int_1^\infty \varphi(x) F^t(\dd x) \to \int_1^\infty \varphi(x) F^\circ(\dd x), \quad t\to\infty.\label{2.2}\end{equation}
Finally, we can write
\[\amsmathbb{E}\left[\exp\Big(\frac{is}{k}\sum_{j=1}^k W_j\Big)\right] = \amsmathbb{E}[b_k^{Z_{n-k,n}}(s) I(Z_{n-k,n}< T)] + \amsmathbb{E}[b_k^{Z_{n-k,n}}(s) I(Z_{n-k,n}\ge T)],\]
where the first term vanishes since $Z_{n-k,n} \stackrel{P}{\to} \infty$ and $|b_k^t(s)| \le 1,$ and the second term converges to $\exp(is\, S_\circ(\varphi))$ by \eqref{2.1} and \eqref{2.2}. This completes the proof.
\end{proof}

\begin{proof}[Proof of Theorem~\ref{T1an}] The ideas of the proofs of Theorems~\ref{Tconsistency} and \ref{T1an} are similar. By Theorem~\ref{T4}, $r_{k,n}$ in the decomposition \eqref{reformulation} does not play a role in the asymptotic behavior of $\sqrt{k} S_{k,n}(\varphi),$ see also the first argument in the proof of Theorem~\ref{Tconsistency}. Denote for brevity
\[S_t(\varphi) = \int_{1}^\infty \varphi \dd F^t, \] 
and consider the characteristic function of the left-hand side of \eqref{main} minus $\sqrt{k}r_{k,n}$ given $Z_{n-k,n} = t.$ %which is $\overline{o}_{P}(1)$ given $Z_{n-k,n} = t.$
We have
\begin{align}\label{ckts}c_k^t(s) &:= \amsmathbb{E}\left[\left.\exp\Big(i s \,\sqrt{k}\Big[\frac{1}{k}\sum_{i=1}^k W_i - \int \varphi\, \dd F^{Z_{n-k,n}}\Big]\Big)\right|Z_{n-k,n} = t\right]\\ &= \left(\amsmathbb{E}\Big[\exp\Big(\frac{i s}{\sqrt{k}}\big[W^t - S_t(\varphi)\big]\Big)\Big]\right)^k,\nonumber %=: c^t_k(s) (1 + o(1)), 
\end{align} where $W^t$ and $\{W_i\}_{i=1}^k$ were defined in the proof of Theorem~\ref{Tconsistency}. By \eqref{2.2}, $\amsmathbb{E}[W^t - S_t(\varphi)] = 0,$ therefore by another application of Lemma A.1, \cite{segers},
\begin{equation}\label{lemmaa1}\amsmathbb{E}\big[\exp(iz\, [W^t - S_t(\varphi)])\big] = 1 - \frac{z^2}{2} \amsmathbb{E} [(W^t - S_t(\varphi))^2] + \zeta(z, t),\end{equation} with $|\zeta(z,t)| \le |z|^{2+\delta} \amsmathbb{E}|W^t - S_t(\varphi)|^{2+\delta}$ and some $\delta>0$ which can be chosen depending on the values of $\gamma_G$ and $\varepsilon$ according to Condition~\ref{envelope_moment}. As in the proof of Theorem~\ref{Tconsistency}, let us show that \begin{equation}\label{2+delta} \sup_{t\ge T}\amsmathbb{E}|W^t - S_t(\varphi)|^{2+\delta} <\infty\end{equation} for some $T\ge 1$ and $\delta>0.$ Indeed, by Minkowski's inequality,
\begin{eqnarray*}\nonumber\left(\amsmathbb{E}|W^t - S_t(\varphi)|^{2+\delta}\right)^\frac{1}{2+\delta} \le |S_t(\varphi)| + \left(\amsmathbb{E}|W^t|^{2+\delta}\right)^\frac{1}{2+\delta}, 
\end{eqnarray*} where $|S_t(\varphi)|$ is bounded uniformly for large $t$ by slight modification of \eqref{2.2} and $\amsmathbb{E}\big[|W^t|^{2+\delta}\big]$ can be handled similarly to the proof of Theorem~\ref{Tconsistency}.

Therefore, returning to \eqref{ckts}, we derive that for every $s\in \amsmathbb{R}$,  as $k\to\infty$  uniformly in $t\ge T$,
\begin{equation*}c_k^t(s) \to \exp(-s^2\amsmathbb{E} [(W^t - S_t(\varphi))^2] /2).\end{equation*} By \eqref{2+delta}, the argument used for proving \eqref{P0}, and the dominated convergence theorem, we have
\[\amsmathbb{E} [(W^t - S_t(\varphi))^2] \to \amsmathbb{E} [(W^\circ - S_\circ(\varphi))^2] = {\rm var}(W^\circ), \quad t\to\infty.\]
Repeating the last argument in the proof of Theorem~\ref{Tconsistency} completes the proof.
\end{proof}

\begin{proof}[The result mentioned in Remark \ref{Rreplacement} and its proof] 

In the next result, we replace in \eqref{main} the random centering sequence $\int \varphi \dd F^{Z_{n-k,n}}$ with its non-random counterpart $\int \varphi \dd F^{U_H(n/k)}.$ %, where $U_H$ is the tail quantile function of $H,$ i.e. 
%\[U_H(x) = \inf\{y: H(y) \ge 1 - 1/x\}, \quad x>1.\]
For this purpose, we need to impose the following condition on the tail quantile functions $U_F$ and $U_H:$ let a tail quantile function $U$ possess a derivative $U^\prime$ and let the following von Mises' convergence hold
\begin{equation}\label{prime}\lim_{t\to\infty} \frac{t U^\prime(t)}{U(t)} = \gamma,\end{equation}
which in itself is sufficient for $U$ to be regularly varying with tail index $\gamma$, and is standard in extreme value
theory (see e.g. Theorem 1.1.11 and Corollary 1.1.12 in~\cite{dehaan}). Assume also that $\varphi$ has a continuous and eventually monotone derivative $\varphi^\prime.$ Then it follows from Condition~\ref{envelope_moment}  and integration by parts, that for some $\varepsilon^\prime>0$
\begin{equation}\label{cond:phi_deriv}
 \int_1^\infty |\varphi^\prime(x)|\,x^{-1/\gamma_F + \varepsilon^\prime} \dd x < \infty.
\end{equation}
We use this relation in the proof of the theorem below. Finally, we need the following
\begin{definition}\label{def:envelope}
Consider a function $f: A \to \amsmathbb{R}$ with $A\subseteq \amsmathbb{R}.$ We call a continuous, positive and eventually monotone function $\phi: A\to \amsmathbb{R}$ an envelope of $f,$ if $\phi(x) \ge |f(x)|,$ $x\in A.$
\end{definition}
\noindent Note that if $f$ is continuous, nonzero and eventually monotone then $|f|$ is its envelope.

%Also, we need the following technical condition.

%\begin{condition}\label{cond:phi_deriv}
%Let $\varphi$ have a derivative $\varphi^\prime$ such that for some $\varepsilon^\prime>0$ 
%\begin{equation}
%    \label{onlyconditionnew} \int_1^\infty \bar{\varphi^\prime}(x)\,x^{-1/\gamma_F + \varepsilon^\prime} \dd x < \infty.
%\end{equation}
%\end{condition}

%\noindent Hereinafter $\bar{\varphi^\prime}$ denotes an envelope of $\varphi^\prime.$ If the behavior of $\varphi^\prime$ is regular enough, in particular, if $\varphi^\prime$ is { continuous and eventually monotone (in this case, $\bar{\varphi^\prime}$ can be chosen to be equal to $|\varphi^\prime|$ for large $x$),} then Condition~\ref{cond:phi_deriv} does not pose any additional complications, since it then follows from Condition~\ref{envelope_moment} { and integration by parts.}

\begin{theorem}\label{T2an} Assume the conditions of Theorem~\ref{T4}. Furthermore, assume $U_F$ and $U_H$ satisfy \eqref{prime}, and $\varphi$ has a continuous and eventually monotone derivative $\varphi^\prime.$ 
%satisfies Condition~\ref{cond:phi_deriv}. 
Then 
\begin{equation}\label{T2anassertion}\sqrt{k}\int \varphi\, \dd (F^{Z_{n-k,n}} - F^{U_H(n/k)}) \stackrel{P}{\to} 0.\end{equation}
\end{theorem}

\noindent As an immediate consequence of Theorems~\ref{T1an} and \ref{T2an} we get the following result.
\begin{corollary}\label{R1} Under the assumptions of  Theorem~\ref{T2an}, we have
\begin{equation}\sqrt{k}\int \varphi\, \dd(\amsmathbb{F}_{k,n} - F^{U_H(n/k)}) \stackrel{d}{\to} N(0, \sigma^2_\varphi).\label{seq}\end{equation}
\end{corollary}

\begin{proof}[Proof of Theorem \ref{T2an}] Let $\xi$ be a Pareto$(1)$-distributed random variable, independent of $\{Z_{i}\}_{i=1}^n.$ Without loss of generality, let us consider only $t$ such that $F$ is strictly monotone in the left neighborhood of $t.$ It is straightforward to verify that $F^t$ is the cdf of the random variable $U_F(s \xi )/U_F(s),$ where $s = U_F^{\leftarrow}(t) = (1 - F(t))^{-1},$ see also \eqref{uft}. Indeed,
\begin{eqnarray}\label{conjugate}\pr(U_F(s \xi)/U_F(s) \le x) = \pr\big(\xi \le s^{-1} U_F^{\leftarrow}(x t)\big) = 1 - \frac{1 - F(xt)}{1 - F(t)} = F^t(x).\end{eqnarray}
Denote \[\mu(s) = \amsmathbb{E}\,\varphi\left(\frac{U_F(s\xi)}{U_F(s)}\right).\] Thus, \eqref{T2anassertion} can be rewritten as
\[\sqrt{k}\big(\mu(U_F^{\leftarrow}(Z_{n-k,n})) - \mu(U_F^{\leftarrow}(U_H(n/k)))\big)\stackrel{P}{\to} 0,\] which is in accordance with relation (4.6) in~\cite{segers}. Denote for brevity $u(s, \xi) = U_F(s\xi)/U_F(s).$ By the mean value theorem, 
\begin{align}\label{meanvalue}|\mu(x) - \mu(y)| \le \amsmathbb{E}\left[\left|\varphi^\prime(x^\ast) \big(u(x, \xi) - u(y, \xi)\big)\right| \right],\end{align}
where $x^\ast$ lies between $u(x, \xi)$ and $u(y, \xi).$ 

Consider the asymptotic behavior of the expression within the modulo on the right-hand side of the latter relation.
Recall that a regularly varying function $U$ with index $\gamma$ is {\it normalized} if there exists a $T>1$ and a real measurable
function $\gamma(\cdot)$ on $[T;\infty)$ converging to $\gamma$, such that for $t\ge T,$
\begin{equation} U(t) = U(T) \exp\left(\int_T^t \gamma(u) \frac{\dd u}{u}\right).\label{normalised}\end{equation}
Clearly, the latter is satisfied for $U_F$ with $\gamma(t) = t U_F^\prime(t)/U_F(t)$ under the condition \eqref{prime}, and thus the conditions of Lemma 4.3 in~\cite{segers} hold. Also note that it is possible to select an envelope $\phi$ of $\varphi^\prime$ such that it satisfies \eqref{cond:phi_deriv}.
%Also note that we can choose an envelope $\bar{\varphi^\prime}$ such that  $\bar{\varphi^\prime}(x) \neq 0,$ $x\ge 1.$
Then by the first assertion of Lemma 4.3 in~\cite{segers}, % and properties of the envelope $\bar{\varphi^\prime}$,
we get
\begin{equation}|\varphi^\prime(x^\ast)| \le C \max(\phi(\xi^{\gamma_F-\varepsilon^\prime}), \phi(\xi^{\gamma_F+\varepsilon^\prime}))\label{segerslemma4.3a}\end{equation}
for every $\varepsilon^\prime>0,$ some $C>0$ independent of $\varepsilon^\prime,$ and large $x$ and $y$.
Next, consider the difference within the modulo on the right-hand side of \eqref{meanvalue}. By the relation (4.9) in Lemma 4.3, \cite{segers}, we get for the above $\varepsilon^\prime$ and large $x$ and $y$
\begin{equation}\label{segerslemma4.3}\Big|u\big(x, \xi\big) - u\big(y, \xi\big)\Big| \le 2\varepsilon^\prime |\log(x/y)| \xi^{\gamma_F+\varepsilon^\prime}.\end{equation}
Let $\{\xi_i\}_{i=1}^n$ be independent Pareto$(1)$-distributed random variables, and $\xi_{1,n}\le \ldots \le \xi_{n,n}$ be the order statistics of this sample such that $Z_{n-k,n} = U_H(\xi_{n-k,n}).$
Since the function $U_F^\leftarrow(U_H(x))$ tends to infinity as $x\to\infty,$ we derive by \eqref{meanvalue}, \eqref{segerslemma4.3a}, and \eqref{segerslemma4.3} for large $n$
\begin{align*}
&\sqrt{k}\Big|\mu(U_F^{\leftarrow}(Z_{n-k,n})) - \mu(U_F^{\leftarrow}(U_H(n/k)))\Big|\\ 
&\le 2\varepsilon^\prime C \sqrt{k} \left|\log\left(\frac{U_F^{\leftarrow}(U_H(\xi_{n-k,n}))}{U_F^{\leftarrow}(U_H(n/k))}\right) \right|\amsmathbb{E}\Big[\xi^{\gamma_F+\varepsilon^\prime} \max\big(\phi(\xi^{\gamma_F-\varepsilon^\prime}), \phi(\xi^{\gamma_F+\varepsilon^\prime})\big)\Big].
\end{align*}
The mathematical expectation in the latter relation is bounded for small $\varepsilon^\prime$ by \eqref{cond:phi_deriv}. Thus, if we show that \begin{equation}\label{biglog}\sqrt{k} \log\left(\frac{U_F^{\leftarrow}(U_H(\xi_{n-k,n}))}{U_F^{\leftarrow}(U_H(n/k))}\right)  = O_{\mathbb{P}}(1)\end{equation} then the proof is complete by letting $\varepsilon^\prime \to 0.$

But Smirnov's lemma, \cite{smirnov}, implies that $\xi_{n-k,n}$ is asymptotically normal as
\begin{equation}\frac{k^{3/2}}{n}(\xi_{n-k,n} - n/k) \stackrel{d}{\to} N(0,1).\label{smirnov}\end{equation} Notice also that the condition \eqref{prime} is equivalent to the following
\begin{equation*}\lim_{t\to\infty} \frac{t K^\prime(t)}{1-K(t)} = \frac{1}{\gamma},\end{equation*} where $U$ is the tail quantile function of cdf $K.$ Then by the latter and \eqref{prime}, the derivative of the function $f(x) = \log (U_F^{\leftarrow}(U_H(x))) = -\log(1 - F(U_H(x)))$ possesses the following property
\[\lim_{x\to\infty} x f^\prime(x) = \lim_{x\to\infty}\frac{U_H(x) F^\prime(U_H(x))}{1 - F(U_H(x))} \cdot \frac{x U_H^\prime(x)}{U_H(x)} = \frac{\gamma_H}{\gamma_F},\]
therefore we derive by \eqref{smirnov} and the delta-method,
\[\sqrt{k}\, \frac{\gamma_F}{\gamma_H}\log\left(\frac{U_F^{\leftarrow}(U_H(\xi_{n-k,n}))}{U_F^{\leftarrow}(U_H(n/k))}\right) \stackrel{d}{\to} N(0, 1).\] Hence \eqref{biglog} follows, which completes the proof.
\end{proof} \end{proof}

\begin{proof}[Proof of Theorem~\ref{T3an}] The proof technique is inspired by Theorem 4.5 in~\cite{segers}.
Using the notation of the proof of Theorem~\ref{T2an}, we get
\[\sqrt{k}\int \varphi \dd(F^{Z_{n-k,n}} - F^\circ) = \sqrt{k}\big(\mu(U_F^{\leftarrow}(Z_{n-k,n})) - \mu(\infty)\big),\] where
\[\mu(\infty):=\amsmathbb{E} \varphi(\xi^{\gamma_F}) = \int\varphi \dd F^\circ.\] Denoting $u(\infty, \xi) = \xi^{\gamma_F},$ we have by \eqref{meanvalue}
\begin{equation}\label{meanvalue1}\mu(x) - \mu(\infty) = \amsmathbb{E}\left[\varphi^\prime(x^\ast) \big(u(x, \xi) - u(\infty, \xi)\big) \right],\end{equation} with $x^\ast$ lying between $u(x, \xi)$ and $u(\infty, \xi).$ By Potter's bounds, for every $\varepsilon^\prime>0$ and large $t,$
\[\frac{1}{2}x^{\gamma_F-\varepsilon^\prime} \le \frac{U_F(tx)}{U_F(t)} \le 2 x^{\gamma_F+\varepsilon^\prime}.\] Since $\varphi^\prime$ is continuous and eventually monotone, the relation \eqref{cond:phi_deriv} holds. %for envelope $\bar{\varphi^\prime} = |\varphi^\prime|.$
Similar to \eqref{segerslemma4.3a}, we can find an envelope $\phi$ of $\varphi^\prime$ satisfying \eqref{cond:phi_deriv}, and then %and assuming w.l.o.g. that the envelope $\bar{\varphi^\prime}$ is positive (the case of $\bar{\varphi^\prime}(x) = 0$ is trivial because of eventual monotonicity of $\varphi^\prime$)
\begin{equation}|\varphi^\prime(x^\ast)| \le C \max\big(\phi(0.5\, \xi^{\gamma_F-\varepsilon^\prime}), \phi(2 \xi^{\gamma_F+\varepsilon^\prime})\big)\label{lemma4.6}.\end{equation} for some $C>0.$
Next, the conditions of Lemma 4.6, \cite{segers}, are satisfied under Condition~\ref{secondorder}, and therefore for some $K>0$ and large $x,$
\begin{equation}|u(x, \xi) - u(\infty, \xi)| = \left|\frac{U_F(x\xi)}{U_F(x)} - \xi^{\gamma_F}\right| \le K \xi^{\gamma_F+\varepsilon^\prime} |a(x)|.\label{lemma4.6a}\end{equation}

Piecing together (\ref{meanvalue1}-\ref{lemma4.6a}) yields
\begin{align}\nonumber
&\sqrt{k}\Big|\mu(U_F^{\leftarrow}(Z_{n-k,n})) - \mu(\infty)\Big|\\ 
&\le K C \sqrt{k}\, \big|a(U_F^{\leftarrow}(Z_{n-k,n}))\big|\,\amsmathbb{E}\Big[\xi^{\gamma_F+\varepsilon^\prime} \max\big(
\phi^(0.5\, \xi^{\gamma_F-\varepsilon^\prime}), \phi(2\xi^{\gamma_F+\varepsilon^\prime})\big)\Big],
\label{major}\end{align} for large $n,$ where the mathematical expectation in the latter relation, similarly to the analogous step in the proof of Theorem~\ref{T2an}, is bounded for small $\varepsilon^\prime$ by~\eqref{cond:phi_deriv}. Since the functions $a,$ $U_F^\leftarrow = (1 - F)^{-1}$ and $U_H$ are regularly varying and $\xi_{n-k,n}/(n/k) \stackrel{P}{\to} 1$ by Smirnov's Lemma~\eqref{smirnov}, we have
\begin{equation}\label{aratio}\frac{a(U_F^{\leftarrow}(Z_{n-k,n}))}{a(U_F^{\leftarrow}(U_H(n/k)))} = \frac{a(U_F^{\leftarrow}(U_H(\xi_{n-k,n})))}{a(U_F^{\leftarrow}(U_H(n/k)))} \stackrel{P}{\to} 1,\end{equation} which completes the proof in the case $\lambda = 0.$

Along with this, for $\lambda\neq 0$ we get
\begin{equation} \label{soappl}\varphi^\prime(x^\ast) \stackrel{a.s.}{\to} \varphi^\prime(\xi^{\gamma_F}), \quad \frac{u(x, \xi) - u(\infty, \xi)}{a(x)} \stackrel{a.s.}{\to} \xi^{\gamma_F} h_\rho(\xi), \quad x\to\infty,\end{equation}
where the former follows by the continuous mapping theorem and %a.e.
continuity of $\varphi^\prime,$ and Condition~\ref{secondorder} implies the latter. Therefore, by (\ref{lemma4.6}-\ref{major}) and the dominated convergence theorem we derive
\[\frac{\mu(x) - \mu(\infty)}{a(x)} \to \amsmathbb{E} \big[\varphi^\prime(\xi^{\gamma_F}) \xi^{\gamma_F} h_\rho(\xi)\big] = C(\gamma_F, \rho), \quad x\to\infty.\]
Finally, by the latter convergence and \eqref{aratio}
\[\sqrt{k}\big(\mu(U_F^{\leftarrow}(Z_{n-k,n})) - \mu(\infty)\big) \stackrel{P}{\to} C(\gamma_F, \rho) \lambda.\] The proof is complete.
\end{proof}
\newpage
\section{Proofs of Section~\ref{genresidual}}\label{AppB}
\begin{proof}[Proof of Theorem~\ref{T1gen}]
The idea of the proof is similar to Theorem 2.1 and Theorem 4.1 in~\cite{segers}. Let us first prove consistency of $R^f_{k,n}.$ 
It is known that the random variables $\{\xi_{n-j+1,n}/\xi_{n-k,n}\}_{j=1}^k$ are independent of $\xi_{n-k,n}$ and distributed like $\{\xi^\ast_{k-j+1,k}\}_{j=1}^k,$ where $\xi_1^\ast, \ldots, \xi_{k}^\ast$ are independent Pareto$(1)$-distributed. Thus, we have for the characteristic function of $R^f_{k,n}$
\begin{eqnarray*}\amsmathbb{E}\big[\exp(isR^f_{k,n})\big] = \amsmathbb{E}\left[\exp\left\{\frac{is}{k}\sum_{j=1}^k\vartheta \Big(f\Big(\frac{\xi_{n-j+1,n}}{\xi_{n-k,n}}, \xi_{n-k,n}\Big)\Big)\right\}\right] = \amsmathbb{E}\big[a_k(\xi_{n-k,n})\big],
\end{eqnarray*}
where $s\in\amsmathbb{R}$ and
\[a_k(t) = \amsmathbb{E}\left[\exp\left\{\frac{is}{k}\sum_{j=1}^k\vartheta \big(f\big(\xi_j^\ast, t\big)\big) \right\}\right].\] Note that for fixed $t$, $\big\{\vartheta (f(\xi_j^\ast, t))\big\}_{j=1}^k$ are independent random variables distributed like $\vartheta (f(\xi, t)),$ therefore
\[a_k(t) = \left(\amsmathbb{E}\left[\exp\Big\{\frac{is}{k}\vartheta (f(\xi, t))\Big\}\right]\right)^k.\] Similarly to \eqref{characteristic1}, we get for $z\in \amsmathbb{R}$
\begin{equation*}\amsmathbb{E}\big[\exp(iz\, \vartheta (f(\xi, t)))\big] = 1 + iz \amsmathbb{E} \big[\vartheta (f(\xi, t))\big] + \eta(z, t),\end{equation*} with $\eta(z,t) \le 2|z|^{1+\delta} \amsmathbb{E}|\vartheta (f(\xi, t))|^{1+\delta}$ for $\delta>0$ from \eqref{cond4}. Repeating the argument from the proofs of Theorem \ref{T2an} and Theorem \ref{T3an}, we can find an envelope $\phi$ of $\vartheta$ satisfying \eqref{cond4}. Then by \eqref{cond2}, %continuity and eventual monotonicity of $\vartheta,$ assuming as before that $\bar\vartheta$ is positive,
we derive for $t\ge T$ with $T$ from Definition \ref{D5} \begin{equation}\label{b05}|\vartheta(f(\xi,t))|\le C \max\big[\phi(g_1(\xi)), \phi(g_2(\xi))\big],\end{equation} and thus by \eqref{cond4}
\begin{equation}\label{1plusdelta}\limsup_{t\to\infty} \amsmathbb{E}\big[|\vartheta (f(\xi, t))|^{1+\delta}\big]<\infty.\end{equation}
Summarizing the above, we have
\[a_k(t) \to \exp\big\{i s \amsmathbb{E}[\vartheta (f(\xi,t))]\big\}, \quad k\to\infty, \]
uniformly in $t\ge T.$
Note that by \eqref{cond1} and %a.e.
continuity of $\vartheta $, $\vartheta (f(\xi,t))\stackrel{a.s.}{\to} \vartheta (g(\xi)), t\to\infty,$ and therefore by \eqref{b05}, \eqref{cond4} and the dominated convergence theorem, \begin{equation}\label{muht}\mu_f(t) = \amsmathbb{E} \big[\vartheta (f(\xi, t))\big] \to \amsmathbb{E}[\vartheta (g(\xi))],\quad  t\to\infty.\end{equation} 
 Appealing to the latter and repeating the last argument in the proof of Theorem~\ref{Tconsistency} completes the proof of consistency.\\

Next, we prove \eqref{T1genassertion}. 
Let us consider the characteristic function of $\sqrt{k}\big(R^f_{k,n} - \mu_f(\xi_{n-k,n})\big).$ We have for $s\in\amsmathbb{R}$
\[\amsmathbb{E}\Big[\exp\big(is\sqrt{k}(R^f_{k,n} - \mu_f(\xi_{n-k,n}))\big)\Big] = \amsmathbb{E}\big[b_k(\xi_{n-k,n})\big],\]
where
\[b_k(t) = \amsmathbb{E}\left[\exp\left\{\frac{is}{\sqrt{k}}\sum_{j=1}^k\left[\vartheta \big(f\big(\xi_j^\ast, t\big)\big) - \mu_f(t)\right]\right\}\right].\]
Denote $Z(t) = \vartheta \big(f\big(\xi, t\big)\big) - \mu_f(t)$ and $Z_j^\ast(t) = \vartheta \big(f\big(\xi_j^\ast, t\big)\big) - \mu_f(t),$ $j=1,\ldots, k.$ As above, for fixed $t$ $\{Z_j(t)\}_{j=1}^k$ are independent zero-mean random variables distributed like $Z(t),$ and thus we can rewrite 
\begin{equation}\label{chararepr}b_k(t) = \amsmathbb{E}\left[\exp\left\{\frac{is}{\sqrt{k}}\sum_{j=1}^k Z_j^\ast(t)\right\}\right] = \left(\amsmathbb{E}\left[\exp\left\{\frac{is}{\sqrt{k}}Z(t)\right\}\right]\right)^k.\end{equation}
Similarly to \eqref{lemmaa1}, we get for $z\in \amsmathbb{R}$
\begin{equation}\label{chara}\amsmathbb{E}\big[\exp(iz\, Z(t))\big] = 1 - \frac{z^2}{2} \amsmathbb{E} [(Z(t))^2] + \zeta(z, t),\end{equation} with $|\zeta(z,t)| \le |z|^{2+\delta^\prime}\, \amsmathbb{E}|Z(t)|^{2+\delta^\prime}$ for some $\delta^\prime>0$. By application of \eqref{cond2}, \eqref{cond4} with $\delta^\prime = \delta-1$ and $\delta>1$ defined in Theorem~\ref{T1gen}, and \eqref{muht}, we get
\begin{equation*}\label{2plusdelta}\limsup_{t\to\infty} \amsmathbb{E}\big[|Z(t)|^{2+\delta^\prime}\big]<\infty.\end{equation*}
Piecing together \eqref{chararepr}, \eqref{chara} and the latter relation, we derive
\[b_k(t) \to \exp\big\{-\amsmathbb{E}[Z^2(t)] s^2/2\big\}, \quad k\to\infty,\]
uniformly in $t\ge T.$ Along with this, by an argument similar to that we used to prove \eqref{muht}, we get by the dominated convergence Theorem~\[\amsmathbb{E}[Z^2(t)] \to {\rm var}\big[\vartheta (g(\xi))\big], \quad t\to\infty.\] Repeating the last argument in the proof of Theorem~\ref{Tconsistency} completes the proof.

\end{proof}

\begin{proof}[Proof of Theorem~\ref{T2gen}]
The proof follows along the lines of Theorem~\ref{T3an}. Taking into account Theorem~\ref{T1gen}, it suffices to show that
\begin{equation}\label{T2gen2}\sqrt{k}\big(\mu_f(\xi_{n-k,n}) - \amsmathbb{E}[\vartheta (g(\xi))]\big) \stackrel{P}{\to}C\,\lambda,
\end{equation} Recall the notation $\mu_f(\infty) := \amsmathbb{E}[\vartheta (g(\xi))].$ Similarly to \eqref{meanvalue} and \eqref{meanvalue1}, we get
\begin{equation}\label{T2gen_1}\mu_f(t) - \mu_f(\infty) = \amsmathbb{E}\big[\vartheta ^\prime(t^\ast) \big(f(\xi, t) - g(\xi)\big)\big],\end{equation} where $t^\ast$ lies between $f(\xi, t)$ and $g(\xi).$ Repeating the argument from the proofs of Theorems~\ref{T2an}, \ref{T3an} and \ref{T1gen}, we can find an envelope $\phi$ of $\vartheta^\prime$ satisfying \eqref{cond8}. By \eqref{cond2}, %and assuming that $\bar{\vartheta^\prime}$ is positive as before, 
we get for $t\ge T$ %by the properties of envelope 
\begin{equation*}\label{T2gen_2}|\vartheta ^\prime(t^\ast)| \le C\max\big(\phi(g_1(\xi)), \phi(g_2(\xi))\big)\end{equation*} for some $C>0.$
Taking into account \eqref{cond6}, we derive the following bound
\begin{equation}\label{T2gen_3} \sqrt{k}\left|\mu_f(\xi_{n-k,n}) - \mu_f(\infty)\right| \le C \sqrt{k} |\hat a(\xi_{n-k,n})|\, \amsmathbb{E}\Big[\hat g(\xi) \max\big\{\phi(g_1(\xi)), \phi(g_2(\xi))\big\}\Big]\end{equation} for large $n,$ where the latter expectation is bounded by \eqref{cond8}, which completes the proof in the case $\lambda = 0.$

Next, for $\lambda\neq0,$ %a.e.
continuity of $\vartheta ^\prime$ together with the continuous mapping theorem and \eqref{cond5} imply that
\[\vartheta ^\prime(t^\ast) \stackrel{a.s.}{\to} \vartheta ^\prime(g(\xi)), \quad \frac{f(\xi, t) - g(\xi)}{\hat a(t)} \stackrel{a.s.}{\to} \tilde g(\xi), \quad t\to\infty.\] Therefore by (\ref{T2gen_1}, \ref{T2gen_3}) and the dominated convergence theorem
\[\frac{\mu_f(t) - \mu_f(\infty)}{\hat a(t)} \to \amsmathbb{E}\big[\tilde g(\xi) \vartheta ^\prime(g(\xi))\big] = C, \quad t\to\infty.\]
Finally, the latter convergence and limit relation $\sqrt{k} \hat a(\xi_{n-k,n})\to\lambda$ imply \eqref{T2gen2}. The proof is complete.
\end{proof}
\newpage
\section{Proofs of Section~\ref{sec:exten}} \label{AppC}

\begin{proof}[Proof of Proposition~\ref{P2}]
(i) 
\begin{enumerate}
\item First, assume $\gamma_F, \gamma_G>0.$ Then $H$ is a regularly varying function with index $-1/\gamma_H = -1/\gamma_F - 1/\gamma_G,$ and thus $U_H(y)$ is a regularly varying function with index $\gamma_H.$ By the properties of regularly varying functions (see e.g. Proposition 1.5.7 (ii) in~\cite{bingham}), $1 - \tilde F$ is regularly varying function with index $-\alpha_F = -\gamma_H/\gamma_F,$ and $\alpha_F>0.$ 
\item Next, assume $\gamma_F, \gamma_G<0.$ In this case, $F$ and $G$ have the finite right endpoints $\tau_F$ and $\tau_G,$ and $\tau_H = \min(\tau_F, \tau_G).$ So, $F$ and $G$ can be represented as follows
\[1 - F(\tau_F - x) = x^{-1/\gamma_F} \ell_F(x), \quad 1 - G(\tau_G - x) = x^{-1/\gamma_G} \ell_G(x), \quad x\ge 0,\]
with $\ell_F, \ell_G$ slowly varying at 0, and thus for $x\ge 0$
\[1 - H(\tau_H - x) = \left\{\begin{array}{cc} x^{-1/\gamma_F-1/\gamma_G} \ell_H(x), & \tau_F = \tau_G;\\
x^{-1/\gamma_F} \ell_H(x), & \tau_F < \tau_G; \\
x^{-1/\gamma_G} \ell_H(x), & \tau_F > \tau_G;\end{array}\right.\] with possibly different slowly varying functions $\ell_H$ in each line. It follows that $H$ belongs to the Weibull max-domain of attraction, i.e. $\gamma_H$ is negative. Note that if $\tau_F>\tau_G$ ($\tau_F<\tau_G$) then $\gamma_F$ ($\gamma_G$) can be thought of as equal to $-\infty$ and we can keep the convenient formula $1/\gamma_H = 1/\gamma_F + 1/\gamma_G.$ Based on this relation, we get
\[U_H(y) = \tau_H - y^{\gamma_H} \ell^\sharp_H(y), \]
with $\ell^\sharp_H$ slowly varying at infinity. Assume first $\tau_F \le \tau_G,$ then $\tau_F$ and $\tau_H$ coincide, and we derive by the properties of regularly varying functions
\[1 - \tilde F(y) = 1 - F(\tau_F - y^{\gamma_H} \ell^\sharp_H(y)) = y^{-\gamma_H/\gamma_F} \ell_{\tilde F}(y)\] with some $\ell_{\tilde F}(y)$ slowly varying at infinity. Note that $\alpha_F = \gamma_H/\gamma_F$ is positive in this case. Next assume that $\tau_F>\tau_G,$ then $\tau_H = \tau_F.$ In this case $\tilde F$ is not a proper cdf because it cannot exceed $F(\tau_G)$. Indeed, we get
\[1 - \tilde F(y) = 1 - F(\tau_G - y^{\gamma_H} \ell^\sharp_H(y)) \to 1 - F(\tau_G) > 0, \quad y\to\infty,\] so actually $1 - \tilde F$ is slowly varying at infinity and thus $\alpha_F = 0.$ 
\item Finally, consider the case $\gamma_F = \gamma_G = 0.$ It may be assumed that $\tau_F = \tau_G$ because other cases are degenerate as above. Then by the von Mises representation (see e.g. Theorem 1.2.6 in~\cite{dehaan}), we get for $x<\tau_F$
\begin{eqnarray}\nonumber 1 - H(x) & = & (1-F(x))(1-G(x)) = \exp\left(-\int_{x_0}^x \frac{\dd s}{f(s)} + c_1(x)\right) \exp\left(-\int_{x_0}^x \frac{\dd s}{g(s)} + c_2(x)\right) \\ &=&  \exp\left(-\int_{x_0}^x \Big(\frac{1}{f(s)} + \frac{1}{g(s)}\Big)\dd s + c_1(x) + c_2(x)\right)\label{vonmisesh}\end{eqnarray} with $\lim_{x\uparrow \tau_F} c_i(x) = c_i>0,$ $i=1,2,$ $f,g$ positive, and $\lim_{x\uparrow \tau_F}f^\prime(x) = \lim_{x\uparrow \tau_F}g^\prime(x) = 0$. Since $h = fg/(f+g)$ is such that $\lim_{x\uparrow \tau_F}h^\prime(x) = 0,$ $H$ belongs to the Gumbel max-domain of attraction, i.e. $\gamma_H = 0,$ by Theorem 1.2.6, \cite{dehaan}.

The neat formula $\alpha_F = \gamma_H/\gamma_F$ does no longer hold in this case since $\gamma_F = \gamma_H = 0.$ Nevertheless, $\alpha_F$ can take both positive (if, e.g., $F$ and $G$ are exponential cdfs) and zero (if, e.g., $F$ is exponential cdf and $G$ is Gaussian cdf) values.
\end{enumerate}

(ii) Let us start from the case $\gamma_F>0.$ However, the proof is similar in all three cases. We have 
\[1 - \tilde F(x) = 1 - F(U_H(x)) = x^{-\alpha_F} \ell_{\tilde F}(x)\] with $\ell_{\tilde F}(x)$ slowly varying at infinity. By the definition of $U_F,$ the latter implies
\begin{equation} \label{uhuf} U_H(x) = U_F\big(x^{\alpha_F} (\ell_{\tilde F}(x))^{-1}\big),\end{equation} which holds independently of the value of $\gamma_F.$ Next, $U_F$ is regularly varying with index $\gamma_F,$ thus by the properties of regularly varying functions; see, e.g., Theorem 1.5.7 (ii) in~\cite{bingham}, $U_H$ is regularly varying with index $\gamma_F \alpha_F>0$, which completes the proof in case $\gamma_F>0.$ 

Now assume $\gamma_F<0.$ It is easily follows from Proposition~\ref{P2} (i) %Example~\ref{Ex2}
that if $\alpha_F>0$ then $\tau_F \le \tau_G$ (to prove this, the condition $\gamma_G<0$ is not necessary). Thus, $\tau_H = \tau_F.$
By \eqref{uhuf}, we derive
\[U_H(x) = \tau_F - f\big(x^{\alpha_F} (\ell_{\tilde F}(x))^{-1}\big),\] where $f$ is a regularly varying function with index $\gamma_F$ by the properties of $U_F.$ Applying again Theorem 1.5.7 (ii), \cite{bingham}, we get that $U_H = \tau_F - h(x),$ where $h(x)$ is a regularly varying function with index $\gamma_F \alpha_F<0,$ which completes the proof in case $\gamma_F<0.$

Finally, for the case $\gamma_F = 0,$ note that the positivity of $\alpha_F$ implies $\tau_F = \tau_H$ independently of the value of $\gamma_F.$ Then, the proof can becarried out similarly to the first and second cases if $\tau_F=\infty$ and $\tau_F<\infty,$ respectively.
\end{proof}

\begin{proof}[Proof of Theorem~\ref{T1genekmi}]
The proof relies on the proofs of Theorem~\ref{theo:Rkn_term}, Theorem~\ref{T4}, Theorem~\ref{Tconsistency}, and Theorem~\ref{T1an} under replacement of $\{Z_i\}$ with $\{\xi_i\},$ $\{Z_{n-i,n}/Z_{n-k,n}\}$ with $\{f(\xi_{n-i,n}/\xi_{n-k,n}, \xi_{n-k,n})\},$ and $F,$ $G,$ $H$ etc. with their tilde-versions. We also use $g_1$ and $g_2$ to bound the argument of $\vartheta$ in the analogue of \eqref{P0} and similar situations instead of Potter's bounds. Consequently, we restrict ourselves to indicating the differences.

First, Lemma~\ref{L1} also holds for {  the pairs} $\{(\xi_i^\ast, \delta_i^t)\}.$ Indeed, since $H$ is eventually continuous, for large $n$ we get $\xi_{n-i+1,n} = (1 - H(Z_{n-i+1,n}))^{-1},$ $i=1, \ldots, k.$ Hence we can choose \[\xi_i^\ast = \frac{1-H(t)}{1 - H(tV_i^t)}, \quad i=1, \ldots, k,\] where the joint distribution of $\{\xi^\ast_{k - i+ 1, k}\}_{i=1}^k$ coincides with the joint distribution of $\{\xi_{n-i+1,n}/\xi_{n-k,n}\}_{i=1}^k$ { (and thus does not depend on $t$),} and therefore pairs $\{(\xi_i^\ast, \delta_i^t)\}$ are i.i.d as well.

The second difference relates to the argument used in the proof of Theorem~\ref{theo:Rkn_term} to show boundedness of certain integrals and first applied to prove \eqref{P0}. Indeed, we cannot use that argument as it is, since for that we would need to impose some technical conditions on $f,$ $g_1$ and $g_2$ like monotonicity and continuity. Let us start with the case of zero $\alpha_F$. Consider as an example the version of the relation \eqref{P0} and assume $\alpha_F = 0$. By the monotone density theorem, see e.g. Theorem 1.7.2 in~\cite{bingham}, 
\[\lim_{t\to\infty}\frac{t \tilde F^\prime(t)}{1 - \tilde F(t)} = 0.\] By Potter's bounds, the index of regular variation of $\tilde F^\prime$ cannot be more than $-1$ since the lower Potter bound must be integrable. Next, by Karamata's theorem, see Theorem 1.5.11 ibid., it cannot be less than $-1,$ since otherwise $\alpha_F<0.$ Thus applying Potter's bounds once more, we get for some positive $\delta$ and large $t$ 
\begin{equation}(\tilde F^t)^\prime(w) = \frac{\tilde F^\prime(w t)}{\tilde F^\prime(t)}\cdot\frac{t\tilde F^\prime(t)}{1 - \tilde F(t)} \le w^{-1+\delta}, \quad w\ge 1.\label{b11}\end{equation}
If $\alpha_F$ is positive then the regular variation property of $(\tilde F^t)^\prime$ follows immediately from the monotone density theorem, therefore, applying Potter's bounds, we get \begin{equation}\label{S4potter}|(\tilde F^t)^\prime(w)| \le A w^{- \alpha_F-1+\delta}, \quad w\ge 1,\end{equation} for some $A, \delta>0$ and large $t.$ Note that \[|\tilde \theta^t| = \int_1^\infty\int_1^w |\vartheta^t_f(w)| \frac{1}{(1 - \tilde G^t(v))} \tilde G^t(\dd v) (\tilde F^t)^\prime(w) \dd w.\]
By (\ref{b11}-\ref{S4potter}) and Potter's bound applied to $1 - \tilde G^{t},$ 
 we derive for large $t$ and some $C_1, C_2, \delta^\prime>0$ similarly to \eqref{a9}
\begin{align}\nonumber &|\vartheta^t_f(w) (\tilde F^t)^\prime(w)| \int_1^w  \frac{1}{(1 - \tilde G^t(v))} \tilde G^t(\dd v)  \le  C_1 I(w\in[1,1+\delta^\prime]) \\\nonumber &\quad+\, C_2 \big|\vartheta(f(w,t)) (\tilde F^t)^\prime(w)\big| \log(w) I(w>1+\delta^\prime)\\ &\quad\le  C_1 + A\,C_2 \max\big(\phi(g_1(w)), \phi(g_2(w))\big) \log (w) \,w^{-\alpha_F-1+\delta}, \label{inequality4.7}\end{align} where $\phi$ is an envelope of $\vartheta$ satisfying Condition \ref{varthetaonlycondition}.
The function in the right-hand side of \eqref{inequality4.7} is integrable by Condition~\ref{varthetaonlycondition}, and so the dominated convergence theorem can be applied to prove the convergence of $\tilde\theta^t.$

\end{proof}

\begin{proof}[Proof of Theorem~\ref{T2genekmi}]
By Theorem~\ref{T1genekmi}, it suffices to show
\begin{equation}\label{mufxi}
\sqrt{k}\big(\mu_f(\xi_{n-k,n}) - \mu_f(\infty)\big) \stackrel{P}{\to}\mu.
\end{equation}
We have,
\begin{eqnarray*}
\mu_f(\xi_{n-k,n}) - \mu_f(\infty) & = & \int \vartheta(f(y, \xi_{n-k,n})) \tilde F^{\xi_{n-k,n}}(\dd y) - \int \vartheta(g(y)) \tilde F^\circ(\dd y) \\   
& = & \int \big(\vartheta(f(y, \xi_{n-k,n})) - \vartheta(g(y))\big) \tilde F^{\xi_{n-k,n}}(\dd y) + \int \vartheta_g \dd(\tilde F^{\xi_{n-k,n}} - \tilde F^\circ) \\
&=:& I_1 + I_2.\end{eqnarray*}
The proof of the relation
\[\sqrt{k} I_2 \stackrel{P}{\to}\lambda\int x \vartheta_g^\prime(x) h_\rho(x^{\alpha_F})\tilde F^\circ(\dd x)\]
actually coincides with the proof of Theorem~\ref{T3an} under replacement of $\varphi,$ $\gamma_F,$ $F^\circ,$ etc. with $\vartheta_g,$ $1/\alpha_F,$ $\tilde F^\circ,$ etc., respectively. Here we also use an analogue of \eqref{cond:phi_deriv}, for handling the behavior of $\vartheta_g^\prime$:
\begin{equation*}
     \int_1^\infty |\vartheta_g^\prime(x)|\,x^{-\alpha_F + \varepsilon} \dd x < \infty,
\end{equation*} 
which follows from Condition~\ref{varthetaonlycondition}, integration by parts, and continuity and eventual monotonicity of $\vartheta_g^\prime.$

So, let us consider $I_1.$ As in the proofs above, let us take an envelope $\phi$ of $\vartheta^\prime$ satisfying Condition~\ref{varthetaprime}. Similarly to (\ref{T2gen_1}- \ref{T2gen_3}), we have by \eqref{S4potter} for large $n$ and some $C, C_1>0$,
\begin{eqnarray}\nonumber \sqrt{k} |I_1| &= & \sqrt{k}\int |\vartheta(f(y, \xi_{n-k,n})) - \vartheta(g(y))| \tilde F^{\xi_{n-k,n}}(\dd y) \\ \label{boringlabel} &\le&C \sqrt{k} \int \max\big( \phi(g_1(y)), \phi(g_2(y)) \big) |f(y, \xi_{n-k,n}) - g(y)| \tilde F^{\xi_{n-k,n}}(\dd y) \\
&\le& C_1 \sqrt{k} |\hat a(\xi_{n-k,n})| \cdot \int \hat g(y) \max\big( \phi(g_1(y)), \phi(g_2(y)) \big) y^{-\alpha_F-1+\varepsilon^\prime}\dd y, \nonumber\end{eqnarray}
where the last expression is finite by the relation $\sqrt{k}\hat a(\xi_{n-k,n}) \stackrel{P}{\to}\hat \lambda$ and Condition~\ref{varthetaprime} and vanishes if $\hat \lambda = 0,$ which completes the proof for this case. Moreover, \eqref{boringlabel} implies
\begin{align} \label{secondboringlabel} &\frac{1}{|\hat a(t)|}\Big| \big[\vartheta(f(y, t)) - \vartheta(g(y))\big](\tilde F^t)^\prime(y)\Big|\\
&\le C_1 \hat g(y) \max\big( \phi(g_1(y)), \phi(g_2(y)) \big) y^{-\alpha_F-1+\varepsilon^\prime}.\nonumber
\end{align}

Next, for $\hat \lambda\neq 0,$ we get given $\xi_{n-k,n} = t$
\begin{eqnarray*} I_1 = \int \big(\vartheta(f(y, t)) - \vartheta(g(y))\big) (\tilde F^{t})^\prime(y) \dd y \\
 = \int \vartheta^\prime(t^\ast)\big(f(y, t) - g(y)\big) (\tilde F^{t})^\prime(y) \dd y,\end{eqnarray*} where $t^\ast$ lies between $f(y,t)$ and $g(y).$ The regular variation property of $(\tilde F)^\prime$ and %a.e.
 continuity of $\vartheta^\prime$ together with the continuous mapping theorem imply for every $y\ge1$
 \[\vartheta^\prime(t^\ast) %\stackrel{a.e.}
 {\to}\vartheta^\prime(g(y)), \quad (\tilde F^{t})^\prime(y) \to (\tilde F^\circ)^\prime(y), \quad t\to\infty.\] %where $\stackrel{a.e.}{\to}$ denotes convergence almost everywhere.
 Therefore, by (\ref{boringlabel}-\ref{secondboringlabel}), the dominated convergence theorem and Condition~\ref{second_order_2},
 \[\frac{1}{\hat a(t)} I_1 \to \int \vartheta^\prime(g(y)) \tilde g(y) \tilde F^\circ(\dd y), \quad t\to\infty,\] which together with the limit relation $\sqrt{k} \hat a(\xi_{n-k,n})\stackrel{P}{\to}\hat\lambda$ implies
 \[\sqrt{k} I_1 \stackrel{P}{\to}\hat{\lambda} \int \vartheta^\prime(g(y)) \tilde g(y) \tilde F^\circ(\dd y).\] The proof is complete.
\end{proof}
\newpage
\section{Proofs of Section~\ref{sec:evi_est}}\label{AppD}
\begin{proof}[Proof of Theorem~\ref{Tmomentconsistency}]
First, let us consider the case $\gamma_F>0.$ Then $\alpha_F$ is positive and we derive by Proposition~\ref{P2} (ii) that $G$ and $H$ have regularly varying tails as well, and thus Theorem~\ref{Tconsistency} can be applied. The assumption $\alpha_F>1/2$ (see also the remark below Theorem~\ref{Tmomentconsistency}) implies Condition~\ref{envelope_moment} for the functions $\varphi_\ell(x) = (\log x)^\ell,$ $\ell = 1,2,$ hence the assumptions of Theorem~\ref{Tconsistency} hold for these two functions, obtaining
\[\amsmathbb{M}_{k,n}^{(1)} \xrightarrow{P} \gamma_F, \quad \amsmathbb{M}_{k,n}^{(2)} \xrightarrow{P} 2\gamma^2_F.\] Thus, the result follows from the properties of convergence in probability and continuous mapping theorem.

The proof for the case $\gamma_F\le 0$ is a bit more challenging and is based on the application of Theorem~\ref{T1genekmi}. Note that since $\alpha_F>0$ Proposition~\ref{P2} (ii) implies $\gamma_H\le 0$ and $\tau_H = \tau_F.$ As in Section~\ref{Sectionexamples}, let us consider a candidate function
\begin{equation} \label{candidatef}f(x,t) = \frac{\log U_H(tx) - \log U_H(t)}{a_H(t)/U_H(t)},\end{equation} and therefore by \eqref{formoment} $g(x) = h_{\gamma_H}(x),$ where $h_{\gamma}$ is defined by \eqref{h}. The candidates for $g_1$ and $g_2$ can be found in Section~\ref{Sectionexamples} as well; see~\eqref{g1g2} for the case $\gamma_F<0$ (one should replace $\gamma$ with $\gamma_H$ there), whereas for $\gamma_F=0$ one can use 0 and $\delta x^\delta$ as $g_1$ and $g_2,$ respectively, for some $\delta\in(0,1)$ since $f(x,t)$ is increasing with respect to $x$.

Consider $\vartheta_1(x) = x$ and $\vartheta_2(x) = x^2.$ By the relation $\alpha_F+\alpha_G = 1$ we get $\alpha_G<1/2$ and thus, Condition~\ref{varthetaonlycondition} holds for $\vartheta_1$ and $\vartheta_2,$ but one should select a properly small $\delta$ in case $\gamma_F=0.$ 
Next, the condition on existence of eventually monotone density for $\tilde F$ can be omitted because it is only needed to prove \eqref{inequality4.7} and similar relations. 
Since $f(w,t)$ is monotone by $w$
and $\tilde F$ is regularly varying, the argument similar to the proof of \eqref{P0} can be used to prove the convergence of $\tilde\theta^t$ and  similar expressions. Therefore, applying Theorem~\ref{T1genekmi} with { the} necessary changes indicated above, we get
\begin{equation}\label{speedexample}\frac{\amsmathbb{M}_{k,n}^{(\ell)}}{\big(a_H(U_H^{\leftarrow}(Z_{n-k,n}))/Z_{n-k,n}\big)^\ell} \xrightarrow{P} \int_1^\infty \big(h_{\gamma_H}(y)\big)^\ell \tilde F^\circ(\dd y) = \frac{1}{\alpha_F^\ell}\amsmathbb{E} \big(h_{\gamma_F}(\xi)\big)^\ell, \quad \ell = 1,2,\end{equation} where $\xi$ is a Pareto(1) random variable and 
\[\amsmathbb{E} \big(h_{\gamma_F}(\xi)\big) = \frac{1}{1 - \gamma_F}, \quad \amsmathbb{E} \big(h_{\gamma_F}(\xi)\big)^2 = \frac{2}{(1 - \gamma_F)(1 - 2\gamma_F)}.\]
An application of the properties of convergence in probability and the continuous mapping theorem completes the proof.
\end{proof}

\begin{proof}[Proof of Theorem~\ref{Tmomentan}]
First, let us prove (i). We start with the following key two-dimensional result.
\begin{lemma}\label{L2dim} Assume the conditions of Theorem~\ref{Tmomentan} (i). Then
\[\sqrt{k}\big(\amsmathbb{M}_{k,n}^{(1)} -  \gamma_F,\; \amsmathbb{M}_{k,n}^{(2)} - 2\gamma^2_F\big) \xrightarrow{d} N\left(\lambda b, \Sigma\right),\] where
\[b = \left(\begin{array}{c} \frac{1}{1 - \rho} \\ \frac{2 \gamma_F(2 - \rho)}{(1 - \rho)^2}\end{array}\right), \quad \Sigma = \left(\begin{array}{cc} \frac{\gamma_G \gamma_F^2}{\gamma_G-\gamma_F} & \frac{2\gamma_G \gamma_F^3(2\gamma_G - \gamma_F)}{(\gamma_G-\gamma_F)^2}\\ \frac{2\gamma_G \gamma_F^3(2\gamma_G - \gamma_F)}{(\gamma_G-\gamma_F)^2} & \frac{4\gamma_G \gamma_F^4(5\gamma_G^2 - 4\gamma_G \gamma_F + \gamma_F^2)}{(\gamma_G - \gamma_F)^3}\end{array}\right).\]
\end{lemma}
\begin{proof} To prove the lemma, we require a two-dimensional analogue of Theorem~\ref{T1an}, since then all other results of Section~\ref{sec:ekmi} concern convergence in probability and hence they are extended to the two-dimensional case straightforwardly. Thus, assume two functions $\varphi^{(1)}$ and $\varphi^{(2)}$ satisfying the assumptions of Theorem~\ref{T1an}. Conducting the proof of Theorem~\ref{theo:Rkn_term} and Theorem~\ref{T4} for the pair $\big(S_{k,n}(\varphi^{(1)}), S_{k,n}(\varphi^{(2)})\big)$ together, we get similarly to \eqref{reformulation} the following decomposition:
\[\big(S_{k,n}( \varphi^{(1)}),\, S_{k,n}(\varphi^{(2)})\big)  =  \left(\frac{1}{k} \sum_{j=1}^k W_j^{(1)} + r^{(1)}_{k,n},\; \frac{1}{k} \sum_{j=1}^k W_j^{(2)} + r^{(2)}_{k,n}\right),\] where for every $j\in \{1, \ldots, k\}$ $W_j^{(1)}$ and $W_j^{(2)}$ given $Z_{n-k,n} = t$ are the functions of $V^t_j$ and $\delta_j^t$ defined similarly to \eqref{wast} by replacing $V^t, \delta^t$ with $V^t_j, \delta^t_j$ and $\varphi$ with $\varphi^{(1)}$ and $\varphi^{(2)},$ respectively. So $\big\{(W_j^{(1)}, W_j^{(2)})\big\}_{j=1}^k$ are i.i.d. given $Z_{n-k,n} = t,$ and both $\sqrt{k} r^{(1)}_{k,n},$ $\sqrt{k} r^{(2)}_{k,n}$ %are $\overline{o}_{P}(1)$ and thus
do not affect the asymptotic by the first argument in the proof of Theorem~\ref{Tconsistency}. Following the proof of Theorem~\ref{T1an} in two dimensions (in fact, the only required change is to write $\langle\Vec{s}, \Vec{\varphi}\rangle = s^{(1)}\varphi^{(1)} + s^{(2)}\varphi^{(2)}$ instead of $s \varphi$), we get that the vector
\[\sqrt{k}\Big(S_{k,n}( \varphi^{(1)}) - S_{Z_{n-k,n}}( \varphi^{(1)}),\, S_{k,n}(\varphi^{(2)}) - S_{Z_{n-k,n}}( \varphi^{(2)})\Big)\] converges in distribution to the zero-mean Gaussian random vector with covariance matrix $\Sigma$ given as the covariance matrix of the vector $\big(W^\circ(\varphi^{(1)}), W^\circ(\varphi^{(2)})\big)$ built according to formula  \eqref{wcirc} and using the same $V^\circ$ and $\delta^\circ.$

The result of the lemma now follows by a routine calculation of the asymptotic bias (following from Theorem~\ref{T3an}) and asymptotic variance for $\varphi^{(1)}(x) = \log x$ and $\varphi^{(2)}(x) = \log^2 x.$
\end{proof}

By an application of the above lemma, we obtain (i), appealing to Cram\'er's delta method.\\

Note that it is possible to prove assertion (i) %of the latter theorem 
by imposing the second-order condition \eqref{standardsecondorder} instead of Condition~\ref{secondorder}, as it was done, e.g., in Theorem 3.5.4, \cite{dehaan} in the uncensored case. We avoid such a route to highlight that for $\gamma_F>0$ the proof follows quite simply  from Theorem~\ref{T3an}.\\

The result of (ii) follows immediately from Theorem~\ref{T2genekmi}, the analogue of Lemma~\ref{L2dim} for $\gamma_F=0$ (the proof technique of the two results is virtually identical), and Cram\'er's delta method. The only detail worth clarifying is the calculation of the bias in Theorem~\ref{T2genekmi}. The function $\tilde g$ appearing in the bias term in this theorem is equal to $H_{\gamma_-, \rho^\prime}(\cdot)$ by \eqref{secondorderlog}, where $\gamma_- = \rho^\prime = 0$ by Lemma B.3.16 in~\cite{dehaan}. Thus, $\tilde g(x) = \frac{1}{2} \log^2 x$ by (B.3.8) ibid. The remaining part of the proof is straightforward. \\ 

Now, let us prove (iii). To obtain minimal conditions, we prove \eqref{T2genekmiassertion} once again for particular $f$ and $g.$

Since under the assumptions of Theorem~\ref{Tmomentconsistency}, Theorem~\ref{T1genekmi} holds for $f$ defined by  \eqref{candidatef} and $g = h_{\gamma_H}$, it is enough to show \eqref{mufxi} in our particular case; see the corresponding steps in the proofs of Theorem~\ref{T1genekmi} and Theorem~\ref{T2genekmi}. Denoting for the moment $s = \xi_{n-k,n},$ we have,
\begin{eqnarray}\nonumber\mu_f(s) - \mu_f(\infty) & = & \int_1^\infty \vartheta(f(y, s)) \tilde F^s(\dd y) - \int_1^\infty \vartheta(g(y)) \tilde F^\circ(\dd y) \\\label{I1}
&=&\int_1^\infty \vartheta\left(\frac{\log U_H(sy) - \log U_H(s)}{a_H(s)/U_H(s)}\right)\dd\left(\frac{F(U_H(sy))}{1 - F(U_H(s))}\right)\\ &&\label{I2}- \int_1^\infty \vartheta\left(\frac{y^{\gamma_H}-1}{\gamma_H}\right)\dd(1 - y^{-\alpha_F}),\end{eqnarray}
where $\alpha_F = \gamma_H/\gamma_F$ by Proposition~\ref{P2}. Denote for brevity $q := a_H/U_H$ and recall that $\tau := \tau_F = \tau_H$ by Proposition~\ref{P2} (ii). By Theorem B.2.18 in~\cite{dehaan}, $q$ can be chosen equal to $-\gamma_H(\tau_F - U_H).$ Let us change the argument of the integral \eqref{I1} to
\[x = \frac{\tau - U_H(sy)}{\tau - U_H(s)},\] and the argument of the integral \eqref{I2} to $x = y^{\gamma_H}.$ Notice that these changes are natural since the latter ratio tends to $y^{\gamma_H}$ by the properties of distributions belonging to the Weibull max-domain of attraction. Then, we can rewrite
\begin{eqnarray*}\mu_f(s) - \mu_f(\infty) &=&  - \int_0^1 \vartheta\left(\frac{\log(\tau - xt) - \log(\tau - t)}{-\gamma_H(\log \tau - \log (\tau - t))}\right) \dd \left(\frac{F(\tau - xt)}{1 - F(\tau - t)}\right) \\ && + \int_0^1 \vartheta\Big(\frac{x-1}{\gamma_H}\Big) \dd(1 - x^{-1/\gamma_F}), \end{eqnarray*}
where $t := \tau - U_H(s)$ tends to $0$ as $s\to\infty.$ Denote $F^\circ(1-x) = 1- x^{-1/\gamma_F},$

\begin{align*}
&F^{\ast t}(1 - x) = \frac{F(\tau - xt) - F(\tau- t)}{1 - F(\tau- t)}\\
&\chi(x,t) = \frac{1}{\gamma_H}\left(\frac{\log(1 - x t/\tau)}{\log(1 - t/\tau)} - 1\right), \quad \chi(x,0) = \frac{x-1}{\gamma_H}.
\end{align*}

Clearly, $\chi(x,0) = \lim_{t\downarrow 0} \chi(x,t).$ Using this notation, we get
\begin{eqnarray*}-(\mu_f(s) - \mu_f(\infty)) &=& \int_0^1\vartheta(\chi(x,t))\dd (F^{\ast t}(1-x)) - \int_0^1\vartheta(\chi(x,0))\dd (F^{\circ}(1- x))\\
&=& \left(\int_0^1\vartheta(\chi(x,t))\dd (F^{\ast t}(1-x)) - \int_0^1\vartheta(\chi(x,0))\dd (F^{\ast t}(1-x))\right)\\ &&+ \left(\int_0^1\vartheta(\chi(x,0))\dd [F^{\ast t}(1-x) - F^{\circ}(1-x)]\right) =: I_1 + I_2.\end{eqnarray*}
First, consider the asymptotic behavior of $I_2.$ Similar to \eqref{conjugate}, notice that
\[U_{F^{\ast t}}(x) = 1 - \frac{\tau - U_F(xv)}{\tau - U_F(v)},\] where $t = \tau - U_F(v),$ is the tail quantile function of $F^{\ast t}$ (here we think about $t$ and $v$ as constants) and that it converges to $1 - x^{\gamma_F}$ as $v\to\infty$ by the regular variation property of $U_F.$ Hence $U_{F^{\ast t}}(\xi)$ has a cdf $F^{\ast t},$ where $\xi$ is a a Pareto(1) random variable, and \[v = U_F^{\leftarrow}(\tau - t) = U_F^{\leftarrow}(U_H(\xi_{n-k,n})) = U_F^{\leftarrow}(Z_{n-k,n}) \] as in the proof of Theorem~\ref{T3an}. The remaining part of finding the asymptotic behavior of $I_2$ is a repetition of the steps of that proof (we let the analogue of Condition~\ref{cond:phi_deriv} hold) for $\varphi(x) = \vartheta(\chi(x,0)),$ $u(v,\xi) = 1 - U_{F^{\ast t}}(\xi),$ and replacement of Condition~\ref{secondorder} by the second-order condition \eqref{secondordernegative}. Therefore, recalling that $t = \tau - Z_{n-k,n},$ we derive
\[\sqrt{k} I_2 \xrightarrow{P} \lambda C_2(\gamma_F, \rho),\]
where \[C_2(\gamma_F, \rho) = \amsmathbb{E}\big[\varphi^\prime(\xi^{\gamma_F})\xi^{\gamma_F} h_\rho(\xi)\big] = \frac{1}{\gamma_H}\int_1^\infty  \vartheta^\prime(\chi(x^{\gamma_F},0)) x^{\gamma_F-2} h_\rho(x) \dd x.\] Now, let us prove that 
\begin{equation}\sqrt{k} I_1 \xrightarrow{P} \hat \lambda C_1(\gamma_F)\label{sqrtki1}\end{equation} with \[C_1(\gamma_F) = \int_0^1 \vartheta^\prime\Big(\frac{x-1}{\gamma_H}\Big) \frac{x(x-1)}{2\gamma_H} \dd (F^\circ(1- x)).\]Denote $\vartheta_H(x) = \vartheta((x-1)/\gamma_H)$ for brevity. Keeping the notation $t = \tau - Z_{n-k,n},$ we have
\begin{eqnarray*} I_1 &=& \int_0^1 \left[\vartheta_H\left( \frac{\log(1 - x t/\tau)}{\log(1 - t/\tau)}\right) - \vartheta_H(x)\right] \dd(F^{\ast t}(1- x)).\end{eqnarray*} Let us analyze the difference within the squared brackets in the latter integral. By Taylor expansion and the mean value theorem, we get uniformly in $x$
\begin{eqnarray}\nonumber\vartheta_H\left( \frac{\log(1 - x t/\tau)}{\log(1 - t/\tau)}\right) - \vartheta_H(x) &=& \vartheta_H^\prime(x^\ast) \left(\frac{\log(1 - x t/\tau)}{\log(1 - t/\tau)} - x\right)\\
&=& \vartheta_H^\prime(x^\ast) \frac{x(x-1)}{2\tau} t(1 + o_P(1)), \label{varthetaprimeh}\end{eqnarray} where \[x^\ast \in \left[\frac{\log(1 - x t/\tau)}{\log(1 - t/\tau)}, x\right],\] since the latter ratio of logarithms is less than $x$ for $x\in[0,1]$ and $t\in(0, \tau),$ which holds with probability tending to $1$ as $n\to\infty.$ Note also that for $x\in[0,1]$
\[0 < \frac{1}{1 - xt/\tau} \le C(1+ o_P(1)), \quad n\to\infty,\] for some positive $C.$ On the other hand, using the simple bounds $\frac{y}{1+y} < \log (1+y) < y$ for $y>-1$ and assuming $|\vartheta_H^\prime|$ increases (as are $\vartheta_1(x) = x$ and $\vartheta_2(x) = x^2$), we get
\begin{eqnarray}\nonumber\sqrt{k} |I_1| &\le& \sqrt{k} \int_0^1 |\vartheta_H^\prime(x)| \left|\frac{x}{1 - xt/\tau} - x\right| \dd(F^{\ast t}(1- x))\\ &\le& C \sqrt{k}(1 - Z_{n-k,n}/\tau) \int_0^1 |\vartheta_H^\prime(x)|\, x^2\, x^{-1/\gamma_F-1+\varepsilon} \dd x (1 + o_P(1)),\label{sqrtk|i1|}\end{eqnarray} where the latter inequality holds by the argument used for proving (\ref{P0}). If 
\begin{equation}\label{conditiontrivial} \int_0^1 |\vartheta_H^\prime(x)|\, x^{1 - 1/\gamma_F+\varepsilon} \dd x <\infty,\end{equation} which holds for $\vartheta_1(x) = x$ and $\vartheta_2(x) = x^2,$ then the expression in the right-hand side of \eqref{sqrtk|i1|} is finite by the relation $\sqrt{k}(\tau - U_H(n/k)) \to \hat\lambda \tau$ and the argument similar to \eqref{aratio},
and, moreover, vanishes if $\hat \lambda = 0,$ which completes the proof of 
\eqref{sqrtki1} for this case.

Next, denote \[f_1(x,t) = \left( \frac{\log(1 - x t/\tau)}{\log(1 - t/\tau)} - x\right) \vartheta_H^\prime\left( \frac{\log(1 - x t/\tau)}{\log(1 - t/\tau)}\right),\] \[f_2(x,t) = \left( \frac{\log(1 - x t/\tau)}{\log(1 - t/\tau)}-x\right) \vartheta_H^\prime(x),\] and assume $\vartheta_H^\prime$ is %a.e.
continuous (as is for $\vartheta_1(x) = x$ and $\vartheta_2(x) = x^2$). Notice that
\[\frac{1}{t}|f_1(x,t)|\le \frac{1}{t}|f_2(x,t)| \le Cx^2\vartheta_H^\prime(x)\] for small $t$; in the sequel we use these relations when applying the dominated convergence theorem. Given $\tau - Z_{n-k,n} = t$ (here we think about $t$ as a real number), $\sqrt{k} I_1$ lies between $\sqrt{k}\amsmathbb{E} \big[f_1(1-U_{F^{\ast t}}(\xi), t)\big]$ and $\sqrt{k}\amsmathbb{E} \big[f_2(1-U_{F^{\ast t}}(\xi), t)\big],$ which both converge to $\hat \lambda C_1(\gamma_F)$ as $t\downarrow 0.$ Indeed, by the dominated convergence theorem, (\ref{varthetaprimeh}-\ref{sqrtk|i1|}), convergence $U_{F^{\ast t}}(x) \to U_{F^\circ}(x)$ and %a.e.
continuity of $f_1(x,t)$ in the first argument 
\begin{equation*}\frac{\tau}{t} \amsmathbb{E} \big[f_1(1-U_{F^{\ast t}}(\xi), t)\big] \to \int_0^1 \vartheta^\prime_H(x) \frac{x(x-1)}{2} \dd(F^\circ(1-x)) = C_1(\gamma_F), \quad t\downarrow 0,\label{frac1t}\end{equation*} and the same holds for $f_2.$ 
The above together with $\sqrt{k}(1 - U_H(n/k)/\tau) \to \hat\lambda$ and an argument similar to \eqref{aratio} gives us \eqref{sqrtki1}. \\

Thus, for $\vartheta_\ell(x) = x^\ell,$ $\ell=1,2,$ we get the following result
\begin{equation}\label{negativeg}\sqrt{k}\left(\frac{\amsmathbb{M}_{k,n}^{(\ell)}}{\big(-\gamma_H(\tau - Z_{n-k,n}))^\ell} - \frac{\gamma_F^\ell}{\gamma_H^\ell}\frac{\ell}{(1 - (\ell-1)\gamma_F)(1-\ell \gamma_F)}\right) \xrightarrow{d} N\left(b_\ell, \sigma^2_\ell\right),\end{equation} where

\[b_{\ell} = \lambda \frac{\ell \gamma_F^{\ell-1} (1 + (\ell-1)(1 - \rho - 3\gamma_F))}{\gamma_H^\ell \prod_{i=1}^{\ell}(1 - i\gamma_F)(1 - \rho - i \gamma_F) } - \hat\lambda \frac{\ell^2 \gamma^\ell_F}{2\gamma_H^\ell \prod_{i=1}^{\ell+1} (1-i\gamma_F) },\]
\[\sigma^2_1 = \frac{(\gamma_F + \gamma_G)^2}{\gamma_G(1 - \gamma_F)^2(\gamma_G(1 - 2\gamma_F)-\gamma_F)},\] \[ \sigma_2^2 = \frac{4(\gamma_F+\gamma_G)^4(22\gamma_F^2\gamma_G^2-21\gamma_F\gamma_G^2 + 5\gamma^2_G+9\gamma_F^2\gamma_G-4\gamma_F\gamma_G+\gamma^2_F)}{\gamma_G^3(1-\gamma_F)^2(1-2\gamma_F)^2 \prod_{j=2}^4(\gamma_G-\gamma_F-j\gamma_F\gamma_G)}.\]
Next, repeating the proof of Lemma~\ref{L2dim} with necessary changes, we get the two-dimensional version of \eqref{negativeg} with the limit distribution $N(\vec{b}, \Sigma),$ where $\vec{b} = (b_1, b_2)^T,$ $\Sigma = \|\Sigma_{ij}\|_{i,j=1}^2,$ $\Sigma_{11} = \sigma_1^2,$ $\Sigma_{22} = \sigma_2^2,$ and
\[\Sigma_{12} = \frac{2(\gamma_F+\gamma_G)^3(2\gamma_G - \gamma_F - 4\gamma_F\gamma_G)}{\gamma_G^2 (1-\gamma_F)^2(1-2\gamma_F) \prod_{j=2}^3(\gamma_G-\gamma_F-j\gamma_F\gamma_G)}.\]
An appeal to Cram\'er's delta method now yields (iii) with
\begin{equation}\sigma^2 = a_1^2 \sigma^2_1 + 2a_1a_2\Sigma_{12} + a^2_2\sigma^2_2.\label{sigma}\end{equation} and
\[a_1 = 1 - \frac{2\gamma_G(1 - \gamma_F)^2(1 - 2\gamma_F)}{\gamma_F+\gamma_G}, \quad a_2 = \frac{\gamma_G^2(1 - \gamma_F)^2(1 - 2\gamma_F)^2}{2(\gamma_F+\gamma_G)^2}.\]
\end{proof}

\end{document}